\documentclass[a4paper,10pt,leqno]{amsart}
        \title{The Borel Conjecture for hyperbolic and $\CAT(0)$-groups}
       \author{Arthur Bartels}
       \author{Wolfgang L\"uck}
      \address{Westf\"alische Wilhelms-Universit\"at M\"unster\\
               Mathematisches Institut\\
               Einsteinstr.~62,
               D-48149 M\"unster, Germany}
        \email{bartelsa@math.uni-muenster.de}
      \urladdr{http://www.math.uni-muenster.de/u/bartelsa}
        \email{lueck@math.uni-muenster.de}
      \urladdr{http://www.math.uni-muenster.de/u/lueck}
         \date{March, 2010}

\usepackage{hyperref}

\usepackage{calc}
\usepackage{enumerate,amssymb}
\usepackage[arrow,curve,matrix,tips,2cell]{xy}
  \SelectTips{eu}{10} \UseTips
  \UseAllTwocells

\DeclareMathAlphabet{\matheurm}{U}{eur}{m}{n}

\DeclareMathOperator{\colim}{colim}

\DeclareMathOperator{\flip}{flip}

\DeclareMathOperator{\id}{id}

\DeclareMathOperator{\inv}{inv}
\DeclareMathOperator{\inc}{inc}
\DeclareMathOperator{\Idem}{Idem}

\DeclareMathOperator{\End}{End}

\DeclareMathOperator{\mor}{mor}
\DeclareMathOperator{\pt}{pt}
\DeclareMathOperator{\pr}{pr}
\DeclareMathOperator{\rk}{rk}
\DeclareMathOperator{\sing}{sing}
\DeclareMathOperator{\sign}{sign}
\DeclareMathOperator{\supp}{supp}

\DeclareMathOperator{\tr}{tr}
\DeclareMathOperator{\Wh}{Wh}

\newcommand{\VCyc}{{\mathcal{V}\mathcal{C}\text{yc}}}

  \newcommand{\IN}{\mathbb{N}}

  \newcommand{\IR}{\mathbb{R}}

  \newcommand{\IZ}{\mathbb{Z}}

  \newcommand{\cala}{\mathcal{A}}
  \newcommand{\calb}{\mathcal{B}}
  \newcommand{\calc}{\mathcal{C}}
  
  \newcommand{\cale}{\mathcal{E}}
  \newcommand{\calf}{\mathcal{F}}
   \newcommand{\calfj}{\mathcal{F}\!\mathcal{J}}

  \newcommand{\call}{\mathcal{L}}

  \newcommand{\calo}{\mathcal{O}}

  \newcommand{\calu}{\mathcal{U}}
  \newcommand{\calv}{\mathcal{V}}

  \newcommand{\bfD}{{\mathbf D}}

  \newcommand{\bfG}{{\mathbf G}}

  \newcommand{\bfK}{{\mathbf K}}
  \newcommand{\bfL}{{\mathbf L}}

  \newcommand{\bfP}{{\mathbf P}}
  \newcommand{\bfQ}{{\mathbf Q}}

  \newcommand{\bfT}{{\mathbf T}}

\newcommand{\EGF}[2]{E_{#2}(#1)}

\newcommand{\intgf}[2]{\mbox{$\int_{#1} #2$}}

\newcounter{commentcounter}

\theoremstyle{plain}
\newtheorem{theorem}{Theorem}[section]

\newtheorem{lemma}[theorem]{Lemma}
\newtheorem{corollary}[theorem]{Corollary}
\newtheorem{proposition}[theorem]{Proposition}

\newtheorem{convention}[theorem]{Convention}
\newtheorem*{theorem*}{Theorem}
\newtheorem*{theoremA*}{Theorem A}
\newtheorem*{theoremB*}{Theorem B}

\theoremstyle{definition}
\newtheorem{definition}[theorem]{Definition}

\newtheorem{example}[theorem]{Example}

\newtheorem{remark}[theorem]{Remark}

\newtheorem*{definition*}{Definition}

\theoremstyle{remark}
\newtheorem*{summary*}{Summary}

\makeatletter\let\c@equation=\c@theorem\makeatother

\newenvironment{numberlist}
  {\begin{list}{}%
   {%
    \setlength{\leftmargin}{\labelwidth+\labelsep}%
   }%
  }%
  {\end{list}}

\newcommand{\x}{{\times}}
\newcommand{\ox}{{\otimes}}
\newcommand{\dd}{{\partial}}
\newcommand{\e}{{\varepsilon}}
\newcommand{\CAT}{\operatorname{CAT}}

\newcommand{\bfLi}{{\bfL^{\langle -\infty \rangle}}}
\newcommand{\Li}{L^{\langle -\infty \rangle}}

\begin{document}

\maketitle

\begin{abstract}
We prove the Borel Conjecture for a 
class of groups containing word-hyperbolic groups and
groups acting properly, isometrically  and cocompactly on
a finite dimensional $\CAT(0)$-space.
\end{abstract}

\newlength{\origlabelwidth}
\setlength\origlabelwidth\labelwidth


\typeout{-------------------  Introduction -----------------}
\section*{Introduction}
\label{sec:introduction}


\subsection*{The Borel Conjecture}

A closed manifold $M$ is said to be \emph{topologically rigid} if
every homotopy equivalence to another closed manifold is homotopic to
a homeomorphism.  In particular, if $M$ is topologically rigid, then
every manifold homotopy equivalent to $M$ is homeomorphic to $M$.
The spheres $S^n$ are topologically rigid as predicted by
the \emph{Poincar\'e Conjecture}. 
We will focus on the
\emph{Borel Conjecture} which asserts:
\begin{center}
  {\em Closed aspherical manifolds are topologically rigid.}
\end{center}  
An important result of Farrell-Jones is that this conjecture holds
for manifolds of dimension $\geq 5$ which support a Riemannian metric
of non-positive sectional curvature~\cite{Farrell-Jones(1993c)}.  In
further work Farrell-Jones extended this result to cover compact
complete affine flat manifolds of dimension $\geq
5$~\cite{Farrell-Jones(1998)}.  This is done by considering complete
non-positively curved manifolds that are not necessary compact.  Note
that the universal cover is in these cases always homeomorphic to
Euclidean space.  We will go considerably beyond the world of
Riemannian manifolds of non-positive curvature.  In particular, we
prove the Borel Conjecture for closed aspherical manifolds of
dimension $\geq 5$, whose fundamental group is hyperbolic in the sense
of Gromov~\cite{Bridson-Haefliger(1999)},\cite{Gromov(1987)} or is
non-positively curved in the sense, that it admits a cocompact
isometric proper action on a finite dimensional $\CAT(0)$-space.

\begin{definition*}[The class of groups $\calb$]
\label{def:calb}
  Let $\calb$ be the smallest class of groups
  satisfying the following conditions:
  \begin{enumerate}
  \item \label{def:calb:hyperbolic}
  Hyperbolic groups belong to $\calb$;
  \item \label{def:calb:CAT(0)}
  If $G$ acts properly cocompactly and isometrically
        on a finite-dimensional $\CAT(0)$-space,
        then $G \in \calb$;
  \item \label{def:calb:subgroups}
  The class $\calb$ is closed under taking subgroups;
  \item \label{def:calb:extensions}
  Let $\pi \colon G \to H$ be a group homomorphism.
        If $H \in \calb$ and $\pi^{-1}(V) \in \calb$ for
        all virtually cyclic subgroups $V$ of $H$,
        then $G \in \calb$; 
  \item \label{def:calb:direct_products}
        $\calb$ is closed under finite direct products;
  \item \label{def:calb:free_products}
        $\calb$ is closed under finite  free products;  
  \item \label{def:calb:colimits}
  The class $\calb$ is closed under directed colimits, i.e.,
        if $\{ G_i \; | \; i \in I \}$ is a directed system of groups
        (with not necessarily injective structure maps)
        such that $G _i\in \calb$ for $i \in I$, then 
        $\colim_{i \in I} G_i$ belongs to $\calb$.
  \end{enumerate}
\end{definition*}

We  refer to groups that admit an action on a $\CAT(0)$-space as 
in (ii) as finite dimensional $\CAT(0)$-group.
Notice that the underlying $\CAT(0)$-space
is automatically complete and proper 
(see~\cite[Exercise~8.4~(1) on page~132]{Bridson-Haefliger(1999)}).
If a group acts properly, cocompactly and isometrically
on a $\CAT(0)$-space, then the boundary of this $\CAT(0)$-space
is finite dimensional~\cite[Theorem~12]{Swenson(1999Cutpoint)}.
It seems to be an open question whether the $\CAT(0)$-space itself 
can be arranged to be finite dimensional.

The following is our main theorem.

\begin{theoremA*}
  Let $M$ be a closed  aspherical manifold of dimension $\geq 5$.
  If $\pi_1(M) \in \calb$, then $M$ is topologically rigid.
\end{theoremA*}

We prove this result by establishing the Farrell-Jones
Conjectures in algebraic $K$-theory and $L$-theory for
this class of groups.
(For finite dimensional $\CAT(0)$-groups this is not quite correct, the
result does not cover higher $K$-theory, but suffices 
for the Borel Conjecture and the applications below.)

Provided that Thurston's Geometrization Conjecture is true, every
closed $3$-manifold with torsionfree fundamental group is
topologically rigid and in particular the Borel Conjecture holds in
dimension three
(see~\cite[Theorem~0.7]{Kreck-Lueck(2009nonasph)}). Theorem~A above
remains true in dimension four if one additionally assumes that the
fundamental group is good in the sense of
Freedman~\cite{Freedman(1983)}. In dimension $\le 2$ the Borel
Conjecture is known to be true by classical results. More information
about topologically rigid (not necessarily aspherical) manifolds can be
found in~\cite{Kreck-Lueck(2009nonasph)}. 

A number of further important applications of our results on the
Farrell-Jones Conjecture can be summarized as follows.
The \emph{Novikov Conjecture} and the \emph{Bass Conjecture} hold for
all groups $G$ that belong to $\calb$.
If $G$ is torsion-free and belongs to $\calb$, then the
Whitehead group $\Wh(G)$ of $G$ is trivial, 
$\widetilde{K}_0(RG) = 0$ if $R$ is a principal ideal domain,
and  $K_n( R G) = 0$ for $n \leq -1$ if $R$ a regular ring.
Furthermore the \emph{Kaplansky Conjecture} holds for such $G$.
These and further applications of the Farrell-Jones
Conjectures are discussed in detail in~\cite{Bartels-Lueck-Reich(2008appl)}
and~\cite{Lueck-Reich(2005)}.
We remark that Hu~\cite{Hu(1993)} proved that if $G$ is the fundamental
group of a finite polyhedron with non-positive curvature,
then $\Wh(G) = 0$, $\widetilde{K}_0(\IZ G) = 0$ and
$K_n(\IZ G) = 0$ for $n \leq -1$.


\subsection*{The Farrell-Jones Conjectures}

According to the Farrell-Jones Conjectures~\cite{Farrell-Jones(1993a)} 
the algebraic $K$-theory and the $L$-theory of a group ring $\IZ G$ 
can in a certain sense be computed in terms of the algebraic $K$-theory 
and the $L$-theory of $\IZ V$ where $V$ runs over the family $\VCyc$ of
virtually cyclic subgroups of $G$.
These conjectures are the key to the Borel Conjecture.
See~\cite{Lueck-Reich(2005)} for a survey on the Farrell-Jones Conjectures.
Positive results on these conjectures for groups acting on non-positively
curved Riemannian manifolds are contained in~\cite{Farrell-Jones(1993a)}.
The Farrell-Jones Conjectures have been generalized to include group
rings $RG$ over arbitrary rings~\cite{Bartels-Farrell-Jones-Reich(2004)},
and further to twisted group rings which are best treated using the language
of actions of $G$ on additive 
categories~\cite{Bartels-Lueck(2009coeff)}, \cite{Bartels-Reich(2007coeff)}.
For a group $G$ the Farrell-Jones Conjectures with coefficients
in the additive $G$-category $\cala$ (with involution) 
assert that the assembly maps
\begin{eqnarray}
  \label{eq:VC-assembly-map-K}
  H_m^G (\EGF{G}{\VCyc};\bfK_{\cala} ) & \to & K_m (\intgf{G}{\cala}); \\
  \label{eq:VC-assembly-map-L}
  H_m^G (\EGF{G}{\VCyc};\bfL^{-\infty}_{\cala} ) & \to &
  L^{\langle -\infty \rangle}_m (\intgf{G}{\cala}), 
\end{eqnarray}
are isomorphisms.
The $K$-theoretic Farrell-Jones Conjectures (with coefficients in an
arbitrary additive category) for hyperbolic groups has been
proven by Bartels-L\"uck-Reich
in~\cite{Bartels-Lueck-Reich(2008hyper)}. 
Here we extend this result to the $L$-theoretic Farrell-Jones Conjecture
and (apart from higher $K$-theory) to $\CAT(0)$ groups.

\pagebreak

\begin{theoremB*} 
Let $G \in \calb$.
\begin{enumerate}
\item The $K$-theoretic assembly map~\eqref{eq:VC-assembly-map-K}
      is bijective in degree $m \le 0$ and 
      surjective in degree $m = 1$ for any additive $G$-category $\cala$;
\item The $L$-theoretic Farrell-Jones assembly
      map~\eqref{eq:VC-assembly-map-L}  
      with coefficients in any additive $G$-category $\cala$ with involution
      is an isomorphism. 
\end{enumerate}
\end{theoremB*}

We point out, that the proof of Theorem~B for $\CAT(0)$ groups depends
on Proposition~\ref{prop:CAT(0)-groups-satisfy},
which is proven in~\cite{Bartels-Lueck(2010CAT(0)flow)}.

For virtually abelian groups Quinn~\cite{Quinn(2005)} proved that
\eqref{eq:VC-assembly-map-K} is an isomorphism for all $n$
(more precisely in \cite{Quinn(2005)} only the untwisted case 
is considered: $\cala$ is the category of finitely generated free $R$-modules
for some ring $R$).
In Theorem~\ref{the:axiomatic} below we will give precise conditions under
which our methods establish the assertions of Theorem~B.
In the proof homotopy actions play a prominent role. 
In $K$-theory, these are easier to treat for the groups $K_i$, $i \leq 1$
than for higher $K$-theory and this is the reason for the restrictions in
the $K$-theory statement in Theorem~B.
A proof of the full $K$-theory statement would presumably have to 
take higher order homotopies into account, but we do not pursue this here.

Next we explain the relation between Theorem~B and Theorem~A.

\begin{proposition} 
\label{prop:FJ-implies-Borel}
Let $G$ be a torsion-free group. 
Suppose that the $K$-theoretic assembly map
$$H_m^G(\EGF{G}{\VCyc};\bfK_{\IZ}) \to K_m(\IZ G)$$
is an isomorphism for $m \le 0$ and surjective for $m = 1$
and that the $L$-theoretic assembly map
$$H_m^G(\EGF{G}{\VCyc};\bfL_{\IZ}^{-\infty}) \to L_m^{-\infty}(\IZ G)$$ 
is an isomorphism for all $m \in \IZ$, where we allow a twisting 
by any homomorphism $w \colon G \to \{\pm 1\}$.
Then the following holds:

\begin{enumerate}

\item \label{prop:FJ-implies-Borel:classical-L-assembly}
The assembly map 
\begin{equation}
\label{eq:classical-assembly-L}
H_n(BG; \bfL^s_{\IZ}) \to L^s_n(\IZ G)
\end{equation}
is an isomorphism for all $n$;

\item \label{prop:FJ-implies-Borel:dim_ge_5}
The Borel Conjecture is true in dimension $\ge 5$, i.e., 
if $M$ and $N$ are closed aspherical manifolds of dimensions $\ge 5$
with $\pi_1(M) \cong \pi_1(N) \cong G$, then $M$ and $N$ are homeomorphic and 
any homotopy equivalence $M \to N$  is homotopic to a homeomorphism
(This is also true in dimension $4$ if we assume that
$G$ is good in the sense of Freedman);

\item \label{prop:FJ-implies-Borel:Poincare}
Let $X$ be a finitely dominated Poincar\'e complex of dimension $\ge 6$
with  $\pi_1(X) \cong G$.
Then $X$ is homotopy equivalent to a compact ANR-homology manifold.
\end{enumerate}
\end{proposition}

\begin{proof}%
\ref{prop:FJ-implies-Borel:classical-L-assembly}
Because $G$ is torsion-free and $\IZ$ is regular, 
the above assembly maps are equivalent to the maps
\begin{eqnarray}
  \label{eq:classical-assembly-K}
  H_m(BG;\bfK_{\IZ}) & \to & K_m(\IZ G); \\
  \label{eq:classical-assembly-L-infty}
  H_m(BG;\bfL_{\IZ}^{-\infty}) & \to & L_m^{-\infty}(\IZ G),
\end{eqnarray}
compare~\cite[Proposition~2.2 on page~685]{Lueck-Reich(2005)}.
Because~\eqref{eq:classical-assembly-K} is bijective for $m \le 0$ 
and surjective for $m = 1$, we have $\Wh(G) = 0$, $\widetilde{K}_0(\IZ G) = 0$
and $K_i(\IZ G) = 0$ for $i < 0$,
compare~\cite[Conjecture~1.3 on page~653 and Remark~2.5 on page~679]{Lueck-Reich(2005)}.
This implies that~\eqref{eq:classical-assembly-L-infty} is 
equivalent to~\eqref{eq:classical-assembly-L},
compare~\cite[Proposition~1.5 on page~664]{Lueck-Reich(2005)}.
\\[1mm]%
\ref{prop:FJ-implies-Borel:dim_ge_5}
We have to show that the geometric structure set of a closed aspherical
manifold of dimension $\ge 5$ consists of precisely one element.
This follows from~\ref{prop:FJ-implies-Borel:classical-L-assembly}
and the algebraic surgery exact sequence of 
Ranicki~\cite[Definition~15.19 on page~169]{Ranicki(1992)} which agrees
for an $n$-dimensional manifold for $n \ge 5$ with the Sullivan-Wall geometric
exact surgery sequence (see~\cite[Theorem~18.5 on page~198]{Ranicki(1992)}).
\\[1mm]%
\ref{prop:FJ-implies-Borel:Poincare}
See~\cite[Main Theorem on page~439]{Bryant-Ferry-Mio-Weinberger(1996)}
and~\cite[Remark~25.13 on page~297]{Ranicki(1992)}.
\end{proof}

The assembly maps appearing in the proposition above are special 
cases of the assembly maps~\eqref{eq:VC-assembly-map-K}
and~\eqref{eq:VC-assembly-map-L},
compare~\cite[Corollary~6.17]{Bartels-Reich(2007coeff)}.
In particular, Theorem~A follows from Theorem~B and
the above Proposition~\ref{prop:FJ-implies-Borel}.
In work with Shmuel Weinberger~\cite{Bartels-Lueck-Weinberger(2009)} 
we use Theorem~B to
show that if the boundary of a torsion-free hyperbolic group is
a sphere of dimension $\geq 5$, then this hyperbolic group is
the fundamental group of a closed aspherical manifold, not
just of an ANR-homology manifold.


\subsection*{Some groups from $\calb$}

The class $\calb$ contains in particular directed colimits of 
hyperbolic groups.
The $K$-theory version of the Farrell-Jones Conjecture holds
in all degrees for directed colimits of hyperbolic 
groups~\cite[Theorem~0.8~(i)]{Bartels-Echterhoff-Lueck(2008colim)}.
Thus Theorem~B implies that the Farrell-Jones Conjecture
in $K$- and $L$-theory hold for
directed colimits of hyperbolic groups.
This class of groups contains a number of groups with unusual
properties.
Counterexamples to the Baum-Connes Conjecture with coefficients
are groups with expanders~\cite{Higson-Lafforgue-Skandalis(2002)}.
The only known construction of such groups is as directed colimits
of hyperbolic groups (see~\cite{Arzhantseva-Delzant(2008)}).
Thus the Farrell-Jones Conjecture in $K$- and $L$-theory holds
for the only at present known counter-examples to the Baum-Connes Conjecture
with coefficients.
(We remark that the formulation of the Farrell-Jones Conjecture we
are considering allows for twisted group rings, so this includes the
correct analog of the Baum-Connes Conjecture with coefficients.)
The class of directed colimits of hyperbolic groups 
contains for instance a torsion-free non-cyclic group all 
whose proper subgroups are cyclic constructed by 
Ol'shanskii~\cite{Olshanskii(1979)}.
Further examples are mentioned in~\cite[page~5]{Olshanskii-Osin-Sapir(2009)}
and~\cite[Section~4]{Sapir(2007)}. 
These later examples all lie in the class of lacunary groups.
Lacunary groups can be characterized as certain colimits
of hyperbolic groups.

A Coxeter system $(W,S)$ is a group $W$ together with a fundamental
set $S$ of generators, see for
instance~\cite[Definition~3.3.2]{Davis(2008book)}.  
Associated to the Coxeter system $(W,S)$ is a simplicial complex $\Sigma$ with a
metric~\cite[Chapter~7]{Davis(2008book)} and a proper isometric
$W$-action.  Moussong~\cite{Moussong(1987)} showed that $\Sigma$ is a
$\CAT(0)$-space, see also~\cite[Theorem~12.3.3]{Davis(2008book)}.  In
particular, if $\Sigma$ is finite dimensional and the action is
cocompact, then $W$ is a finite dimensional $\CAT(0)$-group and
belongs to $\calb$.  
This is in particular the case if $S$ is finite.
If $S$ is infinite, then any finite subset $S_0 \subset S$ generates
a Coxeter group $W_0$, see~\cite[Theorem~4.1.6]{Davis(2008book)}.
Then $W_0$ belongs to $\calb$ and so does $W$ as it is the colimit
of the $W_0$.
Therefore Coxeter groups belong to $\calb$.
Davis constructed for every $n \geq 4$ closed aspherical 
manifolds whose universal cover is not
homeomorphic to Euclidean space~\cite[Corollary~15.8]{Davis(1983)}.
In particular, these manifolds do not support metrics of non-positive
sectional curvature.  The fundamental groups of these examples are
finite index subgroups of Coxeter groups $W$. Thus these fundamental groups
lie in $\calb$ and Theorem~A implies that Davis' examples are
topological rigid (if the dimension is at least $5$).

Davis and Januszkiewicz used Gromov's hyperbolization technique 
to construct further exotic aspherical manifolds.
They showed that
for every $n \geq 5$ there are closed aspherical $n$-dimensional manifolds
whose universal cover is a $\CAT(0)$-space whose fundamental group
at infinite is non-trivial~\cite[Theorem~5b.1]{Davis-Januszkiewicz(1991)}.
In particular, these universal covers are not homeomorphic to Euclidean space.
Because these examples are in addition non-positively curved polyhedron,
their fundamental groups are finite-dimensional 
$\CAT(0)$-groups and belong to $\calb$.
There is a variation of this construction that uses  
the strict hyperbolization of Charney-Davis~\cite{Charney-Davis(1995)}   
and produces closed aspherical manifolds whose
universal cover is not homeomorphic to Euclidean space and 
whose fundamental group is hyperbolic.
All these examples are topologically rigid by Theorem~A.

Limit groups as they appear for instance in~\cite{Zela(2001)} 
have been in the focus of geometric group theory for the last years. 
Expositions about limit groups are for 
instance~\cite{Champetier+Guirardel(2005)} and~\cite{Paulin(2004)}. 
Alibegovi\'c-Bestvina have shown that limit groups are 
$\CAT(0)$-groups~\cite{Alibegovic+Bestvina(2006)}.
A straight forward analysis of their argument shows, that
limit groups are finite dimensional $\CAT(0)$-groups and belong
therefore to our class $\calb$.

If a locally compact group $L$ acts properly cocompactly and isometrically on
a finite dimensional $\CAT(0)$-space, then the same is true for any
discrete cocompact subgroup of $L$.
Such subgroups belong therefore to $\calb$.
For example, let $\bfG$ be a reductive algebraic group defined over
a global field $k$ whose $k$-rank is $0$.
Let $S$ be a finite set of places of $k$ that contains the infinite 
places of $k$.
The group $\bfG_S := \prod_{v \in S} \bfG(k_v)$ admits an isometric proper
cocompact action on a finite dimensional $\CAT(0)$-space,
see for example~\cite[page~40]{Ji(2007SarithI)}.
Because $S$-arithmetic subgroups of $\bfG(k)$ can be realized 
(by the diagonal embedding) as discrete cocompact subgroups
of $\bfG_S$ (see for example~\cite{Ji(2007SarithI)}), 
these $S$-arithmetic groups belong to $\calb$.

Finitely generated virtually abelian groups are finite dimensional 
$\CAT(0)$-groups and belong to $\calb$.
A simple induction shows that this implies that all
virtually nilpotent groups belong to $\calb$,
compare the proof of~\cite[Lemma~1.13]{Bartels-Lueck-Reich(2008appl)}.


\subsection*{Outline of the proof}

In Section~\ref{sec:Axiomatic_formulation} we formulate geometric
conditions under which we can prove the Farrell-Jones Conjectures.
These conditions are satisfied for hyperbolic groups and
finite dimensional $\CAT(0)$-groups (see Section~\ref{sec:proof-of-thm-B}) and
are similar to the conditions under which the $K$-theoretic Farrell-Jones 
Conjecture has been proven in~\cite{Bartels-Lueck-Reich(2008hyper)}.
Very roughly, these conditions assert the existence of a compact space $X$
with a homotopy $G$-action and the existence of a ``long thin'' 
$G$-equivariant cover of $G \x X$.
New is the use of homotopy actions here; this is crucial for the application
to finite dimensional $\CAT(0)$-groups. It suffices to have homotopy
actions at hand since the transfer maps require only homotopy chain actions.

The general strategy of the proof is similar to the one employed 
in~\cite{Bartels-Lueck-Reich(2008hyper)}.
Controlled algebra is used to set up an obstruction category whose
$K$- respectively $L$-theory gives the homotopy fiber of the
assembly map in question, see Theorem~\ref{thm:obstruction-category}.
We will mostly study $K_1$ and $L_0$ of these categories.
In $K$-theory we represent elements by automorphisms 
or more generally by self-chain homotopy equivalences.
In $L$-theory we represent elements by quadratic forms
or more generally by $0$-dimensional ultra-quadratic Poincar\'e complexes,
compare Subsection~\ref{subsec:controlled-algebraic-Poincare-cx}.
For this outline it will be convenient to call these representatives cycles.
In all cases these cycles come with a notion of size.
More precisely, the obstruction category depends on a free $G$-space $Z$
(in the simplest case this space is $G$, 
but it is important to keep this space variable)  
and associated to any cycle is a subset (its support) of $Z \x Z$. 
If $Z$ is a metric space, then a cycle is said to be $\alpha$-controlled
over $Z$ for some number $\alpha > 0$ if $d_Z(x,y) \leq \alpha$ for
all $(x,y)$ in the support of the cycle.
The Stability Theorem~\ref{thm:stability-element} 
for the obstruction category asserts (for a class of metric space),
that there is $\e > 0$ such that the $K$-theory respectively $L$-theory
class of every $\e$-controlled cycle is trivial.
The strategy of the proof is then to prove that
the $K$-theory respectively $L$-theory of the obstruction category is trivial 
by showing that every cycle is equivalent to an $\e$-controlled cycle.

This is achieved in two steps.
Firstly, a transfer replacing $G$ by $G \x X$ for a suitable compact space
$X$ is used.
Secondly, the ``long thin'' cover of $G \x X$ is used to construct
a contracting map from $G \x X$ to a $\VCyc$-$CW$-complex,
see Proposition~\ref{prop:contracting-maps}.
More precisely, this map is contracting with respect to the $G$-coordinate,
but expanding with respect to the $X$-coordinate. 
Thus it is crucial that the output of the transfer is a cycle that is
$\e$-controlled over $X$ for very small $\e$.    
To a significant extend, the argument in the $L$-theory case and the
$K$-theory case are very similar.
For example, the formalism of controlled algebra works for
$L$-theory in the same way as for $K$-theory.
This is because both functors have very similar properties,
compare Theorem~\ref{thm:swindle_plus_karoubi-for-K_plus_L}.
However, the $L$-theory transfer is quite different from
the $K$-theory transfer and requires new ideas.


\subsection*{$L$-theory transfer}

The transfer is used to replace a cycle $a$ in the $K$- respectively $L$-theory 
of the obstruction category over $G$ by a cycle $\tr(a)$ over 
$G \x X$, such that
$\tr(a)$ is $\e$-controlled (for very small $\e$) if control is measured
over $X$ (using the canonical projection $G \x X \to X$).
In $K$-theory the transfer is essentially obtained by taking a tensor product
with the singular chain complex of $X$.
More precisely, we use a chain complex $P$ chain homotopy equivalent to
the singular complex, such that $P$ is in addition $\e$-controlled over $X$,
compare Proposition~\ref{prop:projective-chain-complex}.
(Roughly, this is the simplicial chain complex of a sufficiently fine 
triangulation of $X$.)
The homotopy action on $X$ induces a corresponding action on $P$.
This action is important as it is used to twist the tensor product. 
The homology of $P$ agrees with the homology of a point (because
$X$ is contractible).
This is important as it controls the effect of the transfer in $K$-theory,
i.e., $\tr(a)$ projects to $a$ under the map induced by $G \x X \to G$.
The datum needed for transfers in $L$-theory is a chain complex together with
a symmetric form, i.e., a symmetric Poincar\'e complex.
It is not hard, because $P$ has the homology of a point, to equip
$P$ with a symmetric form.
However, such a symmetric form on $P$ will not be $\e$-controlled over
$X$ and is therefore not sufficient for the purpose of producing a
cycle $\tr(a)$ which is $\e$-controlled over $X$.

In the case treated by Farrell-Jones, where $G$ is the fundamental group of a 
non-positively curved manifold $M$, this problem is solved roughly as follows.
In this situation the sphere bundle $SM \to M$ is considered.
The fiber of this bundle is a manifold and Poincar\'e duality yields
an $\e$-controlled symmetric form on the simplicial chain complex of
a sufficiently fine triangulation of the fiber.
However, the signature of the fiber governs the effect of the transfer in 
$L$-theory and since the signature of the sphere is trivial the transfer
is the zero map in $L$-theory in this case.
This problem is overcome by considering the quotient of the fiber-wise
product $SM \x_M SM$ by the involution that flips the two factors.
The fiber of this bundle is a $\IZ[1/2]$-homology manifold whose
signature is $1$ (if the dimension of $M$ is odd).
In order to get a transfer over $\IZ$ rather than $\IZ[1/2]$ the singularities
of this fiber have to be studied and this leads to very technical arguments but
can be done, see~\cite[Section~4]{Farrell-Jones(1989)}.
The main problem here is that the normal bundle of the fixed point set 
of the flip (i.e., the diagonal sphere in the product) is in general not 
trivial.

For the groups considered here the space $X$ will in general not be a manifold 
and we are forced to use a different approach to the $L$-theory transfer.
Given the chain complex $P$, we use what we call the 
\emph{multiplicative hyperbolic Poincar\'e chain complex} on $P$. 
As a chain complex this is $D := P^{-*} \ox P$ and this chain complex
carries a natural symmetric form given by the canonical isomorphism
$(P^{-*} \ox P)^{-*} \cong P \ox P^{-*}$ followed by the flip
$P \ox P^{-*} \cong P^{-*} \ox P$.
The multiplicative hyperbolic Poincar\'e chain complex can naturally be 
considered as a complex over $X \x X$.
Because of the appearance of the flip in the construction
it is not $\e$-controlled over $X \x X$.
But this flip is the only problem and the multiplicative
hyperbolic Poincar\'e chain complex becomes $\e$-controlled over the 
quotient $P_2(X)$ of $X \x X$ by the flip $(x,y) \mapsto (y,x)$.
This construction appears in the proof of
Proposition~\ref{prop:controlled-symmetric-chain-cx}. 
In the Appendix~\ref{sec:classical-transfer_and_mult-hyperbolic} 
to this paper, we review  classical
(i.e., uncontrolled) transfers in $K$-theory (for the Whitehead group)
and $L$-theory. 
The reader is encouraged to refer to the appendix for motivation
while reading the Sections~\ref{sec:transfer-up-to-homotopy},~%
\ref{sec:The_transfer_in_K-theory} and~\ref{sec:The_transfer_in_L-theory}.
The appendix also contains a discussion of the multiplicative
hyperbolic Poincar\'e chain complex in a purely algebraic context.


\subsection*{Acknowledgments}

We thank Andrew Ranicki and Shmuel Weinberger for fruitful discussions
and comments. We are indebted to Tom Farrell and Lowell Jones for their
wonderful conjecture and work surrounding it.
The work was financially supported by the
Sonderforschungsbereich 478 \--- Geometrische Strukturen in der
Mathematik \--- and the Max-Planck-Forschungspreis and the Leibniz-Preis 
of the second author. 
The second author wishes also to thank the Max-Planck Institute for 
Mathematics in Bonn for its hospitality during his stay in from 
October until December 2007 when parts of this papers were written.\\

\noindent
The paper is organized as follows:
\\[2mm]
\begin{tabular}{ll}%
\ref{sec:Axiomatic_formulation}. & Axiomatic formulation
\\%
\ref{sec:proof-of-thm-B}. & Proof of Theorem~B modulo 
                                 Theorem~\ref{the:axiomatic}
\\%
\ref{sec:S-long_covers_yield_contracting_maps}. & $S$-long covers 
                                    yield contracting maps
\\%
\ref{sec:controlled-algebra_plus_L}. & Controlled algebra 
                                 with a view towards $L$-theory
\\%
\ref{sec:stability_plus_assembly}. & Stability and the assembly map
\\%
\ref{sec:transfer-up-to-homotopy}. & Transfer up to homotopy
\\%
\ref{sec:The_transfer_in_K-theory}. & The transfer in $K$-theory
\\%
\ref{sec:Proof_of_Prop}. & Proof of 
                      Proposition~\ref{prop:projective-chain-complex}
\\%
\ref{sec:The_space_P_2(X)}. & The space $P_2({X})$
\\%
\ref{sec:The_transfer_in_L-theory}. & The transfer in $L$-theory
\\%
\ref{sec:Proof_of_axiomatic_Theorem}. & Proof of Theorem~\ref{the:axiomatic}
\\%
\ref{sec:classical-transfer_and_mult-hyperbolic}. 
& Classical transfers and the multiplicative hyperbolic form
\\ & References
\end{tabular}


\typeout{------------------ Axiomatic formulation  ---------------------}

\section{Axiomatic formulation}
\label{sec:Axiomatic_formulation}

\begin{summary*}
  In this section we describe conditions under which
  our arguments prove the Farrell-Jones Conjectures.
  If these conditions are satisfied for a group $G$ with respect to
  a family $\calf$ of subgroups, then $G$ 
  is said to be \emph{transfer reducible} over $\calf$,
  see Definition~\ref{def:transfer-reducible}.
  Very roughly this means that
  there is a space $X$ satisfying suitable finiteness conditions
  such that $G \x X$ admits $G$-equivariant covers of uniformly
  bounded dimension that are very long in the $G$-direction, up to a
  twist described by a homotopy action on $G$, see in particular
  Definition~\ref{def:S-action_plus_long-covers}~\ref{def:S-action_plus_long-covers:long}.   
  Theorem~\ref{the:axiomatic} is the most general statement about the
  Farrell-Jones Conjectures in this paper.
  It is conceivable that it applies to further interesting groups 
  that do not belong to $\calb$.
\end{summary*}  

A \emph{family  $\calf$ of subgroups of the group $G$} is a set of subgroups
of $G$ closed under conjugation and taking subgroups.

\begin{theorem}[Axiomatic Formulation]
  \label{the:axiomatic}
  Let $\calf$ be a family of subgroups of the group $G$.

  If $G$ is transfer reducible over $\calf$
  (see Definition~\ref{def:transfer-reducible})
  then the following holds:
  \begin{enumerate}
  \item \label{thm:axiomatic:K}
        Let $\cala$ be an additive $G$-category, i.e., an additive
        category with right $G$-action by functors of additive categories.
        Then the assembly map
        \begin{equation}
        \label{eq:assembly-map-calf-K}
        H_m^G (\EGF{G}{\calf};\bfK_{\cala} ) \to K_m (\intgf{G}{\cala})
        \end{equation}
        is an isomorphism for $m < 1$ and surjective for
        $m = 1$;
  \item \label{thm:axiomatic:L}
        Let $\cala$ be an additive $G$-category with involution (in the
        sense of~\cite[Definition~4.22]{Bartels-Lueck(2009coeff)}).
        Then the  assembly map
        \begin{equation}
        \label{eq:assembly-map-calf-2-L}
        H_m^G (\EGF{G}{\calf_2};\bfL^{-\infty}_{\cala} ) \to
        L^{\langle -\infty \rangle}_m (\intgf{G}{\cala})
        \end{equation}
        is an isomorphism for all $m \in \IZ$.
        Here $\calf_2$ is the family of all subgroups
        $V \subseteq G$ for which there is $F \subseteq V$ such that
        $F \in \calf$ and $[V:F] \leq 2$.
  \end{enumerate}
\end{theorem}

The assembly maps appearing above have been introduced
in~\cite{Bartels-Lueck(2009coeff)} and~\cite{Bartels-Reich(2007coeff)}
and the two slightly different approaches are identified
in~\cite[Remark~10.8]{Bartels-Lueck(2009coeff)}. 
If $\calf$ is the family $\VCyc$ of virtually cyclic groups,
then these maps are the assembly maps~\eqref{eq:VC-assembly-map-K}
and~\eqref{eq:VC-assembly-map-L} from the introduction.
(Of course $\VCyc_2 = \VCyc$.)

In the following definition we weaken the notion of an action of a 
group $G$ on a space $X$ to a homotopy action that is only defined for a 
finite subset $S$ of $G$. 
Restriction of a $G$-action to such a finite subset $S$ yields 
a homotopy $S$-action. 
Other examples arise, if we conjugate an honest action by a homotopy
equivalence and restrict then to a finite subset $S$.

\begin{definition}[Homotopy $S$-action]
\label{def:S-action_plus_long-covers}
Let $S$ be a finite subset of a group $G$.
Assume that $S$ contains the trivial element $e \in G$.
Let $X$ be a space.
\begin{enumerate}
\item \label{def:S-action_plus_long-covers:action}
      A homotopy $S$-action on $X$ consists of
      continuous maps $\varphi_g \colon X \to X$
      for $g \in S$ and homotopies
      $H_{g,h} \colon X \times [0,1] \to X$
      for $g,h \in S$ with $gh \in S$ such that
      $H_{g,h}(-,0) = \varphi_g \circ \varphi_h$
      and $H_{g,h}(-,1) = \varphi_{gh}$ holds for $g,h \in S$ with $gh \in S$.
      Moreover, we require that $H_{e,e}(-,t) = \varphi_e = \id_X$
      for all $t \in [0,1]$;
\item \label{def:S-action_plus_long-covers:F}
      Let $(\varphi,H)$ be a homotopy $S$-action on $X$.
      For $g \in S$ let $F_g(\varphi,H)$ be the set of all
      maps $X \to X$ of the form $x \mapsto H_{r,s}(x,t)$
      where $t \in [0,1]$ and $r,s \in S$ with $rs = g$;
\item \label{def:S-action_plus_long-covers:S(g,x)}
      Let $(\varphi,H)$ be a homotopy $S$-action on $X$.
      For $(g,x) \in G \times X$ and $n \in \IN$,
      let $S^n_{\varphi,H}(g,x)$ be the subset
      of $G \times X$ consisting of all $(h,y)$
      with the following property:
      There are
      $x_0,\dots,x_n \in X$,
      $a_1,b_1,\dots,a_n,b_n \in S$,
      $f_1,\widetilde{f}_1,\dots,f_n,\widetilde{f_n} \colon X \to X$,
      such that
      $x_0 = x$, $x_n = y$, $f_i \in F_{a_i}(\varphi,H)$,
      $\widetilde{f}_i \in F_{b_i}(\varphi,H)$,
      $f_i(x_{i-1}) = \widetilde{f}_i(x_i)$ and
      $h = g a_1^{-1} b_1 \dots a_n^{-1} b_n$;
\item \label{def:S-action_plus_long-covers:long}
      Let $(\varphi,H)$ be a homotopy $S$-action on $X$
      and $\calu$ be an open cover of $G \times X$.
      We say that $\calu$ is $S$-long with respect to $(\varphi,H)$
      if for every $(g,x) \in G \times X$ there is $U \in \calu$ containing
      $S^{|S|}_{\varphi,H}(g,x)$ where $|S|$ is the cardinality of $S$.
\end{enumerate}
\end{definition}

If the homotopy action is the restriction of a $G$-action to $S$
and $S$ is symmetric with respect to $s \mapsto s^{-1}$,
then $\varphi_g(x) = x$, $H_{g,h}(x,t) = {gh}x$ for all $t$ and
$S^n_{\varphi,H}(g,x) = 
 \{ (ga^{-1},ax) \; | \; a = s_1 \dots s_{2|S|}, s_i \in S \}$.
We will be able to restrict to a finite subset $S$ of $G$,
because our cycles for elements in the algebraic $K$-theory 
or $L$-theory of the obstruction category will involve only
a finite number of group elements.
For example, if
we are looking at an element in the $K$-theory of $RG$ given 
by an invertible matrix $A$ over $RG$,  
then the set $S$ consist of those group
elements $g$ which can be written as a product $g_1g_2$ for which the
coefficient of some entry in $A$ or $A^{-1}$ for $g_1$ and the
coefficient of some entry in $A$ or $A^{-1}$ for $g_2 $ are
non-trivial.

\begin{definition}[$N$-dominated space]
\label{def:N-dominated_space}
Let $X$ be a metric space and $N \in \IN$.
We say that $X$ is controlled $N$-dominated
if for every $\e > 0$ there is a finite $CW$-complex $K$ of dimension
at most $N$, maps $i \colon X \to K$, $p \colon K \to X$
and a homotopy $H \colon X \times [0,1] \to X$ between $p \circ i$ and
$\id_X$ such that for every $x \in X$
the diameter of $\{ H(x,t) \mid t \in [0,1] \}$  is
at most $\e$.
\end{definition}

\begin{remark}
For a hyperbolic group we will use the compactification 
of the Rips complex for $X$.
This space is controlled $N$-dominated by finite subcomplexes
of the Rips complex.
The homotopy $S$-action on $X$ arises as the restriction to $S$ of the
action of the hyperbolic group on $X$.

For a group $G$ that acts properly cocompactly and isometrically on
a finite dimensional $\CAT(0)$-space $Z$, we will use a ball in $Z$ 
of sufficiently large radius for $X$.
Projection along geodesics provides a homotopy inverse to the inclusion
$X \to Z$.
The homotopy $S$-action on $X$ is obtained by first restricting the $G$-action 
on $Z$ to $S$ and then conjugate it to $X$ using this homotopy equivalence.
The controlled $N$-domination arises in this situation because 
$X$ is a Euclidean neighborhood retract. 
\end{remark}

We recall the following definition
from~\cite[Definition~1.3]{Bartels-Lueck-Reich(2008hyper)}.

\begin{definition}[Open $\calf$-cover]
\label{def:F-cover}
Let $Y$ be a $G$-space. Let $\calf$ be a family of  subgroups of $G$.
A subset $U \subseteq Y$ is called
an \emph{$\calf$-subset} if 
\begin{enumerate}
\item For $g  \in G$ and $U \in  \calu$ we have
      $g(U) = U$ or $U \cap g(U) = \emptyset$, where
      $g(U):=  \{gx \mid x \in U \}$;
\item The subgroup
      $G_U := \{ g \in G \; | \; g(U) = U \}$ lies in $\calf$.
\end{enumerate}
An \emph{open $\calf$-cover} of $Y$ is a collection $\calu$ of open
$\calf$-subsets of $Y$  such that the following conditions are satisfied:
\begin{enumerate}
\item $Y = \bigcup_{U \in \calu} U$;
\item For $g \in  G$, $U \in \calu$  the set
      $g(U) $ belongs to $\calu$.
\end{enumerate}
\end{definition}

\begin{definition}[Transfer reducible]
  \label{def:transfer-reducible}
  Let $G$ be a group and $\calf$ be a family of subgroups.
  We will say that $G$ is \emph{transfer reducible} over 
  $\calf$ if there is a number $N$ with the following property:

  For every finite subset $S$ of $G$ there are
  \begin{itemize}
  \item a contractible compact controlled $N$-dominated
        metric space $X$;
  \item a homotopy $S$-action $(\varphi,H)$ on $X$;
  \item a cover $\calu$ of $G \times X$ by open sets,
  \end{itemize}
  such that the following holds for the $G$-action on $G \times X$ given by
  $g\cdot (h,x) = (gh,x)$:
  \begin{enumerate}
  \item $\dim \calu \leq N$;
  \item $\calu$ is $S$-long with respect to $(\varphi,H)$;
  \item $\calu$ is an open $\calf$-covering.
  \end{enumerate}
\end{definition}

\begin{remark} The role of the space $X$ appearing in
  Definition~\ref{def:transfer-reducible} is to yield enough
  space to be able to find the desired covering $\calu$. 
  On the first glance one might take $X = \{\pt\}$. 
  But this is not a good choice
  by the following observation.

  Suppose that the homotopy action actually comes from an honest $G$-action on 
  $X$. Then for every $x \in X$ and every finitely generated subgroup 
  $H \subseteq G_x$ we have $H \in \calf$ by the following argument. 
  Given a finite subset $S$ of $G_x$ with $e \in S$,  
  we can find $U \in \calu$
  with $\{(s,x) \mid s \in S\} \subseteq U$ since $\calu$ is $S$-long. 
  Then $\{(e,x)\} \in s \cdot U \cap U$ for $s \in S$. 
  This implies $S \subseteq G_U$. 
  Hence the subgroup of $G$ generated by $S$ belongs to $\calf$ 
  since $G_U$ belongs to $\calf$ by assumption and a family is by 
  definition closed under taking subgroups. 

  Of course we would like to arrange that we can  
  choose $\calf$ to be the family $\VCyc$. 
  But this is only possible if all isotropy groups of $X$ 
  are virtually cyclic.

  The main difficulty in finding the desired covering appearing in
  Definition~\ref{def:transfer-reducible} is that the
  cardinality of~$S$ can be arbitrarily large in comparison to the
  fixed number~$N$.
\end{remark}


\typeout{---------  Proof of Theorem B  ----------------}

\section{Proof of Theorem~B modulo Theorem~\ref{the:axiomatic}}
\label{sec:proof-of-thm-B}

\begin{summary*}
  In this section we show that Theorem~\ref{the:axiomatic}
  implies Theorem~B.
  To this end we describe in Lemma~\ref{lem:calfj_closed_under}
  inheritance properties
  of the Farrell-Jones Conjectures and show
  that hyperbolic groups are transfer reducible
  over the family of virtually cyclic subgroups.
  The latter depends ultimately on work of
  Mineyev and Yu~\cite{Mineyev(2005), Mineyev-Yu(2002)}.
  We show in~\cite{Bartels-Lueck(2010CAT(0)flow)} that finite dimensional $\CAT(0)$-groups
  are also transfer reducible
  over the family of virtually cyclic subgroups.
  \end{summary*}

\begin{proposition}
  \label{prop:hyperbolic-satisfy}
  Every hyperbolic is transfer reducible 
  over the family $\VCyc$ of virtually cyclic subgroups.
\end{proposition}

This will essentially follow
from~\cite{Bartels-Lueck-Reich(2008cover)}
and~\cite{Bestvina-Mess(1991)}~, see
also~\cite[Lemma~2.1]{Bartels-Lueck-Reich(2008hyper)}. 
However, the set-up
in~\cite{Bartels-Lueck-Reich(2008cover)} is a little different, there $X$
is a $G$-space and the diagonal action $g \cdot (h,x) = (gh,gx)$ is
considered on $G \times X$, where in this paper the $G$-action $g\cdot
(h,x) =(gh,x)$ is used. The reason for this change is that we do not
have a $G$-action on $X$ available in the more general setup of this
paper, there is only a homotopy $G$-action.

\begin{proof}[Proof of Proposition~\ref{prop:hyperbolic-satisfy}]
Let $d_G$ be a $\delta$-hyperbolic left-invariant word-metric on
the hyperbolic group $G$.
Let $P_d(G)$ be the associated \emph{Rips complex}
for $d > 4\delta + 6$.
It is a finite-dimensional contractible locally finite simplicial complex.
This space can be compactified to $X := P_d(G) \cup \dd G$,
where $\dd G$ is the Gromov boundary of $G$
(see~\cite[III.H.3]{Bridson-Haefliger(1999)}, \cite{Gromov(1987)}).
Then $X$ is metrizable 
(see~\cite[III.H.3.18 (4) on page 433]{Bridson-Haefliger(1999)}).
There is a simplicial action of $G$ on $P_d(G)$ which is proper and
cocompact, and this action extends to $X$.
According to~\cite[Theorem~1.2]{Bestvina-Mess(1991)} the subspace
$\partial P_d(G) \subseteq X$ satisfies the Z-set condition.
This implies the (weaker)~\cite[Assumption~1.2]{Bartels-Lueck-Reich(2008hyper)}
which is a consequence of part (2) of the
characterization of $Z$-sets before Theorem~1.2 in~\cite{Bestvina-Mess(1991)}.
Thus there is a homotopy $H \colon X \x [0,1] \to X$, such that
$H_0 = \id_X$ and $H_t(X) \subset P_d(G)$ for all $t > 0$.
The compactness of $X$ implies that for $t > 0$, $H_t(X)$ is contained
in a finite subcomplex of $P_d(G)$.
Therefore $X$ is
controlled $N'$-dominated, where $N'$ is the dimension of $P_d(G)$.

The main result of~\cite{Bartels-Lueck-Reich(2008cover)}
asserts that there is a number $N$ such that for every $\alpha > 0$
there exists an open cover $\calu_\alpha$ of $G \times X$ equipped
 with the diagonal $G$-action
such that
\begin{itemize}
\item $\dim \calu_\alpha \leq N$;
\item For every $(g,x) \in G \times X$ there is $U \in \calu_\alpha$
      such that
      \begin{equation*}
        g^\alpha \times \{ x \} \subseteq U.
      \end{equation*}
      (Here $g^\alpha$ denotes the open $\alpha$-ball in $G$ around $g$.)
\item $\calu_\alpha$ is a $\VCyc$-cover with respect to the diagonal
      $G$-action $g \cdot (h,x) = (gh,gx)$.
\end{itemize}
The map $(g,x) \mapsto (g,g^{-1}x)$ is a $G$-equivariant
homeomorphism from $G \times X$ equipped with diagonal action to $G
\times X$ equipped with the action $g \cdot (h,x) = (gh,x)$.
Pushing the cover $\calu_\alpha$ forward with this homeomorphism
we obtain a new cover $\calv_\alpha$ of $G \times X$ such that
\begin{itemize}
\item $\dim \calv_\alpha \leq N$;
\item For every $(g,y) \in G \times X$ there is $V \in \calv_\alpha$
      such that
      \begin{equation*}
        \{ (gh,h^{-1}y) \mid h \in e^{\alpha} \} \subseteq V.
      \end{equation*}
      (We denote by $e$ the unit element of $G$.)
\item $\calv_\alpha$ is a $\VCyc$-cover with respect to the left
      $G$-action $g \cdot (h,x) = (gh,x)$.
\end{itemize}

Consider  a finite subset $S$ of $G$ containing $e$. 
Put $n = |S|$.
Pick $\alpha >0$ such that
$$\{l \in G\mid
l = a_1^{-1}b_1\dots a_n^{-1}b_n \text 
  { for } a_i,b_i \in S\} \subseteq e^\alpha.$$
The $G$-action on $X$ induces a homotopy $S$-action $(\varphi,H)$ on
$X$ where  $\varphi_g$ is given by $l_g \colon X \to X, \; x \mapsto gx$ 
for $g \in S$,
and $H_{g,h}(-,t) = l_{gh}$ for  $g,h \in S$ with $gh \in S$ and $t \in [0,1]$.
Notice that in this case
\begin{eqnarray*}
F_g(\varphi,H)
& = &
\{l_g \colon X \to X\};
\\
S^n_{\varphi,H}(g,x)
 & = &
\{gl,l^{-1}x) \mid l = a_1^{-1}b_1\dots a_n^{-1}b_n \text { for } a_i,b_i \in S\}.
\end{eqnarray*}
Hence $\calv_\alpha$ is $S$-long with respect to $(\varphi,H)$.
\end{proof}

\begin{proposition}
  \label{prop:CAT(0)-groups-satisfy}
  Every finite dimensional $\CAT(0)$-group is
  transfer reducible to
  the family $\VCyc$ of virtually cyclic subgroups.
\end{proposition}

The proof of this result is postponed 
to~\cite{Bartels-Lueck(2010CAT(0)flow)}.

Let $\calfj^K$ be the class of groups satisfying
the $K$-theoretic Farrell-Jones Conjecture with coefficients
in arbitrary additive $G$-categories $\cala$, i.e., the class
of groups for which the assembly map~\eqref{eq:VC-assembly-map-K} 
is an isomorphism for all $\cala$.
By $\calfj^K_1$ we denote the class of groups for which this assembly 
map is bijective in degree $m \le 0$ and surjective in degree $m = 1$ for 
any $\cala$.
Let $\calfj^L$ be the class of groups satisfying
the $L$-theoretic Farrell-Jones Conjecture with coefficients
in arbitrary additive $G$-categories $\cala$ with involutions, i.e., the class
of groups for which the assembly map~\eqref{eq:VC-assembly-map-K} 
is an isomorphism for all $\cala$.
(We could define $\calfj_1^L$, but because of the $4$-periodicity of
$L$-theory this is the same as $\calfj^L$.)

\begin{lemma}
\label{lem:calfj_closed_under}
Let $\calc$ be one of the classes $\calfj^K_1$, $\calfj^L$.
\begin{enumerate}
\item \label{lem:calfj_closed_under:subgroups} 
      If $H$ is a subgroup of $G$ and $G \in \calc$, then $H \in \calc$;
\item \label{lem:calfj_closed_under:transitivity}
      Let $\pi \colon G \to H$ be a group homomorphism. 
      If $H \in \calc$ and $\pi^{-1}(V) \in \calc$ for
      all virtually cyclic subgroups $V$ of $H$, then $G \in \calc$;
\item \label{lem:calfj_closed_under:direct_products} 
      If $G_1$ and $G_2$ belong to
      $\calc$, then $G_1 \times G_2$ belongs to $\calc$;
\item \label{lem:calfj_closed_under:free-products} 
      If $G_1$ and $G_2$ belong to $\calc$, then
      $G_1 * G_2$ belongs to $\calc$;
\item \label{lem:calfj_closed_under:colim} 
      Let $\{G_i \mid i\in I\}$ be a
      directed system of groups (with not necessarily injective structure maps)
      such that $G_i \in \calc$ for $i \in I$.
      Then $\colim_{i \in I} G_i$ belongs to $\calc$.
\end{enumerate}
\end{lemma}

\begin{proof}
Note first that the product of two virtually cyclic groups  acts 
properly, isometrically and cocompactly on a proper complete 
$\CAT(0)$-space with finite covering dimension, namely $\IR^2$.
Thus such a product is a $\CAT(0)$ group.
It follows from Theorem~\ref{the:axiomatic}
and Proposition~\ref{prop:CAT(0)-groups-satisfy}
that such products belong to $\calfj_1^K \cap \calfj^L$.
Note also that finitely generated virtually free groups
are hyperbolic and belong $\calfj_1^K \cap \calfj^L$
by Theorem~\ref{the:axiomatic}
and Proposition~\ref{prop:CAT(0)-groups-satisfy}.

For $\calfj^L$ properties~\ref{lem:calfj_closed_under:subgroups},~%
\ref{lem:calfj_closed_under:transitivity},~%
\ref{lem:calfj_closed_under:direct_products},
and~\ref{lem:calfj_closed_under:colim} follow
from~\cite[Corollary~0.8, Corollary~0.9, Corollary~0.10, Remark~0.11]
            {Bartels-Lueck(2009coeff)}.
For~\ref{lem:calfj_closed_under:free-products} we will use a trick
from~\cite{Roushon(2008FJJ3)}. 
For $G_1$, $G_2 \in \calfj_1$ consider the canonical map
$p \colon G_1 * G_2 \to G_1 \x G_2$.
We have already shown that \eqref{eq:VC-assembly-map-L}
is an isomorphism for $G_1 \x G_2$.
By~\ref{lem:calfj_closed_under:transitivity} 
it suffices to show the same for $p^{-1}(V)$
for all virtually cyclic subgroups $V$ of $G_1 \x G_2$.
By~\cite[Lemma~5.2]{Roushon(2008FJJ3)}
all such $p^{-1}(V)$ are virtually free.
Such a virtually free group is the colimit of its finitely generated
subgroups which are again virtually free.
Thus~\ref{lem:calfj_closed_under:colim} implies that
virtually free groups belong to $\calfj^L$.
The $K$-theoretic case can be proved completely
analogously. One has to check that the argument works also for the 
statement that $K$-theoretic assembly map is bijective in degree $m \le 0$ 
and surjective in degree $m = 1$. This follows  from the fact
that taking the colimit over a directed system is an exact functor. 
\end{proof}

The above arguments also show that $\calfj^K$ 
satisfies assertions~\ref{lem:calfj_closed_under:subgroups},%
\ref{lem:calfj_closed_under:transitivity}  
and~\ref{lem:calfj_closed_under:colim} of 
Lemma~\ref{lem:calfj_closed_under}.
Assertions~\ref{lem:calfj_closed_under:direct_products} 
and~\ref{lem:calfj_closed_under:free-products} follow
once the $K$-theoretic Farrell-Jones Conjecture is established 
for groups of the form $V \x V'$, where $V$ and $V'$ are virtually cyclic.
For arbitrary additive $G$-categories $\cala$ this has not been carried out.
See~\cite{Quinn(2005)} for positive results in this direction.

\begin{proof}[Proof of Theorem~B]
In the language of this section Theorem~B can be rephrased to the
statement that $\calb \subseteq \calfj^K_1 \cap \calf^L$.
Propositions~\ref{prop:hyperbolic-satisfy} 
and~\ref{prop:CAT(0)-groups-satisfy}
show that Theorem~\ref{the:axiomatic} applies to hyperbolic
groups and finite dimensional $\CAT(0)$-groups.
Thus all such groups are contained in $\calfj^K_1 \cap \calf^L$.
Lemma~\ref{lem:calfj_closed_under}
implies now that $\calb \subseteq \calfj^K_1 \cap \calf^L$.
\end{proof}


\typeout{----------- S-long covers yield contracting maps -------------}

\section{$S$-long covers yield contracting maps}
\label{sec:S-long_covers_yield_contracting_maps}

\begin{summary*}
  The main result of this section is Proposition~\ref{prop:contracting-maps}
  in which we convert long covers of $G \x X$ in the sense of
  Definition~\ref{def:S-action_plus_long-covers}~\ref{def:S-action_plus_long-covers:long}
  to $G$-equivariant maps $G \x X \to \Sigma$,
  where $\Sigma$ is simplicial complex whose dimension
  is uniformly bounded and whose isotropy groups are not
  to large.
  Moreover, these  maps have  strong contracting property with respect to
  the metric $d_{S,\Lambda}$ from Definition~\ref{def:d_S_lambda}.
  This metric scales (small) distances in the $X$-direction
  by $\Lambda$ (Lemma~\ref{lem:properties-of-d-S-K}~\ref{lem:properties-of-d-S-K:open}),
  while distances in the 
  $G$-direction along the homotopy action are not scaled   
  (Lemma~\ref{lem:properties-of-d-S-K}~\ref{lem:properties-of-d-S-K:blows-up}).
\end{summary*}

Throughout this section we fix the following convention.

\begin{convention}
Let
\begin{itemize}
\item $G$ be a group;
\item $(X,d_X)$ be a compact metric space.
      We equip $G \times X$ with the $G$-action
      $g(h,x) = (gh,x)$;
\item $S$ be a finite subset of $G$ (containing $e$);
\item $(\varphi,H)$ be a homotopy $S$-action on $X$.
\end{itemize}
\end{convention}


\subsection{Homotopy $S$-actions and metrics}
\label{subsec:homotopy-S-actions_plus_metrics}

For every number $\Lambda > 0$ we define a $G$-invariant
(quasi-)metric $d_{S,\Lambda}$ on $G \times X$ as follows.
For $(g,x),(h,y) \in G \times X$ consider $n \in \IZ$, $n \ge 0$,
elements $x_0,\dots,x_n \in X$, $z_0,\dots,z_n$ in $X$,
elements $a_1,b_1,\dots,a_n,b_n$ in $S$ and
maps $f_1,\widetilde{f}_1,\dots,f_n,\widetilde{f}_n \colon X \to X$
such that
\begin{equation}
\label{eq:x-z-f-a-b-satisfy}
\begin{array}{l}
   x = x_0, z_n = y,
   \\f_i \in F_{a_i}(\varphi,H),   \widetilde{f}_i \in F_{b_i}(\varphi,H),
   f_i(z_{i-1}) = \widetilde{f}_i(x_i) \;\text{for } i = 1,2, \dots n;\\
   h = g a_1^{-1} b_1 \dots a_n^{-1} b_n.
\end{array}
\end{equation}
(See Definition~\ref{def:S-action_plus_long-covers}~%
\ref{def:S-action_plus_long-covers:F}
for the definition of $F_s(\varphi,H)$ for $s \in S$.)
If $n = 0$, we just demand $x_0 = x$, $z_0 = y$, $g = h$ and no elements
$a_i$, $b_i$, $f_i$ and $\widetilde{f}_i$ occur.
To this data we associate the number
\begin{equation}
  \label{eq:d_S-for-fixed-sequence}
n + \sum_{i= 0}^n\Lambda \cdot d_X(x_i,z_i).
\end{equation}
\begin{definition} \label{def:d_S_lambda}
For $(g,x), (h,y) \in G \times X$ define
$$d_{S,\Lambda}((g,x),(h,y)) \in [0,\infty]$$
as the infimum of~\eqref{eq:d_S-for-fixed-sequence} over all possible
choices of $n$, $x_i$,  $z_i$,$a_i$, $b_i$, $f_i$ and $\widetilde{f}_i$.
If the set of possible choices is empty, then we put
$d_{S,\Lambda}((g,x),(h,y)) := \infty$.
\end{definition}
Of course, $d_{S,\Lambda}$ depends not only on $S$ and $\Lambda$,
but also on $(X,d)$ and $(\varphi,H)$.
That this is not reflected in the notation will hopefully not
be a source of confusion. 
Recall that a quasi-metric is the same as a metric
except that it may take also the value $\infty$.

\begin{lemma}
  \label{lem:properties-of-d-S-K}
\begin{enumerate}
\item \label{lem:properties-of-d-S-K:metric}
      For every $\Lambda > 0$, $d_{S,\Lambda}$ is a well-defined
      $G$-invariant quasi-metric on $G \times X$.
      The set $S$ generates $G$ if and only if $d_{S,\Lambda}$ is a metric;
\item \label{lem:properties-of-d-S-K:blows-up}
      Let $(g,x),(h,y) \in G \times X$ and let $m \in \IZ$, $m \ge 1$.
      If  $d_{S,\Lambda}( (g,x), (h,y) ) \leq m$
      for all $\Lambda$, then $(h,y) \in S^{m}_{\varphi,H}(g,x)$;
      (The set $S^{m}_{\varphi,H}(g,x)$ is defined in
      Definition~\ref{def:S-action_plus_long-covers}~\ref{def:S-action_plus_long-covers:S(g,x)}.)
\item \label{lem:properties-of-d-S-K:open}
For $x,y \in X$ and $g \in G$ we have
$d_{S,\Lambda}((g,x),(h,y)) < 1$ if and only if $g = h$ and $\Lambda \cdot  d_X(x,y) < 1$ hold.
In this case we get
$$d_{S,\Lambda}((g,x),(h,y)) = \Lambda \cdot d_X(x,y).$$
The topology on $G \times X$ induced by $d_{S,\Lambda}$ is the product
topology on $G \times X$ for the discrete topology on $G$ and the given one on $X$.

\end{enumerate}
\end{lemma}

\begin{proof}\ref{lem:properties-of-d-S-K:metric}
One easily checks that $d_{S,\Lambda}$ is symmetric and
satisfies the triangle inequality.
Obviously $d_{S,\Lambda}( (g,x), (g,x) ) = 0$.
Suppose $d_{S,\Lambda}( (g,x), (h,y) ) = 0$. 
Given any real number $\epsilon$ with
$0 < \epsilon < 1$, we can find
$n$, $x_i$,  $z_i$,$f_i$, $\widetilde{f}_i$,$a_i$ and $b_i$ 
as in \eqref{eq:x-z-f-a-b-satisfy} satisfying
$$n + \sum_{i=0}^{n} \Lambda \cdot d_X(x_i,z_i) \le \epsilon.$$
We conclude $n = 0$ and hence $\Lambda \cdot d_X(x,y) \le \epsilon$.
Since $\Lambda > 0$ and this holds for all $0 < \epsilon < 1$, we conclude
$d_X(x,y) = 0$ and hence $x = y$.

Obviously $d_{S,\Lambda}$ is $G$-invariant since for $k \in G$ we have
 $h = g a_1^{-1} b_1 \dots a_n^{-1} b_n$ if and only if
$kh = kg a_1^{-1} b_1 \dots a_n^{-1} b_n$ and $G$ acts on $G \times X$
by $k \cdot (h,x) = (kh,x)$.

The sets $F_g(\varphi,H)$ for $g \in S$ are never empty and
$F_e(\varphi,H)$ contains always $\id_X$.
Hence the infimum in the definition of $d_{S,\Lambda}((g,x),(h,y))$ is finite, if
and only if  we can find $n \in \IZ, n \ge 0$ and elements $a_i,b_i \in S$ with
$g^{-1}h = a_1^{-1} b_1 \dots a_n^{-1} b_n$.
\\[1ex]\ref{lem:properties-of-d-S-K:blows-up}
Let $(\Lambda^{\nu})_{\nu \ge 1}$ be sequence of numbers such that
$\lim_{\nu \to \infty}\Lambda^{\nu} = \infty$.
The assumptions imply that there are
$n^{\nu}$, $x_0^{\nu},\dots,x^v_{n^{\nu}}$, $z_0^{\nu},\dots,z^{\nu}_{n^{\nu}}$,
$a_1^{\nu},b_1^{\nu},\dots,a_{n^{\nu}}^{\nu},b_{n^{\nu}}^{\nu} \in S$ and
$f_1^{\nu},\widetilde{f}_1^{\nu},\dots,f_{n^{\nu}}^{\nu},
                            \widetilde{f}_{n^{\nu}}^{\nu}$
such that~\eqref{eq:x-z-f-a-b-satisfy}  and
\begin{equation}
  \label{eq:d-less-than-m-for-nu}
n^{\nu} + \sum_{i=0}^{n^{\nu}}  \Lambda^{\nu} \cdot d_X(x^{\nu}_i,z^{\nu}_i)
                                                    < m + 1 / \nu
\end{equation}
hold.
In particular, $n^{\nu} \leq m$ for all $\nu$. For each $\nu$
we define $a_j^{\nu} = b_j^{\nu} = e$, $x_j^{\nu} = z_j^{\nu} = y$,
$f_j^{\nu} = \widetilde{f}_j^{\nu} = \id_X$
for $j \in \{n^{\nu}+1, \dots ,m\}$.
Hence we have now for each $\nu$ and each $i \in \{1,2, \dots ,m\}$
elements $a_i^{\nu}$, $b_i^{\nu}$, $x_i^{\nu}$,$z_i^{\nu}$, $f_i^{\nu}$
and $\widetilde{f_i}^{\nu}$ and $x_0^{\nu} = x$ and $z_m^{\nu} = y$.

Because $X$ is compact, we can arrange by passing to a 
subsequence of $(\Lambda^{\nu})_{\nu \ge 1}$
that for each $i \in \{0,1,2, \dots ,m\}$ there are
$x_i\in X$ with $\lim_{\nu \to \infty} x_i^{\nu} = x_i$ and $z_i \in X$ with 
$\lim_{\nu \to \infty} z_i^{\nu} \to z_i$.
{}From~\eqref{eq:d-less-than-m-for-nu} we deduce for $i \in \{0,1,2,\dots ,m\}$.
\begin{equation*}
d_X(x_i^{\nu},z_i^{\nu}) < \frac{m + 1/ \nu}{\Lambda^{\nu}}.
\end{equation*}
Since $\lim_{\nu \to \infty} \frac{m + 1/ \nu}{\Lambda^{\nu}} = 0$,
we conclude $d_X(x_i,z_i) = 0$ and therefore
$$x_i = z_i \quad \;\text{for } i \in \{0,1,2, \dots ,m\}.$$

Choose for $i \in \{0,1,2, \dots ,m\}$  elements
$t_i^{\nu}, \widetilde{t}_i^{\nu} \in [0,1]$,
$r_i^{\nu},s_i^{\nu}, \widetilde{r}_i^{\nu}, \widetilde{s}_i^{\nu} \in S$
with $r_i^{\nu} s_i^{\nu} = a_i^{\nu}$ and
$\widetilde{r}_i^{\nu}\widetilde{s}_i^{\nu} = b_i^{\nu}$ such that
$f_i^{\nu}= H_{r_i^{\nu},s_i^{\nu}}(-,t_i^{\nu})$ and
$\widetilde{f}_i^{\nu} = 
  H_{\widetilde{r}_i^{\nu},\widetilde{s}_i^{\nu}}(-,\widetilde{t}_i^{\nu})$
holds. 
Since $S$ is finite and $[0,1]$ is compact, we can arrange by passing to a
subsequence of  $\{\Lambda^{\nu}\}$ that there exist elements
$r_i,s_i,\widetilde{r}_i,\widetilde{s}_i \in S_i$ and
$t_i,\widetilde{t}_i \in [0,1]$ such that
$r_i^{\nu} = r_i$, $s_i^{\nu} = s_i$, $\widetilde{r}_i^{\nu} = \widetilde{r}_i$
and $\widetilde{s}_i^{\nu} = \widetilde{s}_i$ holds for all $\nu$ and
$\lim_{\nu \to \infty} t_i^{\nu} = t_i$ and
$\lim_{\nu \to \infty} \widetilde{t}_i^{\nu} = \widetilde{t}_i$ is valid. 
Put $f_i= H_{r_i,s_i}(-,t_i)$ and
$\widetilde{f}_i= H_{\widetilde{r}_i,\widetilde{s}_i}(-,\widetilde{t}_i)$.
Then for $i \in \{0,1,2, \dots m\}$
\begin{eqnarray*}
f_i & \in & F_{a_i}(\varphi,H);
\\
\widetilde{f}_i & \in & F_{b_i}(\varphi,H);
\\
\lim_{\nu \to \infty} f_i^{\nu}(x_{i-1}^{\nu}) & = & f_i(x_{i-1});
\\
\lim_{\nu \to \infty} \widetilde{f}_i^{\nu}(z_i^{\nu})  & = & \widetilde{f}_i(z_i).
\end{eqnarray*}

To summarize, we have constructed $x_0,\dots,x_{m} \in X$,
$a_1,b_1,\dots,a_{m},b_{m} \in S$,
$f_1,\widetilde{f}_1,\dots f_{m},\widetilde{f}_{m} \colon X \to X$
such that
$x_0 = x$, $x_{m} = y$, $f_i \in F_{a_i}(\varphi,H)$,
$\widetilde{f}_i \in F_{b_i}(\varphi,H)$,
$f_i(x_{i-1}) = \widetilde{f}_i(x_i)$
for $i \in \{1,2, \dots ,m\}$ and
$h = g a_1^{-1} b_1 \dots a_{m}^{-1} b_{m}$ holds.
Thus $(h,y) \in S^{m}_{\varphi,H}(g,x)$.
\\[1ex]\ref{lem:properties-of-d-S-K:open}
Suppose $d_{S,\Lambda}((g,x),(h,y)) < 1$. 
For every $\epsilon > 0$ with
$\epsilon < 1 - d_{S,\Lambda}((g,x),(h,y))$ we can find appropriate
$n$, $x_i$,  $z_i$,$f_i$, $\widetilde{f}_i$,$a_i$ and $b_i$ with
$$n + \sum_{i=0}^{n}  \Lambda \cdot d_X(x_i,z_i) < 
                d_{S,\Lambda}((g,x),(h,y)) + \epsilon.$$
Since $d_{S,\Lambda}((g,x),(h,y)) + \epsilon < 1$, we conclude $n = 0$
and hence  $g = h$ and 
$\Lambda \cdot d_X(x,y) < d_{S,\Lambda}((g,x),(h,y)) + \epsilon$.
Since this holds for all such $\epsilon$, we get
$\Lambda \cdot d_X(x,y) \le d_{S,\Lambda}((g,x),(h,y))$.  
Obviously
$\Lambda \cdot d_X(x,y) \ge d_{S,\Lambda}((g,x),(h,y))$ because of $g = h$.
This proves $d_{S,\Lambda}((g,x),(h,y)) = \Lambda \cdot d_X(x,y)$ and $g = h$ 
provided that $d_{S,\Lambda}((g,x),(h,y)) < 1$.

One easily checks that $g = h$ and $d_X(x,y) < 1$ implies
$\Lambda \cdot d_{S,\Lambda}((g,x),(h,y)) < 1$.

The claim about the topology is now obvious.
\end{proof}


\subsection{Contracting maps}
\label{subsec:contracting-maps}

\begin{proposition}
  \label{prop:large-Lebesgue}
Let $\calu$ be an $S$-long finite-dimensional $G$-equivariant
cover of $G \times X$. Let $m$ be any number with $m \le |S|$.  Then there
is $\Lambda > 0$ such that the Lebesgue number of $\calu$ with respect
to $d_{S,\Lambda}$ is at least $m/2$, i.e., for every $(g,x)$ there is $U
\in \calu$ containing the open $m/2$-ball $B_{m/2,\Lambda}(g,x)$
around $(g,x)$ with respect to the metric $d_{S,\Lambda}$.
\end{proposition}

\begin{proof}
Fix $x \in X$. First we show the existence of $\Lambda_x > 0$
and $U_x \in \calu$ such that the  open $m$-ball $B_{m,\Lambda_x}(e,x)$
around $(e,x)$ with respect to $d_{S,\Lambda_x}$ lies in $U_x$.

Let $\calu_x := \{ U \in \calu \mid (e,x) \in U \}$.
Then $\calu_x$ is finite, because $\calu$ is finite dimensional.
We proceed by contradiction. So assume that for every $\Lambda > 0$
no $U \in \calu_x$ contains $B_{m,\Lambda}(e,x)$.
Thus we can find a monotone increasing sequence $(\Lambda_n)_{n \ge 1}$ 
of positive real numbers
with $\lim_{n \to \infty} \Lambda_n = \infty$
such that for every $U \in \calu_x$ and $n \ge 1$ there is
$(h^{U,n},y^{U,n}) \in (G \times X) \setminus U$ satisfying
\begin{equation}
\label{eq:distance-from-g-x-to-h-y-U-n}
d_{\Lambda_n,S}((e,x),(h^{U,n},y^{U,n})) < m.
\end{equation}
Because $X$ is compact, we can arrange by passing to a 
subsequence of $(\Lambda_n)_{n \ge 1}$
that for each $U \in \calu_x$ there is $y^U \in X$ satisfying 
$\lim_{n \to \infty} y^{U,n} = y^U$.
The definition of $d_{\Lambda_n,S}$ and
\eqref{eq:distance-from-g-x-to-h-y-U-n} imply
that each $h^{U,n}$ can be written as a product
of at most $2m$ elements from $S \cup S^{-1}$.
Therefore the $h^{U,n}$-s range over a finite subset of $G$.
Thus we can arrange by passing to a subsequence of $(\Lambda_n)_{n \ge 1}$
that for each $U \in \calu_x$ there is $h^U \in G$
such that $h^{U,n} = h^U$ holds for all $n$.
We get for $k \ge n$ from~\eqref{eq:distance-from-g-x-to-h-y-U-n}
since $\Lambda_n \le \Lambda_k$
\begin{eqnarray*}
d_{S,\Lambda_n}\left((e,x),(h^U,y^U)\right)
& \le &
d_{S,\Lambda_n}\left((e,x),(h^U,y^{U,k})\right) + 
d_{S,\Lambda_n}\left((h^U,y^{U,k}),(h^U,y^{U})\right)
\\
& \le &
d_{S,\Lambda_k}\left((e,x),(h^U,y^{U,k})\right) + 
d_{S,\Lambda_n}\left((h^U,y^{U,k}),(h^U,y^{U})\right)
\\
& < &
m + d_{S,\Lambda_n}\left((h^U,y^{U,k}),(h^U,y^{U})\right).
\end{eqnarray*}
Lemma~\ref{lem:properties-of-d-S-K}~\ref{lem:properties-of-d-S-K:open}
implies 
$\lim_{k \to \infty} d_{S,\Lambda_n}\left((h^U,y^{U,k}),(h^U,y^{U})\right) = 0$.
We conclude $d_{S,\Lambda_n}\left((e,x),(h^U,y^U)\right) \le  m$ for all
$U \in \calu_x$. 
By Lemma~\ref{lem:properties-of-d-S-K}~\ref{lem:properties-of-d-S-K:blows-up}
this implies $(h^U,y^U) \in S^m_{\varphi,H}(g,x)$ for all $U \in \calu_x$. 
Because $\calu$ is assumed to be $S$-long there is $U_0 \in \calu_x$ such that
$S^m_{\varphi,H}(g,x) \subseteq U_0$. 
Thus $(h^{U_0},y^{U_0}) \in U_0$. 
But this yields the desired contradiction: 
$$\lim_{n \to \infty} (h^{U_0},y^{U_0,n}) =
\lim_{n \to \infty} (h^{U_0,n},y^{U_0,n}) = (h^{U_0},y^{U_0})$$
together with the fact that $(h^{U_0,n},y^{U_0,n})$ lies in the closed 
subset $(G \times X) \setminus U_0$, implies 
$(h^{U_0},y^{U_0}) \in (G \times X) \setminus U_0$.

Now we can finish the proof of Proposition~\ref{prop:large-Lebesgue}.
For $x \in X$ the subset
$$B_{m/2,\Lambda_x}(e,x)\cap \{e\} \times X \subseteq \{e\} \times X = X$$
is open in $X$ because of
Lemma~\ref{lem:properties-of-d-S-K}~\ref{lem:properties-of-d-S-K:open}.
Since $X$ is compact, we can find finitely many elements 
$x_1$, $x_2$, $\dots$, $x_l$
such that
$$X = \{e\} \times X
= \bigcup_{i=1}^l \left(B_{m/2,\Lambda_{x_i}}(e,x_i)\cap \{e\} \times X\right).$$
Put $\Lambda := \max\{\Lambda_{x_1}, \dots ,\Lambda_{x_l}\}$. 
Consider $(g,x) \in G \times X$.
Then we can find $i \in \{1,2, \dots ,l\}$ such that 
$(e,x) \in B_{m/2,\Lambda_{x_i}}(e,x_i)$.
Hence
$$B_{m/2,\Lambda}(e,x) \subseteq  B_{m/2,\Lambda_{x_i}}(e,x) 
\subseteq B_{m,\Lambda_{x_i}}(e,x_i).$$
We have already shown that there exists
$U \in \calu$ with $B_{m,\Lambda_{x_i}}(e,x_i) \subseteq U$. This implies
$$B_{m/2,\Lambda}(g,x) = g\left(B_{m/2,\Lambda}(e,x)\right) \subseteq g(U).$$
Since $\calu$ is $G$-invariant, this finishes the proof of
Proposition~\ref{prop:large-Lebesgue}.
\end{proof}

In the following proposition $d^1$ denotes the $l^1$-metric on simplicial
complexes, compare~\cite[Subsection~4.2]{Bartels-Lueck-Reich(2008hyper)}.

\begin{proposition}
  \label{prop:contracting-maps}
  Let $G$ be a finitely generated group that is transfer
  reducible over the family $\calf$.
  Let $N$ be the number appearing in 
  Definition~\ref{def:transfer-reducible}.
  Let $S$ be a finite subset of $G$ (containing $e$) that generates $G$.
  Let $\e > 0$, $\beta > 0$.
  Then there are
  \begin{itemize}
  \item a compact contractible controlled $N$-dominated metric
        space $(X,d)$;
  \item a homotopy $S$-action $(\varphi,H)$ on $X$;
  \item a positive real number $\Lambda$;
  \item a simplicial complex $\Sigma$ of dimension
        $\le N$ with a simplicial cell
        preserving $G$-action;
  \item a $G$-equivariant map $f \colon G \times X \to \Sigma$,
  \end{itemize}
  satisfying:
  \begin{enumerate}
  \item The isotropy groups of $\Sigma$ are members of $\calf$;
  \item \label{prop:contractin-maps:estimate}
        If $(g,x),(h,y) \in G \times X$ and 
        $d_{S,\Lambda}((g,x),(h,y)) \leq \beta$,
        then $$d^1(f(g,x),f(h,y)) \leq \epsilon.$$
  \end{enumerate}
\end{proposition}

\begin{proof}
Set $D := \frac{|S|}{2}$. Since for $S \subseteq T$ we have
$d_{T,\Lambda} \le d_{S,\Lambda}$, we can arrange by possibly
enlarging $S$ 
\begin{equation*}
\beta  \le  \frac{D}{4N} \quad \mbox{and} \quad
\frac{16N^2\beta}{D}  \le  \epsilon.
\end{equation*}
Because $G$ is transfer reducible over $\calf$ there exists 
a contractible compact controlled $N$-dominated space $X$,
a homotopy $S$-action $(\varphi,H)$ on $X$ and an
$S$-long cover $\calu$ of $G \times X$
such that $\calu$ is an $N$-dimensional open $\calf$-covering.
Using Proposition~\ref{prop:large-Lebesgue}
we find $\Lambda > 0$ such that the Lebesgue number of $\calu$
with respect to $d_{S,\Lambda}$ is at least $D$.
Let $\Sigma := |\calu|$ be the realization of the nerve
of $\calu$.
Since $\calu$ is an open $\calf$-cover, $\Sigma$ inherits
a simplicial cell preserving $G$-action whose isotropy groups
are members of $\calf$.
Let now $f \colon G \times X \to \Sigma$ be the map induced by
$\calu$, i.e.,
\begin{equation*}
  f(x) := \sum_{U \in \calu}
          \frac{d_{S,\Lambda}(x,G \times X - U)}
               {\sum_{V \in \calu} d_{S,\Lambda}(x,G \times X - V) } U.
\end{equation*}
This is a $G$-equivariant map since $d_{S,\Lambda}$ is $G$-invariant.
From~\cite[Proposition~5.3]{Bartels-Lueck-Reich(2008hyper)}
we get
\begin{equation*}
d_{S,\Lambda} ((g,x),(h,y)) \leq \frac{D}{4N} \quad \implies \quad
d^1(f(g,x),f(h,y)) \leq \frac{16 N^2}{D} d_{S,\Lambda}((g,x),(h,y)).
\end{equation*}
We conclude
\begin{equation*}
d_{S,\Lambda} ((g,x),(h,y)) \leq \beta \quad \implies \quad
d^1(f(g,x),f(h,y)) \leq \frac{16 N^2\beta}{D} \le \epsilon.
\end{equation*}
This finishes the proof of Proposition~\ref{prop:contracting-maps}.
\end{proof}


\typeout{-------- Controlled algebra with a view towards L-theory --------}

\section{Controlled algebra with a view towards $L$-theory}
\label{sec:controlled-algebra_plus_L}

\begin{summary*}
  A crucial tool in the proof of Theorem~\ref{the:axiomatic}
  is controlled algebra.
  In this section we give a  brief review of this theory 
  where we emphasize the $L$-theory aspects.
  In Subsection~\ref{subsec:obstruction-category} we define
  the obstruction categories whose $K$- respectively 
  $L$-theory will appear as the homotopy groups of homotopy fibers
  of assembly maps.
  Elements in these $K$- and $L$-groups will be represented 
  by chain homotopy equivalences in $K$-theory and by
  ultra-quadratic Poincar\'e complexes over
  these categories.
  (These are the cycles referred to in the introduction.)
\end{summary*}


\subsection{Additive $G$-categories with involution}
\label{subsec:additive-G-categories}

By an \emph{additive category $\cala$} we will mean from now on a 
small additive 
category with a functorial strictly associative direct sum.
For a group $G$ an \emph{additive $G$-category} is by 
definition such an additive 
category together with a strict (right) $G$-action 
that is compatible with the direct sum. 
By an \emph{additive $G$-category with involution} we will mean an additive
$G$-category that carries in addition a strict involution $\inv$
(i.e., $\inv \circ \inv = \id_\cala$ that is
strictly compatible with the $G$-action (i.e., $\inv \circ g = g \circ \inv$)
and the sum (i.e., $\inv(A \oplus B) = \inv(A) \oplus \inv(B)$), 
see~\cite[Definition~10.6]{Bartels-Lueck(2009coeff)}.
The assembly maps~\eqref{eq:assembly-map-calf-K}
and~\eqref{eq:assembly-map-calf-2-L} are defined for more general $\cala$, 
but the assembly maps are isomorphisms for all such more general
$\cala$ if and only if they are isomorphism for all 
additive $G$-categories (with involution) as above, 
see~\cite[Theorem~0.12]{Bartels-Lueck(2009coeff)}.  


\subsection{The category $\calc^G(Y,\cale,\calf;\cala)$.}
\label{subsec:category-calc-Y-cale-calf}

Let $G$ be a group, $Y$ a space and $\cala$ be
a additive category.
Let $\cale \subseteq \{E \mid  E \subseteq Y \times Y \}$ and
$\calf \subseteq \{F \mid F \subseteq Y \}$ be collections
satisfying the conditions from~\cite[page~167]{Bartels-Farrell-Jones-Reich(2004)}.
(These conditions are designed to ensure that we indeed obtain an
additive category (with involution) 
and are satisfied in all cases of interest.)
The category $\calc(Y;\cale,\calf;\cala)$ is defined as follows.
Objects are given by sequences $(M_y)_{y \in Y}$ of objects
in $\cala$ such that
\begin{numberlist}
\item [\label{nl:calf-controlled}]
      $M$ is $\calf$-controlled: there is $F$ in $\calf$ such
      that the \emph{support} $\supp M := \{ y \mid M_y \neq 0 \}$
      is contained in $F$;
\item [\label{nl:locally-finite}]
      $M$ has locally finite support: for every $y \in Y$
      there is an open neighborhood $U$ of $y$ such that
      $U \cap \supp M$ is finite.
\end{numberlist}
A morphism $\psi$ from $M = (M_y)_{y \in Y}$ to
$M' = (M'_y)_{y \in Y}$
is given by a collection
$(\psi_{y',y} \colon M_y \to M_{y'}')_{(y',y) \in Y \times Y}$
of morphisms in $\cala$ such that
\begin{numberlist}
\item [\label{nl:psi-cale-controlled}]
      $\psi$ is $\cale$-controlled: there is $E \in \cale$
      such that the \emph{support}
      $\supp(\psi) := \{ (y',y) \mid \psi_{y',y} \neq 0 \}$
      is contained in $E$;
\item [\label{nl:psi-finite-cr}]
      $\psi$ is row and column finite:
      for every $y \in Y$ the sets $\{ y' \in Y \mid (y,y') \in \supp \psi \}$
      and $\{ y' \in Y \mid (y',y) \in \supp \psi \}$ are finite.
\end{numberlist}
Composition of morphisms is given by matrix multiplication, i.e.,
$(\psi' \circ \psi)_{y'',y} = \sum_{y' \in Y} \psi_{y'',y'} \circ \psi_{y',y}$.
If $\inv \colon \cala \to \cala$ is a strict involution,
then $\calc(Y;\calf,\cale;\cala)$ inherits a strict involution.
For objects it is defined  by $(\inv (M))_y = \inv ( M_y )$,
for morphisms it is defined by $(\inv(\psi))_{y',y} = \inv ( \psi_{y,y'})$.
Let now $Y$ be a (left) $G$-space and assume that $\cala$
is equipped with a (strict) right $G$-action, i.e., $\cala$ is an additive
$G$-category.
Assume that the $G$-action on $Y$ preserves both $\calf$ and $\cale$.
Then $\calc(Y,\cale,\calf;\cala)$ inherits a (right) $G$-action
making it an additive $G$-category.
For an object $M$ and $g \in G$ the action is given by
$(Mg)_y = (M_{gy})g$.
If the action on $\cala$ is compatible with a (strict) 
involution $\inv$ on $\cala$,
i.e., if $\cala$ is an additive $G$-category with involution,
then $\calc(Y;\cale,\calf;\cala)$ is also an additive $G$-category 
with involution under the induced action and involution.
We will denote by $\calc^G(Y;\cale,\calf;\cala)$ the subcategory of
$\calc(Y;\cale,\calf;\cala)$ that is (strictly) fixed by $G$.


\subsection{Metric control - the category $\calc(Z,d;\cala)$}
\label{subsec:metric-control}

Let $(Z,d)$ be a metric space.
Let $\cale(Z,d) := \{ E_\alpha \mid \alpha > 0 \}$
where $E_\alpha := \{ (z,z') \mid d(z,z') \leq \alpha \}$.
For an additive category $\cala$ 
(with or without involution)
we define $\calc(Z,d;\cala) := \calc(Z;\cale(Z,d),\{ Z \};\cala)$.
Let $\e > 0$.
A morphism $\psi$ in $\calc(Z,d;\cala)$ is said to be
\emph{$\e$-controlled} if $\supp (\psi) \subseteq E_\e$.

The \emph{idempotent completion} $\Idem(\cala)$ of an additive category
$\cala$ is the following additive category.
Objects are morphisms $p \colon M \to M$ in $\cala$ satisfying $p^2 = p$.
A morphism $f \colon (M,p) \to (N,q)$ in $\Idem(\cala)$ is a morphism
$f \colon M \to N$ satisfying $q \circ f \circ p = f$. Composition
and the additive structure are inherited from $\cala$ in the obvious way.
Recall that for us an additive category is always understood to be small, i.e.,
the objects form a set.
If $\cala$ is an additive category which is equivalent to the category of 
finitely generated free $R$-modules, then $\Idem(\cala)$ is equivalent to the 
category of finitely generated projective $R$-modules.

An object $A = (M,p) \in \Idem(\calc(Z,d;\cala))$
(where $p \colon M \to M$ is an idempotent in
$\calc(Z,d;\cala)$) is called \emph{$\e$-controlled}
if $p$ is $\e$-controlled.
A morphism $\psi \colon (M,p) \to (M',p')$ in $\Idem(\calc(Z,d;\cala))$ is 
called \emph{$\e$-controlled} if $\psi \colon M \to M'$ is $\e$-controlled
as a morphism in $\calc(Z,d;\cala)$.
A chain complex $P$ over $\Idem(\calc(Z,d;\cala))$ is called 
\emph{$\e$-controlled} if $P_n$ is $\e$-controlled for all $n$, and the 
differential $\dd_n \colon P_n \to P_{n-1}$ is $\e$-controlled for all $n$.
A graded map $P \to Q$ of chain complex over $\Idem(\calc(Z,d;\cala))$ 
is said to be \emph{$\e$-controlled} if it consists of morphisms in 
$\Idem(\calc(Z,d;\cala))$ that are $\e$-controlled.
A chain homotopy equivalence $\psi \colon P \to Q$ of chain complexes over
$\Idem(\calc(Z,d;\cala))$ is said to be an 
\emph{$\e$-chain homotopy equivalence over $\Idem(\calc(Z,d;\cala))$}
if there is a chain homotopy inverse $\varphi$ for $\psi$
and chain homotopies $H$ from $\varphi \circ \psi$ to $\id_{P}$
and $K$ from $\psi \circ \varphi$ to $\id_{Q}$ such that $P$, $Q$, $\varphi$, $\psi$,
$H$ and $K$ are $\e$-controlled.

By $\calf(\IZ)$ we denote the following small model for the category
of finitely generated free $\IZ$-modules.
Objects are $\IZ^n$ with $n \in \IN \cup \{ 0 \}$.
Morphisms are given by matrices over $\IZ$. 
Composition is given by matrix multiplication.
The category $\calf(\IZ)$ is an additive category by taking sums of matrices
and has a (strictly associative functorial) direct sum which is given on objects
by $\IZ^m \oplus \IZ^n = \IZ^{m+n}$.
We will use the (strict) involution of additive categories on $\calf(\IZ)$
which acts as the identity
on objects and by transposition of matrices on morphisms.
We write $\calc(Z,d;\IZ) := \calc(Z,d;\calf(\IZ))$.


\subsection{The obstruction category $\calo^G(Y,Z,d;\cala)$}
\label{subsec:obstruction-category}

Let $Y$ be a $G$-space and let $(Z,d)$ be a metric space with isometric
$G$-action. 
Let $\cala$ be an additive $G$-category 
(with or without involution).
In~\cite[Definition~2.7]{Bartels-Farrell-Jones-Reich(2004)}
(see also~\cite[Section~3.2]{Bartels-Lueck-Reich(2008hyper)})
the \emph{equivariant continuous control condition}
$\cale^Y_{Gcc} \subseteq \{ E \subseteq (Y \times [1,\infty))^{\times 2} \}$
has been introduced.
Define $\cale(Y,Z,d)$ as the collection of all
$E \subseteq (G \times Z \times Y \times [1,\infty))^{\times 2}$
that satisfy the following conditions:
\begin{numberlist}
\item [\label{nl:Gcc-for-calo}]
      $E$ is $\cale^Y_{Gcc}$-controlled: there exists an element $E' \in \cale^Y_{Gcc}$
      with the property that $((g,z,y,t),(g',z',y',t')) \in~E$ implies
      $((y,t),(y',t')) \in E'$;
\item [\label{nl:G-bounded-for-calo}]
      $E$ is bounded over $G$: there is a finite subset $S$ of $G$
      with the property that $((g,z,y,t),(g',z',y',t')) \in E$ implies
      $g^{-1}g' \in S$;
\item [\label{nl:Z-bounded-for-calo}]
      $E$ is bounded over $Z$: there is $\alpha > 0$
      such that $((g,z,y,t),(g',z',y',t')) \in E$ implies
      $d(z,z') \leq \alpha$.
\end{numberlist}
We define $\calf(Y,Z,d)$ to be  the collection of all
$F \subseteq G \times Z \times Y \times [1,\infty)$ for which there is
a compact subset $K$ of $G \times Z \times Y$ such that for $(g,z,y,t) \in F$
there is $h \in G$ satisfying $(hg,hz,hy) \in K$.
Then we define
\begin{equation}
\label{eq:define-calo}
\calo^G(Y,Z,d;\cala) := \calc^G(G \times Z \times Y \times [1,\infty);
                                \cale(Y,Z,d),\calf(Y,Z,d);\cala),
\end{equation}
where we use the $G$-action on $G \times Z \times Y \times [0,\infty)$ 
given by $g(h,z,y,t) := (gh,gz,gy,t)$.
We will also use the case where $Z$ is trivial, i.e., a point,
in this case we write $\calo^G(Y;\cala)$ and drop the
point from the notation.

We remark that all our constructions on this category will happen
in the $G \x Z$ factor of $G \x Z \x Y \x [1,\infty)$;
in particular, it will not be important for the reader to know the precise
definition of the equivariant continuous control condition $\cale^Y_{Gcc}$.
(We will on the other hand use 
results from~\cite{Bartels-Lueck-Reich(2008hyper)} that depend very much on
the precise definition of $\cale^Y_{Gcc}$.)

Let $S \subseteq G$ and $\e > 0$.
A morphism $\psi$ in $\calo^G(Y,Z,d;\cala)$ is said to
be \emph{$(\e, S)$-controlled} if
$((g,z,y,t),(g',z',y',t')) \in \supp ( \psi)$
implies $d(z,z') \leq \e$ and $g^{-1}g' \in S$.
If $\psi$ is an isomorphism such that both $\psi$ and $\psi^{-1}$
are $(\e,S)$-controlled,
then $\psi$ is said to be an \emph{$(\e,S)$-isomorphism}.
An object $A = (M,p) \in \Idem(\calo^G(Y,Z,d;\cala))$
(where $p \colon M \to M$ is an idempotent in
$\calc^G(Y,Z,d;\cala)$) is called \emph{$(\e,S)$-controlled}
if $p$ is $(\e,S)$-controlled.
A morphism $\psi \colon (M,p) \to (M',p')$ in $\Idem(\calo^G(Y,Z,d;\cala))$
is called \emph{$(\e,S)$-controlled} if
$\psi \colon M \to M'$ is $(\e,S)$-controlled
as a morphism in $\calo^G(Y,Z,d;\cala)$.
A chain complex $P$ over $\Idem(\calo^G(Y,Z,d;\cala))$
is called \emph{$(\e,S)$-controlled} if $P_n$ is $(\e,S)$-controlled
for all $n$, and the differential $\dd_n \colon P_n \to P_{n-1}$
is $(\e,S)$-controlled for all $n$.
A graded map $P \to Q$ of chain complexes over
$\Idem(\calo^G(Y,Z,d;\cala))$ is called \emph{$(\e,S)$-controlled}
if it consists of $(\e,S)$-controlled morphisms in $\Idem(\calo^G(Y,Z,d;\cala))$.
A chain homotopy equivalence $\psi \colon P \to Q$ of chain complex over
$\Idem(\calo^G(Y,Z,d;\cala))$ is said to be an \emph{$(\e,S)$-chain homotopy
equivalence over $\Idem(\calo^G(Y,Z,d;\cala))$}
if there is a chain homotopy inverse $\varphi$ for $\psi$
and chain homotopies $H$ from $\varphi \circ \psi$ to $\id_{P}$
and $K$ from $\psi \circ \varphi$ to $\id_{Q}$  such that 
$P$, $Q$, $\varphi$, $\psi$, $H$ and $K$ are $(\e,S)$-controlled.
We write \emph{$\e$-controlled} for $(\e,G)$-controlled and
\emph{$S$-controlled} for $(\infty,S)$-controlled.

Note that every $S$-controlled morphism has a unique decomposition
\begin{equation}
\label{eq:decomposition-of-S-mor}
\psi = \sum_{a \in S} \psi_a
\end{equation}
where $\psi_a$ is  $\{ a \}$-controlled. Namely, put
$(\psi_a)_{(g,z,y,t),(g',z',y',t')} = \psi_{(g,z,y,t),(g',z',y',t')}$ if 
$g^{-1}g' = a$, and $(\psi_a)_{(g,z,y,t),(g',z',y',t')} = 0$  otherwise.

\begin{remark}
\label{rem:calo-is-calo}
If $G$ is finitely generated, then~\eqref{nl:G-bounded-for-calo}
can be expressed using a word-metric $d_G$ as it is done
in~\cite[Section~3.4]{Bartels-Lueck-Reich(2008hyper)}.
However, the notation there is slightly different:
the category~\eqref{eq:define-calo} is denoted
in~\cite{Bartels-Lueck-Reich(2008hyper)}
by $\calo^G(Y,G \times Z,d_G \times d;\cala)$.
\end{remark}


\subsection{Controlled algebraic Poincar{\'e} complexes}
\label{subsec:controlled-algebraic-Poincare-cx}

We give a very brief review of the part of Ranicki's
algebraic $L$-theory that we will need. 
We will follow~\cite[Section~17]{Ranicki(1992a)}.
Let $\cala$ be an additive category with involution $\inv$.
Ranicki defines for such a category $L$-groups
$L_n^{\langle j \rangle}(\cala)$ where $n \in \IZ$
and $j \in \{ 1,0,-1,\dots,-\infty \}$
(see~\cite[Definition~17.1 on page~145 and 
            Definition~17.7 on page~148]{Ranicki(1992a)}).
If $R$ is a ring with involution
and we take $\cala$ to be the additive category of finitely generated free
$R$-modules, then $L_n^{\langle 1 \rangle}(\cala)$ agrees with $L^h_n(R)$
and $L_n^{\langle 1 \rangle}(\cala)$ agrees with $L_n^p(R)$
(see~\cite[Example~17.4 on page~147]{Ranicki(1992a)}).

For a chain complex $C$ over $\cala$ we write $C^{-*}$ for
the chain complex over $\cala$ with $(C^{-*})_n := \inv(C_{-n})$
and differential $\dd_n := (-1)^n \inv(d_{-n+1})$,
where $d_n \colon C_n \to C_{n-1}$ is the $n$-th differential of $C$.
For a map $f \colon C \to D$ of degree $k$
the map $f^{-*}$ of degree $k$ is defined by
$(f^{-*})_n := (-1)^{n k} \inv(f_{-n})
\colon (D^{-*})_n \to (C^{-*})_{n+k}$.
Note that if $f = \sum_{a \in S} f_a$
and $f^{-*} = \sum_{a \in S} (f^{-*})_a$
where $f_a$ and $(f^{-*})_a$ are $\{ a \}$-controlled,
then $(f^{-*})_a = (f_{a^{-1}})^{-*}$.
A \emph{$0$-dimensional ultra-quadratic Poincar{\'e} complex}
$(C,\psi)$ over $\cala$ is a finite-dimensional
chain complex $C$ over $\cala$ together with
a chain map $\psi \colon C^{-*} \to C$ (of degree $0$),
such that $\psi + \psi^{-*}$ is a chain homotopy equivalence.
If $(C,\psi)$ is concentrated in degree $0$, then it is a quadratic form
over $\cala$.

For us the following facts will be important.
\begin{numberlist}
\item Every $0$-dimensional ultra-quadratic Poincar{\'e} complex
      $(C,\psi)$ over $\cala$ yields an element
      $[(C,\psi)] \in L_0^{\langle 1 \rangle}(\cala)$;
\item [\label{nl:things-are-equal-in-L}]
      If $(C,\psi)$ and $(D,\varphi)$ are both $0$-dimensional
      ultra-quadratic Poincar{\'e} complexes over $\cala$ and
      $f \colon C \to D$ is chain homotopy equivalence such that
      $f \circ \psi \circ f^{-*}$ is chain homotopic to $\varphi$,
      then $[(C,\psi)] = [(D,\varphi)] \in  L_0^{\langle 1 \rangle}(\cala)$;
\item Every element in $L_0^{\langle 1 \rangle}(\cala)$ can be realized
      by a quadratic form;
\item [\label{nl:K-vanishes-decorations}]
      If $K_n(\cala) = 0$  for $n \leq 1$,
      then the natural map
      $L^{\langle 1 \rangle}_0 (\cala) \to \Li_0 (\cala) = \Li_0 (\Idem \cala)$
      is an isomorphism (see~\cite[Theorem~17.2 on page~146]{Ranicki(1992a)}).
\end{numberlist}

\begin{definition}[$0$-dimensional ultra-quadratic 
                      $(\e,S)$-Poincar{\'e} complex]
\label{def:poincare}
Let $Y$ be a $G$-space. 
Let $(Z,d)$ be a metric space equipped with an isometric $G$-action.
Let $\cala$ be an additive $G$-category with involution.
Consider  $S \subseteq G$ and $\e > 0$.
A \emph{$0$-dimensional ultra-quadratic $(\e,S)$-Poincar{\'e} complex}
over $\Idem( \calo^G(Y,Z,d;\cala ) )$ is a
$0$-dimensional ultra-quadratic Poincar{\'e} complex
over $\Idem( \calo^G(Y,Z,d;\cala ) )$
such that $\psi$ is $(\e,S)$-controlled
and $\psi + \psi^{-*}$ is an $(\e,S)$-chain homotopy equivalence.
\end{definition}


\typeout{-----------  Stability and the assembly map  ---------------}

\section{Stability and the assembly map}
\label{sec:stability_plus_assembly}

\begin{summary*}
  Theorem~\ref{thm:obstruction-category} asserts that
  the vanishing of the algebraic $K$- and $L$-theory
  of the obstruction categories yields isomorphism statements for 
  the corresponding assembly maps.
  Theorem~\ref{thm:stability-element} shows that 
  sufficiently controlled chain homotopy equivalences 
  represent the trivial element the algebraic $K$-theory 
  of the obstruction  category.
  Similarly this result shows that sufficiently controlled  
  ultra-quadratic Poincar\'e complexes represent the trivial
  element in the $L$-theory of the obstruction 
  category.  
\end{summary*}

Let $\cala$ be an additive category with involution.  
Its $L$-groups $\Li_n(\cala)$, $n \in \IZ$ can be constructed 
as the homotopy groups
of a (non-connective) spectrum $\bfLi(\cala)$ which is constructed
in~\cite[Definition~4.16]{Carlsson-Pedersen(1995a)} following ideas of
Ranicki.  
Similarly, the $K$-groups $K_n(\cala)$, $n \in \IZ$ of an
additive category $\cala$ are defined as the homotopy groups of a
(non-connective) spectrum $\bfK(\cala)$ which has been constructed
in~\cite{Pedersen-Weibel(1985)}.  (See~\cite[Section~2.1
and~2.5]{Bartels-Farrell-Jones-Reich(2004)} for a brief review.) If
$R$ is a ring (with involution) and one takes $\cala$ to be the
category of finitely generated free $R$-modules, then $\Li_n(\cala)$
and $K_n(\cala)$ reduce to the classical groups $\Li_n(R)$ and
$K_n(R)$ for all $n \in \IZ$

It is of course well-known that the functors $\bfLi$ and
$\bfK$ have very similar properties.
To emphasize this, we recall the following
important properties of these functors.
Recall that an additive category (with involution) is called
\emph{flasque} if there is a functor of such categories
$\Sigma^{\infty} \colon \cala \to \cala$ together with a
natural equivalence of functors of such categories
$\id_{\cala} \oplus \Sigma^{\infty} \xrightarrow{\cong} \Sigma^{\infty}$.
A functor $F \colon \cala \to \calb$ of additive categories 
(with or without involutions) is called
an \emph{equivalence} if for any object $B \in \calb$ there is an 
object $A \in \cala$
such that $F(A)$ and $B$ are isomorphic in $\calb$ and for any two objects
$A_0,A_1 \in \cala$ the map 
$\mor_{\cala}(A_0,A_1) \to \mor_{\calb}(F(A_0),F(A_1))$
sending $f$ to $F(f)$ is bijective. 
For the notion of a \emph{Karoubi filtration}
we refer for instance to~\cite{Cardenas-Pedersen(1997)}.

\begin{theorem}
\label{thm:swindle_plus_karoubi-for-K_plus_L}
$ $
\begin{enumerate}
\item \label{thm:swindle_plus_karoubi:swindle}
      If $\cala$ is a flasque additive category,
      then $\bfK(\cala)$ is weakly contractible.
      If $\cala$ is a flasque additive category with involution,
      then $\bfLi(\cala)$ is weakly contractible;
\item \label{thm:swindle_plus_karoubi:karoubi}
      If $\cala \subseteq \calu$ is a Karoubi filtration
      of additive categories, then
      \begin{equation*}
        \bfK(\cala) \to \bfK(\calu) \to \bfK(\calu / \cala)
      \end{equation*}
      is a homotopy fibration sequence of spectra.
      If $\cala \subseteq \calu$ is a Karoubi filtration
      of additive categories with involutions, then
      \begin{equation*}
        \bfLi(\cala) \to \bfLi(\calu) \to \bfLi(\calu / \cala)
      \end{equation*}
      is a homotopy fibration sequence of spectra;
\item \label{thm:swindle_plus_karoubi:equivalence}
      If $\varphi \colon \cala \to \calb$ is an
      equivalence of additive categories,
      then $\bfK(\varphi)$ is a weak equivalence of spectra.
      If $\varphi \colon \cala \to \calb$ is an
      equivalence of additive categories with involution,
      then $\bfLi(\varphi)$ is a weak equivalence of spectra;
\item \label{thm:swindle_plus_karoubi:colim}
      If $\cala = \colim_i \cala_i$ is a colimit of additive
      categories over a directed system, then the natural map
      $\colim_i \bfK(\cala_i) \to \bfK(\cala)$ is a weak
      equivalence.
      If $\cala = \colim_i \cala_i$ is a colimit of additive
      categories with involution over a directed system,
      then the natural map
      $\colim_i \bfLi(\cala_i) \to \bfLi(\cala)$ is a weak
      equivalence.
\end{enumerate}
\end{theorem}

\begin{proof}\ref{thm:swindle_plus_karoubi:swindle} 
This is the well-known \emph{Eilenberg-swindle}. 
See for instance~\cite[Lemma~4.12]{Carlsson-Pedersen(1995a)}.
\\[1mm]\ref{thm:swindle_plus_karoubi:karoubi}
See~\cite{Cardenas-Pedersen(1997)} 
and~\cite[Theorem~4.2]{Carlsson-Pedersen(1995a)}.
\\[1mm]\ref{thm:swindle_plus_karoubi:equivalence} 
See for instance~\cite[Lemma~4.17]{Carlsson-Pedersen(1995a)}.
\\[1mm]\ref{thm:swindle_plus_karoubi:colim} For $K$-theory this
follows from~\cite[(12) on page~20]{Quillen(1973)}. 
The proof  for $L$-theory
in~\cite[Lemma~5.2]{Bartels-Echterhoff-Lueck(2008colim)} for rings
carries over to additive categories.
\end{proof}

Many $K$-theory results in controlled algebra
depend only on the properties of $K$-theory listed
in Theorem~\ref{thm:swindle_plus_karoubi-for-K_plus_L}
and there are therefore corresponding results in
$L$-theory.
This applies in particular to
Proposition~3.8 and Theorem~7.2 in~\cite{Bartels-Lueck-Reich(2008hyper)}.
In the following we give minor variations of these results.

\begin{theorem}
\label{thm:obstruction-category}
Let $G$ be a group.
Let $\calf$ be a family of subgroups of $G$.
\begin{enumerate}
\item \label{thm:obstruction-category:K}
      Suppose that there is $m_0 \in \IZ$ such that
      \begin{equation*}
           K_{m_0}(\calo^G(E_\calf G;\cala)) = 0
      \end{equation*}
      holds for all additive $G$-categories $\cala$.

      Then the assembly map~\eqref{eq:assembly-map-calf-K}
      is an isomorphism for $m < m_0$ and surjective  for $m=m_0$
      for all such $\cala$.
\item \label{thm:obstruction-category:L}
      Suppose that there is $m_0 \in \IZ$ such that
      \begin{equation*}
           \Li_{m_0}(\calo^G(\EGF{G}{\calf_2};\cala)) = 0
      \end{equation*}
      holds for all additive $G$-categories $\cala$ with involution.

      Then the assembly map~\eqref{eq:assembly-map-calf-2-L}
      is an isomorphism for all $m$ and such $\cala$.
\end{enumerate}
\end{theorem}

\begin{proof}
For $K$-theory the statement is 
almost~\cite[Proposition~3.8]{Bartels-Lueck-Reich(2008hyper)}.
The only difference to the present statement is that
in the above reference the vanishing of the $K$-group is assumed
for all $m \geq m_0$ and the conclusion is an isomorphism for
all $m$.
A straightforward modification of the proof 
from~\cite{Bartels-Lueck-Reich(2008hyper)} yields the proof
of our present $K$-theory statement.
This proof uses in fact only the properties of
$K$-theory listed in Theorem~\ref{thm:swindle_plus_karoubi-for-K_plus_L}
and carries therefore over to $L$-theory.
Because $L$-theory is periodic we get in this case the stronger
statement stated above.
\end{proof}

In order to formulate the next result, we quickly recall that for an 
additive category $\calb$, elements of $K_1(\calb)$ can be thought of 
as self-chain homotopy equivalences over $\calb$.
If $f\colon C \to C$ is a  self-chain homotopy
equivalence of a finite chain complex over $\cala$, 
then the \emph{self-torsion} is an element
$$[(C,f)] \in K_1(\calb).$$
It depends only on the chain homotopy class of $f$.
If $f \colon C \to D$ and $g \colon D \to C$ are chain homotopy equivalences
of finite $\calb$-chain complexes, then we obtain $[(C,g \circ f)]~=[(D,f \circ g)]$ 
In particular we get $[(C,f)] = [(D,g)]$ for  self-chain
homotopy equivalences of finite $\calb$-chain
complexes $f \colon C \to C$ and $g \colon D \to D$ provided
that there is a chain homotopy equivalence $u \colon C \to D$ with
$u \circ f \simeq g \circ u$. 
If $v \colon  B\to B$ is an automorphism in $\calb$
and $0(v) \colon 0(B) \to 0(B)$ is the obvious automorphism of the 
$\calb$-chain complex $0(B)$ which is concentrated in dimension $0$ and 
given there by $B$, then the class $[v]$ in $K_1(\cala)$ coming from the 
definition of $K_1(\cala)$ agrees with the self-torsion $[0(v)]$ 
(see~\cite{Gersten(1967)}, \cite[Example~12.17 on page~246]{Lueck(1989)}, 
\cite{Ranicki(1985_torsI)}).

In the following theorem $d^1$ denotes the $l^1$-metric on simplicial
complexes, compare~\cite[Subsection~4.2]{Bartels-Lueck-Reich(2008hyper)}.

\begin{theorem}
\label{thm:stability-element}
Let $N \in \IN$.
Let $\calf$ be a family of subgroups of a group $G$.
Let $S$ be a finite subset of $G$.
\begin{enumerate}
\item \label{thm:stability-element:K}
      Let $\cala$ be an additive $G$-category.
      Then there exists a positive real number $\e~=~\e(N,\cala,G,\calf,S)$ 
      with the following property: if
      $\Sigma$ is a simplicial complex of dimension $\leq N$
      equipped with a simplicial action of $G$
      all whose isotropy groups are members of $\calf$
      and $\alpha \colon C \to C$ is an
      $(\e,S)$-chain homotopy equivalence over
      $\calo^G(E_\calf G,\Sigma,d^1;\cala)$
      where $C$ is concentrated in degrees $0,\dots,N$, then
      \begin{equation*}
        [(C,\alpha)] = 0 \in K_1(\calo^G(E_\calf G,\Sigma,d^1;\cala)).
      \end{equation*}
\item \label{thm:stability-element:L}
      Let $\cala$ be an additive $G$-category with involution.
      Then there exists a positive real number 
      $\e = \e(N,\cala,G,\calf,S)$ with the following property: if
      $\Sigma$ is a simplicial complex of dimension $\leq N$
      equipped with a simplicial action of $G$
      all whose isotropy groups are members of $\calf$
      and $(C,\psi)$ is a $0$-dimensional
      ultra-quadratic $(\e,S)$-Poincar{\'e} complex
      over $\Idem(\calo^G(E_\calf G,\Sigma,d^1;\cala))$
      concentrated in degrees $-N,\dots,N$, then
      \begin{equation*}
        [(C,\psi)] = 0 \in \Li_0(\calo^G(E_\calf G,\Sigma,d^1;\cala)).
      \end{equation*}
\end{enumerate}
\end{theorem}

\begin{proof}
The $K$-theory statement can be deduced 
from~\cite[Theorem~7.2]{Bartels-Lueck-Reich(2008hyper)}
in roughly the same way as~\cite[Corollary~4.6]{Bartels(2003a)}
is deduced from~\cite[Proposition~4.1]{Bartels(2003a)}.

For the convenience of the reader we give more details.
We will proceed by contradiction and assume that there is no such 
$\e = \e(N,\cala,G,\calf,S)$.
Then for every $n \in \IN$ there are 
\begin{itemize}
\item a simplicial complex $\Sigma_n$ of dimension $\leq N$
      equipped with a simplicial action of $G$
      all whose isotropy groups are members of $\calf$; 
\item an $(1/n,S)$-chain homotopy equivalence 
      $\alpha^n \colon C^n \to C^n$
      over the additive category $\calo^G(E_\calf G,\Sigma_n,d^1;\cala)$
      where $C^n$ is concentrated in degrees $0,\dots,N$,
\end{itemize}
such that 
\begin{equation*}
   [(C^n,\alpha^n)] \neq 0 \in K_1(\calo^G(E_\calf G, \Sigma_n,d^1;\cala)).
\end{equation*}
Now consider the product category 
\begin{equation*}
  \prod_{n \in \IN} \calo^G(E_\calf G, \Sigma_n,d^1;\cala).
\end{equation*}
Objects of this category are given by sequences $(M_n)_{n \in \IN}$ where each
$M_n$ is an object in $\calo^G(E_\calf G, \Sigma_n,d^1;\cala)$;
morphisms $(M_n)_{n \in \IN} \to (N_n)_{n \in \IN}$ are given by sequences
$(\psi_n \colon M_n \to N_n)_{n \in \IN}$ where each $\psi_n$ is a
morphism in $\calo^G(E_\calf G, \Sigma_n,d^1;\cala)$.
We will use the subcategory $\call$ of this product category that has the 
same objects as the product category but has fewer morphisms:
a morphism $(\psi_n)_{n \in \IN}$ is a morphism in $\call$ is and only if
there are a finite subset $T \subset G$ and a number $A > 0$,
such that $\psi^n$ is $(A/n,T)$-controlled for each $n \in \IN$.
Observe that 
$(\alpha^n)_{n \in \IN} \colon (C^n)_{n \in \IN} \to (C^n)_{n \in \IN}$ 
is a chain homotopy equivalence
in this category $\call$.
We denote by $[(C^n,\alpha^n)_{n \in \IN}] \in K_1(\call)$ its 
$K$-theory class.   
Let 
\begin{equation*}
  \call_{\oplus} := \bigoplus_{n \in \IN} \calo^G(E_\calf G, \Sigma_n,d^1;\cala).
\end{equation*}
This is in a canonical way a subcategory of $\call$ and it is proven 
in~\cite[Theorem~7.2]{Bartels-Lueck-Reich(2008hyper)} that
this inclusion $\iota \colon \call_{\oplus} \to \call$ 
induces an isomorphism in $K$-theory.
Consider an element $a \in K_1(\call_\oplus)$ satisfying 
$\iota_*(a) = [(C^n,\alpha^n)_{n \in \IN}] \in K_1(\call)$.
Denote by $p_n \colon \call \to \calo^G(E_\calf G, \Sigma_n,d^1;\cala)$
the canonical projection.
The definition of $\call_{\oplus}$ as a direct sum implies that there is
$m_0 \in \IN$ such that $(p_m \circ \iota)_* (a) = 0$ for all $m \geq m_0$.
Thus we obtain the desired contradiction
\begin{equation*}
  [(C^m,\alpha^m)] = (p_m)_*([(C^n,\alpha^n)_{n \in \IN}]) =
     (p_m \circ \iota)_* (a) = 0
\end{equation*}
for $m \geq m_0$.

In~\cite[Theorem~7.2]{Bartels-Lueck-Reich(2008hyper)} it is assumed
that the action of $G$ on $\Sigma_n$ is in addition cell preserving.
This assumption makes no real difference to our result here: we can always
replace $\Sigma$ by it first barycentric subdivision and obtain a cell
preserving action.
The subdivision changes the metric only in a uniformly controlled way.
(On the other hand, the proof 
of~\cite[Theorem~7.2]{Bartels-Lueck-Reich(2008hyper)}
does not really use the assumption cell preserving.)

This proof carries over to $L$-theory in a straightforward fashion,
because it only depends on the properties of $K$-theory that are listed
in Theorem~\ref{thm:swindle_plus_karoubi-for-K_plus_L} and also hold 
in $L$-theory.
\end{proof}


\typeout{------------------  Transfer up to homotopy  -----------------------}

\section{Transfer up to homotopy}
\label{sec:transfer-up-to-homotopy}

\begin{summary*}
  In this section we transfer morphisms $\psi$ in $\calc^G(Y;\cala)$
  to chain maps $\tr^{\bfP} \psi$ over $\calc^G(Y,Z;\cala)$.
  This transfer depends on the choice of a chain complex $\bfP$ over
  $\calc(Z;\IZ)$ equipped with a homotopy action, 
  see Definition~\ref{def:homotopy-chain-cx}.
  It is functorial up to homotopy, 
  see Lemma~\ref{lem:homotopy-functorial}.  
\end{summary*}

Throughout this section we fix the following convention.

\begin{convention}
Let
\begin{itemize}
\item $G$ be a group;
\item $Y$ be a $G$-space;
\item $(Z,d)$ be a compact metric space; 
\item $\cala$ be an additive $G$-category.
\end{itemize}
We will use the following $G$-actions:
$g \in G$ acts trivially on $Z$, on $G \times Y \times
[1,\infty)$ by $g \cdot (h,y,t) = (gh,gy,t)$ and on $G \times Z \times
Y \times [1,\infty)$ by $g \cdot (h,z,y,t) = (gh,z,gy,t)$.
\end{convention}

We will use the following chain complex analogue of homotopy $S$-actions.

\begin{definition}[Chain homotopy $S$-action]
\label{def:homotopy-chain-cx}
Let $S$ be a finite subset of $G$  (containing $e$).
\begin{enumerate}
\item \label{def:homotopy-chain-cx:objects}
      Let $P$ be a chain complex over $\Idem (\calc(Z,d;\IZ))$.
      A \emph{homotopy $S$-action} on $P$ consists of
      chain maps $\varphi_g \colon P \to P$
      for $g \in S$ and chain homotopies $H_{g,h}$
      for $g,h,\in S$ with $gh \in S$ from
      $\varphi_g \circ \varphi_h$ to $\varphi_{gh}$.
      Moreover, we require $\varphi_e = \id$ and $H_{e,e} = 0$.
      In this situation we will also say that
      $(P,\varphi,H)$ is a \emph{homotopy $S$-chain complex over
      $\Idem (\calc(Z,d;\IZ))$};
\item \label{def:homotopy-chain-cx:maps}
      Let $\bfP = (P,\varphi^P,H^P)$ and
      $\bfQ = (Q,\varphi^Q,H^Q)$
      be homotopy $S$-chain complexes over $\Idem (\calc(Z,d;\IZ))$.
      A homotopy $S$-chain map $\bfP \to \bfQ$
      is a chain map $f \colon P \to Q$
      such that $f \circ \varphi^P_g$ and
      $\varphi^Q_g \circ f$ are chain homotopic for all $g \in S$.
      It is called a homotopy $S$-chain equivalence, if
      $f$ is in addition a chain homotopy equivalence;
\item \label{def:homotopy-chain-cx:trivial}
      Let $z_0 \in Z$.
      The {\em trivial homotopy $S$-chain complex
      $\calc(Z,d;\IZ))$ at $z_0$},
      which we will denote by
      $\bfT = (T,\varphi^T,H^T)$, is defined by
      $(T_0)_{z_0} = \IZ$, $(T_n)_{z} = 0$ unless $n=0$, $z=z_0$,
      $\varphi^T_a = \id_{T}$ and $H^T_{a,b} = 0$.
\end{enumerate}
\end{definition}

Let $\calf(\IZ)$ denote our choice of a small model for the category
of finitely generated free $\IZ$-modules (see
Subsection~\ref{subsec:metric-control}).  
Recall from Subsection~\ref{subsec:additive-G-categories}
that $\cala$ comes with a strictly associative
functorial sum $\oplus$. 
We define a functor of additive
categories called the \emph{tensor product functor}
$$\otimes \colon \cala \times \calf(\IZ) \to \cala$$
as follows. On objects put $A \otimes \IZ^n = \bigoplus_{i=1}^n A$.  Given
a morphism $f \colon A \to B$ in $\cala$ and a morphism $U\colon \IZ^m
\to \IZ^n$ defined by a matrix $U = (u_{i,j})$, let
$$f \otimes U \colon A \otimes \IZ^m \to B \otimes \IZ^n$$
be the morphism $\bigoplus_{i=1}^m A \to \bigoplus_{j=1}^n B$ which is
given by the matrix $(u_{i,j} \cdot f)$ of morphisms in $\cala$.  
This construction is functorial in $\cala$.
For objects $M \in \calo^G(Y;\cala) = \calc^G(G \times Y \times [1,\infty);
\cale(Y),\calf(Y);\cala)$ and $F \in \calc(Z,d;\IZ)$ we 
define
$$M \otimes F \in \calo^G(Y, Z,d;\cala) =
\calc^G(G \times Z \times Y \times [1,\infty);
                   \cale(Y,Z,d),\calf(Y,Z,d);\cala)$$
by putting
\begin{equation*}
  (M \otimes F)_{g,z,y,t} := M_{g,y,t} \otimes F_{z}.
\end{equation*}
This construction is clearly
functorial in $F$ and $M$.  It is easy to check that also the control
conditions $\cale(Y,Z,d)$ are satisfied because they are implemented
by projections to one of the spaces $Z$, $Y\times [1,\infty) $ or $G$.  
Thus we obtain a tensor product functor
\begin{equation}
  \label{eq:calc-ox-calo-to-calo}
  \otimes \colon \calo^G(Y;\cala) \otimes \calc(Z,d;\IZ) \to 
                                           \calo^G(Y,Z,d;\cala).
\end{equation}
This functor can in particular be applied to an object $M \in
\calo^G(Y;\cala)$ and a chain complex $P$ over $\Idem( \calc(Z,d;\IZ)
)$ to produce a chain complex $M \otimes P$ over $\Idem( \calo^G(Y,
Z,d;\cala) )$.

Next we will consider homotopy
$S$-actions on $P$ to twist the functoriality in $M$. Let $S$ be a
finite subset of $G$ and $\bfP = (P,\varphi^P,H^P)$ be a homotopy
$S$-chain complex over $\Idem (\calc(Z,d;\IZ))$.
For an $S$-morphism $\psi \colon M \to N$ in $\calo^G(Y;\cala)$ 
we define a chain map 
$\tr^\bfP \psi \colon M \otimes P \to N \otimes P$ by putting
\begin{equation*}
  (\tr^\bfP \psi)_{(g,z,y,t),(g',z',y',t')} :=
        \psi_{(g,y,t),(g',y',t')} \otimes \left(\varphi^P_{g^{-1}g'}\right)_{z,z'}.
\end{equation*}
If we write $\psi = \sum_{a \in S} \psi_a$ as in
\eqref{eq:decomposition-of-S-mor}
then $\tr^\bfP \psi = \sum_{a \in S} \psi_a \otimes \varphi^P_a$.
This is not strictly functorial in $M$,
see Lemma~\ref{lem:homotopy-functorial} below.
(The definition of $\tr^\bfP \psi$ is very much in the spirit of the classical 
$K$-theory transfer, compare Section~\ref{subsec:Whitehead-transfer}
and in particular~\eqref{eq:a_otimes_tC}.)

Let $f \colon \bfP \to \bfQ$ be a map of homotopy $S$-chain
complexes over $\calc(Z,d;\IZ)$, where $\bfQ = (Q,\varphi^Q,H^Q)$. It induces a
chain map $\id_M \otimes f \colon  M \otimes P \to M \otimes Q$ over
$\calo^G(Y, Z, d;\cala)$. If $f \colon \bfP \to \bfQ$ is a
homotopy $S$-chain equivalence over $\calc(Z,d;\IZ)$, then $\id_M
\otimes f$ is a chain homotopy equivalence over  $\calo^G(Y, Z, d;\cala)$.
If $\psi \colon M \to N$ is an $S$-morphism, 
then $(\id_N \otimes f) \circ \tr^\bfP \psi $ and 
$\tr^\bfQ \psi \circ (\id_M \otimes f)$ are
homotopic as chain maps over  $\calo^G(Y, Z,d;\cala)$.

\begin{lemma}
\label{lem:homotopy-functorial}
Let $S$ be finite subset of $G$ (containing $e$) and
$\bfP = (P,\varphi,H)$ be a homotopy $S$-chain complex over $\calc(Z,d;\IZ)$.
Let $T$ be a subset of $S$ (also containing $e$) such that
$a,b \in T$ implies $ab \in S$.
Let $\psi = \sum_{a \in T} \psi_a \colon M \to M'$,
$\psi' = \sum_{a \in T} \psi'_a \colon M' \to M''$
be $T$-morphisms in $\calc^G(Y,G;\cala)$,
where $\psi_a$ and $\psi'_a$ are $\{ a \}$-morphisms.
Then
\begin{equation*}
  \sum_{a,b \in T} (\psi'_a \circ \psi_b) \otimes H_{a,b}
\end{equation*}
is a chain homotopy over $\calo^G(Y,Z,d;\cala)$
from $\tr^\bfP \psi' \circ \tr^\bfP \psi$ to
$\tr^\bfP (\psi' \circ \psi)$.
\end{lemma}

\begin{proof}
This is a straightforward calculation.
\end{proof}


\typeout{--------  The transfer in K-theory  ---------}

\section{The transfer in $K$-theory}
\label{sec:The_transfer_in_K-theory}

\begin{summary*}
  In this section we construct a controlled $K$-theory transfer.
  The construction from the previous section is applied
  to lift a given automorphism $\alpha$ in the obstruction category
  $\calo^G(Y;\cala)$ to chain homotopy selfequivalences $\hat \alpha$ over
  the idempotent completion of $\calo^G(Y,G \x X,d_{S,\Lambda};\cala)$.
  The contractibility of $X$ is used to show that
  this transfer lifts $K$-theory elements.
  Important are in addition the control properties of $\hat \alpha$.
  This construction is a  
  variation of~\cite[Section~12]{Bartels-Farrell-Jones-Reich(2004)}. 
  A review of the classical (uncontrolled) $K$-theory transfer can be
  found in Subsection~\ref{subsec:Whitehead-transfer} of the Appendix.
\end{summary*}

Throughout this section we fix the following convention.

\begin{convention}
\label{conv:transfer-in-K}
Let
\begin{itemize}
\item $G$ be a group;
\item $N \in \IN$;
\item $(X,d) = (X,d_X)$ be a compact contractible controlled 
      $N$-dominated metric space;
\item $Y$ be a $G$-space;
\item $\cala$ be an additive $G$-category.
\end{itemize}
\end{convention}

\begin{proposition}
\label{prop:projective-chain-complex}
Let $S$ be a finite subset of $G$ (containing $e$)
and $(\varphi,H)$ be a homotopy $S$-action on $X$.
For every $\e > 0$ there exists
a homotopy $S$-chain complex $\bfP = (P,\varphi^P,H^P)$
over $\Idem( \calc(X,d;\IZ))$ satisfying:
\begin{enumerate}
\item \label{prop:projective-chain-complex:length}
      $P$ is concentrated in degrees $0,\dots,N$;
\item \label{prop:projective-chain-complex:e-controlled}
      $P$ is $\e$-controlled;
\item \label{prop:projective-chain-complex:contractible}
      there is a homotopy $S$-chain equivalence
      $f \colon \bfP \to \bfT_{x_0}$
      to the trivial homotopy $S$-chain complex at $x_0 \in X$ for some
      (and hence all) $x_0 \in X$;
\item \label{prop:projective-chain-complex:control-varphi}
      if $g \in S$ and $(x,y) \in \supp \varphi^P_g$, then
      $d(x,\varphi_g(y)) \leq \e$;
\item \label{prop:projective-chain-complex:control-H}
      if $g,h\in S$ with $gh \in S$ and $(x,y) \in \supp H^P_{g,h}$,
      then there is $t \in [0,1]$   such that
      $d(x,H_{g,h}(y,t)) \leq \e$.
\end{enumerate}
\end{proposition}

The idea of the proof of Proposition~\ref{prop:projective-chain-complex}
is not complicated:
Consider the subcomplex $C^{\sing,\e}(X)$ of the singular chain complex of 
$X$ spanned by singular simplices of diameter bounded by an
appropriate small constant.
This chain complex is in an $\e$-controlled way finitely dominated, because 
$X$ is controlled $N$-dominated, and can therefore up to controlled 
homotopy be replaced by finite projective chain complex $P$.
The homotopy $S$-action  on $X$ induces through this homotopy equivalence
the chain homotopy $S$-action on $P$.
The details of this proof are somewhat involved and postponed to the next 
section.

\begin{proposition}
\label{prop:K-theory-transfer-lift}
Let $T \subseteq S$ be finite subsets of $G$ (both containing $e$)
such that for $g,h \in T$, we have $gh \in S$.
Let $\alpha \colon M \to M$ be a $T$-automorphism in $\calo^G(Y;\cala)$.
Let $\Lambda > 0$.
Then there is an $(S,2)$-chain homotopy equivalence
$\hat \alpha \colon C \to C$ over 
$\Idem(\calo^G(Y, G \times X,d_{S,\Lambda};\cala))$
where $C$ is concentrated in degrees $0,\dots,N$, such that
\begin{equation*}
      [p(C,\hat \alpha)] = [(M,\alpha)] \in
                K_1(\Idem(\calo^G(Y;\cala)))
\end{equation*}
where  $p \colon \Idem(\calo^G(Y, G \times X,d_{S,\Lambda};\cala))
             \to \Idem(\calo^G(Y;\cala))$
is the functor induced by the projection $ G \times X \to \pt$.
\end{proposition}

\begin{proof}
Let $\e := 1 / \Lambda$.
Let $\bfP = (P,\varphi^P,H^P)$ be a homotopy $S$-chain complex over
$\Idem (\calc(X,d;\IZ))$
that satisfies the assertion of Proposition~\ref{prop:projective-chain-complex}.
It follows from Lemma~\ref{lem:homotopy-functorial} that
$\tr^\bfP(\alpha) \colon M \otimes P \to M \otimes P$
is an $S$-chain homotopy equivalence over
$\calo^G(Y, X,d;\cala)$.

Let $f \colon \bfP \to \bfT_{x_0}$ be the weak equivalence from
assertion~\ref{prop:projective-chain-complex:contractible} in
Lemma~\ref{prop:projective-chain-complex}. 
Let $q \colon \calo^G(Y,X,d;\cala) \to \calo^G(Y;\cala)$ 
be the functor induced by $X \to \pt$. 
Then $\id_M \otimes f$ is a chain homotopy equivalence, 
$\tr^P \alpha \circ (\id_M \otimes f)$ is chain homotopic to 
$(\id_M \otimes f) \circ \tr^{\bfT_{x_0}} \alpha$ and 
$q(\tr^{\bfT_{x_0}} \alpha) = \alpha$ 
(up to a canonical isomorphism $q( M \otimes T)  \cong M$).
Therefore $q[(M \otimes P, \tr^\bfP \alpha)]
                          = q[(M \otimes T, \tr^{\bfT_{x_0}} \alpha)]
                          = [(M, \alpha) ] \in K_1(\Idem(\calo^G(Y;\cala)))$.
Let 
$F \colon \calo^G(Y,X,d;\cala) \to \calo^G(Y, G \times X, d_{S,\Lambda};\cala)$
be the functor induced by the map
$(g,x,y,t) \mapsto (g,g,x,y,t)$ and set
$(C,\hat \alpha) := F (M \otimes P,\tr^\bfP \alpha)$.
Since $p \circ F = q$ we have $[p(C,\hat \alpha)] = [\alpha]$.
That $\hat \alpha$ is a $(S,2)$-chain homotopy equivalence follows from our 
choice of $\e$, Definition~\ref{def:d_S_lambda} of the metric $d_{S,\Lambda}$
and the concrete formula for $\tr^\bfP(\alpha)$. The key observation
is that for $t \in T$ and $(x,y) \in \supp(\varphi^P_t)$ we have
$$d_{S,\Lambda}((e,x),(t,y)) 
\le  1 + \Lambda \cdot d_X(x,\phi_g(x)) 
\le 1 + \Lambda \cdot \epsilon = 2.$$
\end{proof}


\typeout{----- Proof of prop:projective-chain-complex -----}

\section{Proof of Proposition~\ref{prop:projective-chain-complex}}
\label{sec:Proof_of_Prop}

Throughout this section we use Convention~\ref{conv:transfer-in-K}.

Let $Z$ be a metric space. If $q \colon B \to Z$ is a
map, then a homotopy $H \colon A \x [0,1] \to B$ is called
$\e$-controlled over $q$ if for every $a \in A$,
the set $\{ q(H(a,t)) \; | \; t \in [0,1] \}$ has diameter 
at most $\e$ in $Z$. 
The following lemma shows that we can replace the $CW$-complexes
appearing in the definition of controlled $N$-dominated metric spaces 
by simplicial complexes.

\begin{lemma}
  \label{lem:controlled-simplicial-for-CW}
  Let $q \colon K \to Z$ be a map from an $N$-dimensional 
  finite $CW$-complex to a metric space. 
  Let $\e > 0$ be given.
  Then there is an $N$-dimensional finite simplicial complex $L$,
  maps $i \colon K \to L$, $p \colon L \to K$
  and a homotopy $h \colon p \circ i \simeq \id_K$
  that is $\e$-controlled over $q$.     
\end{lemma}

\begin{proof}
  We proceed by induction over the skeleta of $K$.
  For $K^{(0)}$ the claim is obviously true.
  Assume for the induction step that there is $0 < \delta < \e$,
  a finite \mbox{$n$-dimensional} simplicial complex $L$,
  maps $i \colon K^{(n)} \to L$, $p \colon L \to K^{(n)}$
  and a homotopy $h \colon p \circ i \simeq \id_{K^{(n)}}$
  that is $\delta$-controlled over $q$.
  For a given $\delta'$ with $\delta < \delta' < \e$,
  we will construct a finite $(n+1)$-dimensional simplicial
  complex $L'$, maps $i' \colon K^{(n+1)} \to L'$, $p' \colon L \to K^{(n+1)}$
  and a homotopy $h' \colon p' \circ i' \simeq \id_{K^{(n+1)}}$
  that is $\delta'$-controlled over $q$.
  
  Let $\varphi \colon \amalg_I S^n \to K^{(n)}$ for some finite index set $I$
  be the attaching map of the $(n+1)$-skeleton, i.e.,
  $K^{(n+1)} = D^{n+1} \cup_\varphi \amalg_I K^{(n)}$.
  Pick $\alpha > 0$ such that $\delta + \alpha < \delta'$.
  By subdividing $L$ we can assume that the image
  under $q \circ p$ of each
  simplex in $L$ has diameter at most $\alpha$ in $Z$.
  Let $\psi$ be a simplicial approximation of
  $i \circ \varphi$, i.e., $\amalg_I S^n$ is  a simplicial complex and
  $\psi$ is a simplicial map such that for any $x \in \amalg_I S^n$ the point 
  $\psi(x)$ is contained in the
  smallest simplex of $L$ that contains $i(\varphi(x))$.
  In particular, there is a homotopy 
  $k \colon \psi \simeq i \circ \varphi$
  that is $\alpha$-controlled over $q \circ p$.
  Define $L' := \amalg_I D^{n+1} \cup_\psi L$. Since $\psi$ is a simplicial
  map and $L'$ is the mapping cone of $\psi$, 
  we can extend the simplicial structure of $L$ to a simplicial structure
  on $L'$. Pick $\alpha' > 0$ such that $\delta + \alpha + \alpha' < \delta'$
  and choose $\beta > 0$ such that for any $x \in \amalg_I S^n$ 
  the diameter of $\{ q( rx ) \; | \; r \in [1-2\beta,1] \}$
  is at most $\alpha'$.
  In order to extend $i$ to a map $i' \colon K^{(n+1)} \to L'$ it suffices
  to specify a map $\amalg_I D^{n+1} \to L'$ such that its restriction to the boundary
  is $i \circ \varphi$.  We use polar coordinates on $\amalg_I D^{(n+1)}$.
  Define the desired extension $i'$ by setting
  \[
  i'(rx) := 
    \left\{ \begin{array}{lll}
            rx                      & \mbox{if} & r \in [0,1-2\beta]; \\  
            (2r -1 + 2\beta)x       & \mbox{if} & r \in [1-2\beta,1-\beta]; \\
            k(x,(r - 1 + \beta)/\beta) & \mbox{if} & r \in [1-\beta,1],
           \end{array} \right.
  \]
  for $r \in [0,1]$ and $x \in \amalg_I S^n$, 
  where $rx$ and $(2r -1 + 2\beta)x$ 
  are understood to be the images of these points in
  $\amalg_I D^{n+1}$ under the canonical map 
  $\amalg_I D^{n+1} \to L'   :=  \amalg_I D^{n+1} \cup_\psi L$.  
  Notice that the map $\amalg_I D^{n+1} \to L'$
  above is the identity on $\amalg_I D^{n+1}$ except for a neighborhood of
  the boundary, where we use the homotopy $k$, and this neighborhood is 
  smaller the smaller $\beta$ is.

  By a similar formula we extend $p$ to a map 
  $p' \colon L' \to K^{(n+1)}$, where we use the homotopy 
  $(h^- \circ \varphi \x \id_{[0,1]}) * (p \circ k^-)  
   \colon \varphi \simeq p \circ \psi$
  in place of $k$. 
  (Here $*$ denotes concatenation of homotopies and 
  $h^-$ and $k^-$ are $h$ and $k$ run backwards.)
  Then $p' \circ i'$ is an extension of $p \circ i$ such that
  \[
  (p' \circ i') (rx) =
      \left\{ \begin{array}{lll}
            rx        & \mbox{if} & r \in [0,1-2\beta]; \\
            (4r - 3 + 6\beta)x
                      & \mbox{if} & r \in [1-2\beta,1-3\beta/2]; \\
            h^-(\varphi(x),(4r - 4 + 6\beta)/\beta)
                      & \mbox{if} & r \in [1-3\beta/2,1-5\beta/4]; \\
            (p \circ k^-) (x, (4r -4 +5\beta)/\beta)          
                      & \mbox{if} & r \in [1-5\beta/4,1-\beta)]; \\
            (p \circ k) (x,(r - 1 + \beta)/\beta) 
                      & \mbox{if} & r \in [1-\beta,1],
           \end{array} \right.
  \]
  for $r \in [0,1]$ and $x \in \amalg_I S^n$.
  
  In the next step we cancel $p \circ k$ and $p \circ \bar{k}$ 
  appearing above:
  the constant homotopy $(y,t) \mapsto (p \circ i)(y)$ on $K$
  has a canonical extension to a homotopy $h'_0$ from
  $p' \circ i'$ to an extension $f$ of $p \circ i$
  such that
  \[
  f (rx) =
      \left\{ \begin{array}{lll}
            rx        & \mbox{if} & r \in [0,1-2\beta]; \\  
            (4r - 3 + 6\beta)x
                      & \mbox{if} & r \in [1-2\beta,1-3\beta/2]; \\
            h^-(x,(4r - 4 + 6\beta)/\beta)
                      & \mbox{if} & r \in [1-3\beta/2,1-5\beta/4]; \\
            (p \circ i \circ \varphi)(x) 
                      & \mbox{if} & r \in [1-5\beta/4,1],
           \end{array} \right.
  \]
  for $r \in [0,1]$ and $x \in \amalg_I S^n$.
  This homotopy $h'_0$ is $\alpha$-controlled over $q$,
  because $p \circ k$ is $\alpha$-controlled over $q$.
  
  In the final step we use the appearance of $h$
  in the above formula for $f$.
  This (and reparametrization in $r \in [1-2\beta,1]$)
  yields an extension of
  the homotopy $h$ to a homotopy
  $h'_1$ from $f$ to $\id_{K^{(n+1)}}$, and this homotopy
  is $\e + \alpha'$-controlled over $q$.
  (This comes from the control of $h$ and the control
  of reparametrizations in $r \in [1-2\beta,1]$ by our choice of $\beta$.)
  Then $h' := h'_0 * h'_1 \colon p' \circ i' \simeq \id_{K^{(n+1)}}$ is 
  $\delta'$-controlled over $q$, because 
  $\delta + \alpha + \alpha' < \delta'$.
  This finishes the construction of $L'$, $i'$, $p'$ and $h'$
  and concludes the induction step.    
\end{proof}

\begin{remark} \label{rem:iverlinecalc}
If $q \colon K \to X$ is a map from a finite simplicial complex to $X$,
then the simplicial complex $C(K)$ of $K$ is in a natural way 
(using the images of barycenters under $q$) a chain complex over
$\calc(X,d;\IZ)$.
We will also need to use the subcomplex  $C^{\sing,\e}(X)$ 
of singular chain complex of $X$ spanned
by singular simplices of diameter $\leq \e$ in $X$.
This is not naturally a chain complex over $\calc(X,d;\IZ)$,
because it fails the locally finiteness condition.
However, if we drop this conditions 
(and allow a large class of $\IZ$-modules at every point) then we
get an additive category $\overline{\calc}(X,d;\IZ)$. There is an obvious
inclusion $\calc(X,d;\IZ) \subset \overline{\calc}(X,d;\IZ)$ of
additive categories that is full and faithful on morphism sets,
Moreover, $C^{\sing,\e}(X)$ is naturally
a chain complex over $\overline{\calc}(X,d;\IZ)$. Namely, a singular simplex
$\sigma \colon\Delta \to X$ defines a  point in $X$, the image of the barycenter of
$\Delta$ under $\sigma$.
In particular, the notion of $\e$-control is defined for
maps involving $C^{\sing,\e}(X)$. It will be important to carry out certain
construction in the larger category $\overline{\calc}(X,d;\IZ)$ and then
go back to $\calc(X,d;\IZ)$. The latter step is described in the next remark.
\end{remark}

\begin{remark}
  \label{rem:Ranickis-formulas}
  Let $\calc \subset \overline{\calc}$ be an inclusion of additive categories
  that is full and faithful on morphism sets.
  Let $C$ be a chain complex  over $\overline{\calc}$ and
  $D$ be a chain complex over $\calc$, where
  $C$ is concentrated in non-negative degree and 
  $D$ is concentrated in degree $0,\dots,N$.
  Let $i \colon C \to D$, $r \colon D \to C$ be
  chain maps (over $\overline{\calc}$) and 
  let $h \colon r \circ i \simeq \id_C$ be a chain homotopy.
  In the following we recall explicit formulas from~\cite{Ranicki(1985)},
  that allow to construct from this data a chain
  complex $P$ over $\Idem(\calc)$, chain maps $f \colon C \to P$,
  $g \colon P \to C$ over $\Idem(\overline{\calc})$, and 
  chain homotopies $k \colon f \circ g \simeq \id_P$,
  $l \colon g \circ f \simeq \id_C$.
  
  Define the chain complex $C'$ over $\calc$ by 
  by defining its $m$-th chain object to be
  $$
  C_m' = \bigoplus_{j = 0}^m D_j
  $$
  and its $m$-th differential to be
  $$c'_m \colon C_m' = \bigoplus_{j= 0}^m D_j \to 
                C_{m-1}' = \bigoplus_{k = 0}^{m-1} D_k$$
  where the $(j,k)$-entry $d_{j,k} \colon D_j \to D_k$ for
  $j \in \{0,1,2 \ldots , m\}$ and $k \in \{0,1,2 \ldots ,m-1\}$ is given by
  $$d_{j,k} :=
    \begin{cases}
    0 & \text{if } j \ge k+2;
    \\
    (-1)^{m+k} \cdot d_j & \text{if } j = k+1;
    \\
    \id - r_j \circ i_j & \text{if } j = k, j \equiv m \mod 2;
    \\
    r_j \circ i_j & \text{if } j = k, j \equiv m+1 \mod 2;
    \\
    (-1)^{m+k+1} \cdot i_k \circ h_{k-1} \circ \ldots 
          \circ h_{j} \circ r_j & \text{if } j \le k-1.
    \end{cases}
  $$
  Define chain maps $f' \colon C \to C'$ and $g' \colon C'\to C$ by
  $$f'_m \colon C_m \to C_m' = D_0 \oplus D_1 \oplus \cdots \oplus D_m, 
      \quad x \mapsto (0,0, \ldots, i_m(x))$$ 
  and
  $$g'_m \colon C'_m = D_0 \oplus D_1 \oplus \cdots \oplus D_m \to C_m, 
     \quad (x_0,x_1, \cdots x_m) \mapsto 
   \sum_{j = 0}^m h_{m-1} \circ \cdots \circ h_j \circ r_j(x_j).$$ 
  We have $g'\circ f' = r \circ i$ and hence $h$ is a
  chain homotopy $g'\circ f' \simeq \id_C$. 
  We obtain a chain homotopy $k' \colon f'\circ g' \simeq \id_{C'}$ if
  $$k_m' \colon C'_m = D_0 \oplus D_1 \oplus \cdots \oplus D_m \to
      C'_{m+1} = D_0 \oplus D_1 \oplus \cdots \oplus D_m \oplus D_{m+1}$$ 
  is the obvious inclusion.

  Recall that $D$ is $N$-dimensional. 
  Thus we get $C_m' = C _N'$ for $m \ge N$ and 
  $c_{m+1}' = \id - c_m'$ for $m \ge N+1$. 
  Since $c_{m+1}' \circ c_m' = 0$ for all $m$,
  we conclude $c_m' \circ c_m' = c_m'$ for $m \ge N+1$. 
  Hence $C'$ has the form
  $$
  \dots \to C_N' \xrightarrow{c_{N+1}'} C_N' \xrightarrow{\id - c_{N+1}'}  C_N'
          \xrightarrow{c_{N+1}'}  C_N' \xrightarrow{c_N'} C'_{N-1}   
          \xrightarrow{c_{N-1}'}
    \dots \xrightarrow{c_1'} C_0'\to 0 \to  \dots.
  $$
  Define an $N$-dimensional chain complex
  $D'$ over $\Idem (\calc)$  by
  $$0 \to 0 \to (C_N',\id -c_{N+1}) \xrightarrow{c_N' \circ i} C'_{N-1}
               \xrightarrow{c_{N-1}'} \ldots \xrightarrow{c_1'} C_0' 
             \to 0 \to  \dots,$$ 
  where $i \colon (C_N',\id -c_{N+1})  \to C_N'$ 
  is the obvious morphism in $\Idem (\calc)$ which is given by 
  $\id -c_{N+1} \colon C_N' \to C_N'$. 
  Let
  $$u \colon D' \to C'$$
  be the chain map for which $u_m$ is the identity for $m \le N-1$,  $u_N$ is 
  $i \colon  (C_N',\id -c_{N+1})  \to C_N'$, and $u_m \colon 0 \to C_m$ 
  is the canonical map for $m \ge N+1$. 
  Let
  $$v \colon C' \to D'$$
  be the chain map which is given by the identity  for $m \le N-1$, by the 
  canonical projection $C'_m \to 0$ for $m \ge N+1$ and for $m = N$ by the
  morphism $C_N \to  (C_ N',\id -c_{N+1})$ defined by 
  $\id -c_{N+1} \colon C_N' \to C_N'$. 
  Obviously $v \circ u = \id_{D'}$. 
  We obtain a chain homotopy $l' \colon \id_{C'} \sim u \circ v$ if we take
  $l_m = 0$ for $m \le N$, $l_m = c_{N+1}'$ for 
  $m \ge N, m - N \equiv 0 \mod 2$ and $l_m = 1 - c_{N+1}'$ for 
  $m \ge N, m - N \equiv 1 \mod 2$.

  Define the desired chain complex $P$ by $P := D'$. 
  Define
  $$f \colon  C \to P$$
  to be the composite $v \circ f'$. Define
  $$g \colon P \to C$$
  to be the composite $g' \circ u$.
  We obtain chain homotopies
  $$k = v \circ h \circ u \colon f \circ g \simeq \id_P$$
  and
  $$l = h - g' \circ l' \circ f' \colon g \circ f \simeq \id_C.$$
\end{remark}

\begin{lemma} \label{lem:chain-cxs-D-over-X}
  Let $\e > 0$ be given. 
  Then there is an $N$-dimensional 
  $\e$-controlled chain complex $D$ over $\calc(X,d;\IZ)$,
  $\e$-controlled chain maps $i \colon C^{\sing \e}(X) \to D$,
  $r \colon D \to C^{\sing,\e}(X)$ and an $\e$-controlled 
  chain homotopy $h \colon r \circ i \simeq \id_C$. 
\end{lemma}

\begin{proof}
  Since $(X,d)$ is a compact contractible $N$-dominated metric space,
  we can find a finite $CW$-complex $K$ of dimension $\le N$ and maps
  $j \colon X \to K$ and $q \colon K \to X$ and an $\e$-controlled homotopy 
  $H \colon q \circ j \simeq \id_X$. 
  Because of Lemma~\ref{lem:controlled-simplicial-for-CW} 
  we can assume without loss
  of generality that $K$ is a finite simplicial complex of dimension $\le N$.
  
  Subdividing $K$, if necessary, we can assume that the diameter
  of the images of simplices of $K$ under $q$ are at most $\e$.
  Using $q$ we consider the simplicial chain complex $C(K)$ of $K$ as
  a chain complex over $\calc(X,d;\IZ)$.
  Similarly, the subcomplex $C^{\sing,\e}(K)$ of the singular chain complex
  spanned by singular simplices in $K$ whose image under $q$ has diameter
  $\leq \e$ is a chain complex over $\overline{\calc}(X,d;\IZ)$.
  Analogously to the proof
  of~\cite[Lemma~6.9]{Bartels-Lueck-Reich(2008hyper)} one shows that
  $$
  C^{\sing , \epsilon}(X) \xrightarrow{j_*}
  C_{\ast}^{\sing, 2\epsilon} (K) \xrightarrow{q_*} C_{\ast}^{\sing,
    2\epsilon} (X)
  $$
  is well defined and that the composition is homotopic to the inclusion
  $$
  \inc_* \colon C_{\ast}^{\sing,\epsilon}(X) \to C_{\ast}^{\sing,
    2\epsilon}(X)
  $$
  by a chain homotopy that is $2 \epsilon$-controlled. 
  A slight modification of the proof
  of~\cite[Lemma~6.7~(iii)]{Bartels-Lueck-Reich(2008hyper)} shows
  that the canonical chain map
  $$a \colon C(K) \to C^{\sing,2\epsilon}(K)$$
  is a $2\e$-chain homotopy equivalence over $\calc(X,d;\IZ)$.
  Let $b \colon C^{\sing,2\epsilon}(K) \to  C(K)$ be a
  $2\epsilon$-controlled chain homotopy inverse of $a$. 
  Put$D := C(K)$.   Define
  $$i \colon C^{\sing,\epsilon}(X) \to D$$
  to be $b \circ j_*$. 
  Define
  $$r \colon D \to C^{\sing,\epsilon}(X)$$
  to be the composite of an $2\epsilon$-controlled inverse of $\inc_*$,
  $q_*$ and $a$.  
  Then $i$ resp.\ $r$ are $3 \e$ resp.\ $4 \e$-controlled
  over $(X,d)$ and there exists a chain homotopy 
  $h \colon r \circ i \simeq \id_{C}$ which is  
  $5\e$-controlled over $(X,d)$.
  This finishes the proof since $\e$ is arbitrary.
\end{proof}

\begin{lemma} \label{lem:controlled-chain-homotopy-from-map-homotopy}
  Let $\e, \delta >0$ let 
  $\varphi, \varphi' \colon X \to X$ be maps such that satisfying
  the following growth condition: if $d(x,y) \leq \e$,
  then $d(\varphi(x),\varphi(y)),d(\varphi'(x),\varphi'(y)) \leq \delta$.
  If $H \colon \varphi \simeq \varphi'$ is a homotopy,
  then there is a chain homotopy $H_* \colon \varphi_* \simeq \varphi'_*$
  over $\overline{\calc}(X,d;\IZ)$, such that
  \[
  \supp H_* \subseteq \{  (H(x,t),y) \; | \; t \in [0,1], d(x,y) \leq \e \}.
  \]
  (Here $\varphi_*,\varphi'_* \colon C^{\sing,\e}(X) \to C^{\sing,\delta}(X)$ 
  denote the induced chain maps.) 
\end{lemma}

\begin{proof} 
  The usual construction of a chain homotopy associated to a homotopy $H$ 
  uses suitable simplicial structures on $\Delta \x [0,1]$, but in general
  this yields only a chain homotopy between chain maps 
  $C^{\sing,\e}(X) \to C^{\sing,\delta}(X)$, because we do control not the 
  diameter of images of simplices in $\Delta \x [0,1]$, under 
  $H \circ (\sigma \x \id_{[0,1]})$, where $\sigma \colon \Delta^n \to X$ 
  is a singular simplex in $X$ (whose image has diameter $\leq \e$).
  This can be fixed using subdivisions.
  It is not hard to construct (by induction on $n$) for every such $\sigma$  
  a (possibly degenerate) simplicial structure $\tau_\sigma$ on 
  $\Delta^n \x [0,1]$ with the following properties:
  the image of every simplex of $\tau_\sigma$ under
  $H \circ (\sigma \x \id_{[0,1]})$ has diameter $\leq \delta$,
  $\tau_\sigma$ is natural with respect to restriction to
  faces of $\sigma$, $\tau_\sigma$ yields the standard simplicial
  structure on $\Delta^n \x \{0,1\}$.
  Degenerated simplices may appear for the following reason:
  in the induction step we need to extend a given simplicial structure
  on the boundary of $\Delta^n \x [0,1]$ to all of $\Delta^n \x [0,1]$.
  In order to arrange for the diameters of images of simplices to be small
  we may need to use barycentric subdivision and 
  this changes the given simplicial structure on the boundary.
  However, using  degenerated simplices we can interpolate 
  between a simplex and its barycentric subdivision.
\end{proof}

\begin{lemma}
  \label{lem:functions-alpha-beta}
  Let $S$ be finite subset of $G$ (containing $e$),
  and $(\varphi,H)$ be a homotopy $S$-action on $X$. 
  Then there are maps $\alpha, \beta \colon (0,\infty) \to (0,\infty)$
  such that the following holds:
  \begin{enumerate}
  \item if $d(x,y) \leq \e$, $g \in S$ 
        then $d(\varphi_g(x),\varphi_g(y)) \leq \beta(\e)$;
        if $d(x,y) \leq \e$, $g, h \in S$ with $gh\in S$ and $t \in [0,1]$,
        then $d(H_{g,h}(x,t),H_{g,h}(y,t)) \leq \beta(\e)$;
  \item \label{lem:functions-alpha-beta:lim} 
        $\lim_{\e \to 0} \beta(\e) = 0$;
  \item if $d(x,y) \leq \alpha(\e)$, then $d(\varphi_g(x),\varphi_g(y)) \leq \e$
        for all $g \in S$.
  \end{enumerate}
\end{lemma}

\begin{proof}
  This is an easy consequence of the compactness of $X$ and $X \x [0,1]$. 
\end{proof}

\begin{proof}[Proof of Proposition~\ref{prop:projective-chain-complex}]
  Consider any $\epsilon  > 0$. 
  Applying the construction from Remark~\ref{rem:Ranickis-formulas}
  to $C := C^{\sing,\e}(X)$ and $D$, $i$, $r$ and $h$ as in the
  assertion of Lemma~\ref{lem:chain-cxs-D-over-X} we
  obtain a chain complex $P$ over $\Idem(\calc(X,d;\IZ))$,
  chain maps $f \colon C \to P$, $g \colon P \to C$
  and chain homotopies $k \colon f \circ g \simeq \id_P$   
  $l \colon g \circ f \simeq \id_C$.
  By inspecting the formulas from Remark~\ref{rem:Ranickis-formulas}
  we see that $P$, $f$, $g$, $k$ and $l$
  are $(N+2)\epsilon$-controlled.  
  (Here we use that control is additive under composition
  and that the sum of $\e$-controlled maps is again $\e$-controlled.) 
  In particular 
  this takes care of assertions~\ref{prop:projective-chain-complex:length}
  and~\ref{prop:projective-chain-complex:e-controlled} since $\epsilon > 0$ 
  is arbitrary.

  Next we define the desired homotopy $S$-action on $P$.
  Let $\alpha$ be the function from Lemma~\ref{lem:functions-alpha-beta}.
  Put 
  $$\delta := \alpha(\e) \quad \text{ and} \quad 
  \gamma := \alpha(\delta) = \alpha \circ \alpha(\e).$$
  In the sequel we abbreviate $C^{\epsilon} := C^{\sing,\epsilon}(X)$,
  $C^{\delta} := C^{\sing,\delta}(X)$, and $C^{\gamma} := C^{\sing,\gamma}(X)$.
  Let $(\varphi_h)_*$ be the chain map  $C^{\gamma} \to C^{\delta}$,
  $C^{\delta} \to C^{\epsilon}$ or $C^{\gamma} \to C^{\epsilon}$ respectively
  induced by $\varphi_h \colon X \to X$.
  Let $r \colon C^{\epsilon} \to C^{\gamma}$, $r \colon C^{\epsilon} \to C^{\delta}$,
  and $r \colon C^{\delta} \to C^{\gamma}$ 
  respectively be an $\epsilon$-controlled
  chain homotopy inverse of the inclusion
  $C^{\gamma} \to C^{\epsilon}$, $C^{\delta} \to C^{\epsilon}$, and
  $C^{\gamma} \to C^{\delta}$ respectively. 
  For their existence  
  see~\cite[Lemma~6.7~(i)]{Bartels-Lueck-Reich(2008hyper)}).
  For $h \in S$ define
  $$\varphi^P_h \colon P \to P$$
  to be the composite
  $P \xrightarrow{g} C^{\epsilon} \xrightarrow{r} C^{\gamma}
  \xrightarrow{(\varphi_h)_*} C^{\epsilon} \xrightarrow{f} P$,
  if $h \not= e$ and $\varphi^P_h = \id_P$ if $h = e$.
  
  Recall that $r$, $g$ and $f$ are $\epsilon$-controlled. 
  We have $(x,y) \in \supp((\varphi_h)_*)$ if and only if
  $y = \varphi_h(x)$. Consider $(x,y) \in \supp(\varphi^P_h)$. 
  Then there exists $x_1$, $x_2$, $x_3$, $x_4$ and $x_5$ with $x = x_1$, 
  $y = x_5$,
  $(x_1,x_2) \in \supp(g)$, $(x_2,x_3) \in \supp(r)$, 
  $(x_3,x_4) \in \supp((\varphi_h)_*)$ and $(x_4,x_5) \in \supp(f)$.
  This implies $d(x,x_2) \le \epsilon$, $d(x_2,x_3) \le \epsilon$,
  $x_4 = \varphi_h(x_3)$ and $d(x_4,x_5) \le \epsilon$.  
  Using the function $\beta$ appearing in 
  Lemma~\ref{lem:functions-alpha-beta} we conclude 
  $d(\varphi_h(x),\varphi(x_4)) \le \beta(2\epsilon)$ and hence 
  $d(\varphi_h(x),y) \leq \beta(2\epsilon) + \e$. Thus we have shown
  $d(\varphi_h(x),y) \leq \beta(2\epsilon) + \e$ for 
  $(x,y) \in \supp((\varphi_h)_*)$.
  if $h \in S$ and $(x,y) \in \supp(\varphi^P_h)$. 
  Since $\epsilon > 0$ is arbitrary, and because of 
  Lemma~\ref{lem:functions-alpha-beta}~\ref{lem:functions-alpha-beta:lim}
  this takes care of 
  assertion~\ref{prop:projective-chain-complex:control-varphi}.

  Consider $h,k \in S$ with $hk \in S$. 
  Consider the following diagram of chain maps
  of chain complexes over $\Idem( \calc(X,d;\IZ))$.
  $$
  \xymatrix{P \ar[r]^-g
  & C^{\epsilon} \ar[r]^-r
  & C^{\gamma} \ar[rr]^-{(\varphi_k)_*} \ar[rrd]_-{(\varphi_k)_*}
  \ar@/_3em/[rrrrddd]_-{(\varphi_{hk})_*}
  && C^{\epsilon} \ar[rr]^-f \ar[rrd]^-{\id} \ar[d]^-r
  && P \ar[d]^-g
  \\
  & & &
  & C^{\delta} \ar[drr]^-r \ar[rrdd]_-{(\varphi_h)_*}
  && C^{\epsilon} \ar[d]^-r \ar[ll]_-r
  \\
  & & & &
  && C^{\gamma} \ar[d]^-{(\varphi_h)_*}
  \\
  & & & &
  && C^{\epsilon} \ar[d]^-f
  \\
  & & & &
  && P
  }
  $$
  The chain maps $f$, $r$ and $g$ are $\epsilon$-controlled.
  For all triangles appearing in the above diagram 
  we have explicit chain homotopies
  which make them commute up to homotopy: 
  The homotopy for the triangle involving $(\varphi_k)_*$, $(\varphi_h)_*$, 
  and $(\varphi_{hk})_*$ is induced by $H_{h,k}$, 
  see Lemma~\ref{lem:controlled-chain-homotopy-from-map-homotopy}. 
  In particular,   
  $\supp (H_{h,k})_* \subseteq \{  (H(x,t),y) \; | \; 
                           t \in [0,1], d(x,y) \leq \e \}$.
  The chain homotopy for the triangle involving
  $f$,$g$ and $\id$ is the chain homotopy $l$ which is $\epsilon$-controlled.
  The chain homotopy for the triangle involving $r$, $r$ and 
  $\id$ is the trivial one.
  The chain homotopy $K$ for the triangle involving $(\varphi_k)_*$, 
  $(\varphi_k)_*$ and $r$ comes from a $\epsilon$-controlled
  chain homotopy from the composite 
  $C^{\delta} \xrightarrow{i} C^{\epsilon} \xrightarrow{r} C^{\delta}$ 
  for $i$ the inclusion to $\id \colon  C^{\delta} \to C^{\delta}$.
  We have $\supp K \subseteq \{ (x,y) \; | \; d(\varphi(y),x) \leq \e \}$.
  We obtain $\epsilon$-controlled chain homotopies for the remaining triangles
  analogously. 
  The composite obtained by going first horizontally 
  from the left upper corner to the right
  upper corner and then vertically to the  right
  lower corner is by definition $\varphi^P_h \circ \varphi^P_k$. 
  The composite obtained by going diagonally from the left upper corner to 
  the right lower corner is by definition $\varphi^P_{hk}$. 
  Putting all these
  chain homotopies together yields a chain homotopy
  $$H^P_{h,k} \colon \varphi^P_h \circ \varphi^G_k \simeq \varphi^P_{hk}$$
  such that
  \[
  \supp H^P_{h,k} \subseteq 
      \{ (x,y) \; | \; \exists t \in [0,1] : d(x,H(t,y)) \leq \e' \}
  \]
  where $ \e' := (2 \beta(\beta(\e)) + 3\beta(\e) + \e)$.
  (We leave the verification of this precise formula to the interested 
  reader; note however that the precise formula is not
  important for us.)
  Since $\epsilon > 0$ is arbitrary and because of
  Lemma~\ref{lem:functions-alpha-beta}~\ref{lem:functions-alpha-beta:lim}, 
  this takes care of assertion~\ref{prop:projective-chain-complex:control-H}.

  It remains to deal with 
  assertion~\ref{prop:projective-chain-complex:contractible}.
  The inclusion $i \colon C^{\sing,\epsilon}(X) \to C^{\sing}(X)$ 
  is a chain homotopy equivalence  
  (see~\cite[Lemma~6.7~(i)]{Bartels-Lueck-Reich(2008hyper)}).
  Hence the composition
  $a \colon P \xrightarrow{g}  C^{\sing,\epsilon}(X) \xrightarrow{i} C^{\sing}(X)$
  is a chain homotopy equivalence.
  One easily checks that
  $C^{\sing}(\varphi_g) \circ a \simeq a \circ \varphi^P_g$ holds for all 
  $g \in G$. 
  The inclusion $\{ x_0 \} \to X$ and augmentation induce chain maps
  $j \colon T_{x_0} \to  C^{\sing}(X)$ and $q \colon C^{\sing}(X) \to T_{x_0}$. 
  Obviously $q \circ j = \id_{T_{x_0}}$. 
  Since $X$ is contractible, there is also a chain homotopy 
  from $j \circ q$ to  $\id_{C^{\sing}(X)}$. 
  Obviously $q \circ C^{\sing}(\phi_g) = \phi^T_g \circ q$ for all $g \in S$.
  Hence the composite $q \circ a \colon P \to T$ is a chain homotopy
  equivalence of homotopy $S$-chain complexes over $\Idem( \calc(X,d;\IZ))$.
  This finishes the proof of Proposition~\ref{prop:projective-chain-complex}.
\end{proof}


\typeout{--------------------  The space P_2(X)  -----------------------}

\section{The space $P_2({X})$}
\label{sec:The_space_P_2(X)}

\begin{summary*}
  In this section we introduce the space $P_2(X)$.
  As explained in the introduction, this space will be the fiber 
  for the $L$-theory transfer. 
  We also prove a number of estimates for specific metrics on 
  $P_2(X)$, $G \x P_2(X)$ and $P_2(G \x X)$.
  These will be used later to produce a contracting map defined 
  on $G \x P_2(X)$ using the contracting map defined on $G \x X$
  from Proposition~\ref{prop:contracting-maps}. 
\end{summary*}

\begin{definition}[The space $P_2({X})$]
\label{def:P_2}
Let $X$ be a space.
\begin{enumerate}
\item \label{def:P_2:P_2}
      Let $P_2(X)$ denote the space of unordered pairs of points in $X$,
      i.e., $P_2(X) = X \times X / \sim$ where $(x,y) \sim (y,x)$ for
      all $x,y \in X$.
      We will use the notation $(x:y)$ for unordered pairs.
      Note that $X \mapsto P_2(X)$ is a functor;
\item \label{def:P_2:d}
      If $d$ is a metric on $X$, then
      \begin{equation*}
         d_{P_2(X)}((x:y),(x':y')) :=
              \min \{ d(x,x') + d (y,y'), d(x,y') + d(y,x') \}
      \end{equation*}
      defines a metric on $P_2(X)$.
\end{enumerate}
\end{definition}

\begin{lemma}
\label{lem:calf_2-isotropie-for-P_2}
Let $\calf$ be a family of subgroups of a group $G$.
Denote by $\calf_2$ the family of subgroups of $G$ which are contained in
$\calf$ or contain a member of $\calf$ as subgroup of index two.
Let $G$ act on a space $X$ such that all isotropy groups belong to $\calf$.
Then the isotropy groups for the induced action on $P_2(X)$ are all members of
$\calf_2$.
\end{lemma}

\begin{proof}
Let $(x:y) \in P_2(X)$ and $g \in G_{(x:y)}$.
Then either ($gx = x$ and $gy = y$), or ($gx = y$ and $gy = x$).
Obviously $G_x \cap G_y \subseteq G_{(x:y)}$ and
$G_x \cap G_y \in \calf$. Hence it remains to show that
the index of $G_x \cap G_y$ in $G_{(x:y)}$ is two
if $G_x \cap G_y \not= G_{(x:y)}$. 
Choose $g_0 \in G_{(x:y)} \setminus G_x \cap G_y$.
Then for every $g \in G_{(x:y)} \setminus G_x \cap G_y$ we 
have $gg_0 \in G_x \cap G_y$.
\end{proof}

\begin{remark}[The role of $\calf_2$] 
In $K$-theory one can replace the family $\VCyc$ by
the family $\VCyc_I$ of subgroups which are either 
finite or virtually cyclic of type I, see~\cite{Davis-Khan-Ranicki(2008)},~\cite{Davis-Quinn-Reich(2010)}.
The corresponding result does not hold for $L$-theory:
in the calculation of the $L$-theory of the infinite dihedral
group non-trivial UNil-terms appear (see~\cite {Cappell(1974c)}). 
Hence in the
proof of the $L$-theory case there must be an argument in the proof,
which does not appear in the $K$-theory case and where
one in contrast to the $K$-Theory case needs to consider virtually
cyclic groups of type II as well. 
This happens actually in
the previous Lemma~\ref{lem:calf_2-isotropie-for-P_2} which forces us
to replace $\calf$ by $\calf_2$.

In the $L$-theory case the situation is just the other way around,
it turns out that one can ignore the virtually
cyclic groups of type I (see~\cite[Lemma~4.2]{Lueck(2005heis)}), 
but not the ones of type II.
\end{remark}

\begin{lemma}
\label{lem:P_2-simplicial}
Let $\Sigma$ be a finite dimensional simplicial complex.
Then $P_2(\Sigma)$ can be equipped with the structure of a simplicial
complex such that
\begin{enumerate}
\item \label{lem:P_2-simplicial:action}
      for every simplicial automorphism $f$ of $\Sigma$, the induced
      automorphism $P_2(f)$ of $P_2(\Sigma)$ is simplicial;
\item \label{lem:P_2-simplicial:d_P_2-vs-d-1}
      for $\e > 0$ there is $\delta > 0$, depending only on $\e$ and
      the dimension of $\Sigma$ such that for $z,z' \in P_2(\Sigma)$
      \begin{equation*}
        d_{P_2(\Sigma,d^1)}( z,z' ) \leq \delta \quad \implies \quad
        d^1_{P_2(\Sigma)}( z,z' ) \leq \e,
      \end{equation*}
      where $d^1_{P_2(\Sigma)}$ is the $l^1$-metric for the simplicial complex
      $P_2(\Sigma)$ and $d_{P_2(\Sigma,d^1)}$ is the metric induced
      from the $l^1$-metric on $\Sigma$, see
      Definition~\ref{def:P_2}~\ref{def:P_2:d}.
\end{enumerate}
\end{lemma}

\begin{proof}
Let $\Sigma^1$ denote the first barycentric subdivision of $\Sigma$.
The vertices of each simplex in $\Sigma^1$ are canonically ordered.
Then $\Sigma \times \Sigma$ can be given a simplicial structure as follows:
the set of vertices is $\Sigma^1(0) \times \Sigma^1(0)$, where $\Sigma^1(0)$
denotes the set of vertices of $\Sigma^1$.
The simplices are of the form
$\{ (e_0,f_0),\dots,(e_n,f_n) \}$ where
\begin{itemize}
\item $e_i, f_i \in \Sigma^1(0)$;
\item $\Delta := \{e_0,\dots,e_n\}$ and $\Delta' := \{f_0,\dots,f_n\}$
      are simplices  of $\Sigma^1$;
\item for $i = 1,\dots,n$ we have $e_{i-1} \leq e_i$ and
      $f_{i-1} \leq f_i$ with respect to the order of the
      simplices of $\Delta$ and $\Delta'$.
\end{itemize}
The flip map $\Sigma \times \Sigma \to \Sigma \times \Sigma, \quad
(x,y) \mapsto (y,x)$ is a simplicial map. If for a simplex $\tau$ the
interior of $\tau$ and the image of the interior of $\tau$ under the
flip map have a non-empty intersection, then the flip map is already
the identity on $\tau$. Thus we obtain an induced simplicial structure
on $P_2(\Sigma)$. It is now easy to see that this simplicial structure
has the required properties mentioned
in~\ref{lem:P_2-simplicial:action}.

It remains to prove assertion~\ref{lem:P_2-simplicial:d_P_2-vs-d-1}.
Fix $\epsilon > 0$. Let $\Delta_{4(\dim(\Sigma)+1)-1}$ be the
simplicial complex given by the standard
$(4(\dim(\Sigma)+1)-1)$-simplex.  A priori we have four topologies on
$P_2(\Delta_{4(\dim(\Sigma)+1)-1})$.  The first one comes from the
topology on $\Delta_{4(\dim(\Sigma)+1)-1}$, the second one from the
simplicial structure on $P_2(\Delta_{4(\dim(\Sigma)+1)-1})$
constructed above, and the third and fourth come from the metrics
$d^1_{P_2(\Delta_{4(\dim(\Sigma)+1)-1})}$ and
$d_{P_2(\Delta_{4(\dim(\Sigma)+1)-1},d^1)}$. Since
$P_2(\Delta_{4(\dim(\Sigma)+1)-1})$ is compact and hence locally finite,
one easily checks that all these topologies agree.
Since $P_2(\Delta_{4(\dim(\Sigma)+1)-1})$
is compact, we can find $\delta > 0$ such that for all $z,z' \in
P_2(\Delta_{4(\dim(\Sigma)+1)-1})$
\begin{equation}
  \label{eq:for-epsilon-there-is-delta}
  d_{P_2(\Delta_{4(\dim(\Sigma)+1)-1},d^1)}(z,z') \leq \delta \; \implies \;
  d^1_{P_2(\Delta_{4(\dim(\Sigma)+1)-1})}(z,z') \leq \epsilon.
\end{equation}
For the general case we make the following three observations.
Firstly, the $l^1$-metric is preserved under inclusions of
subcomplexes. Secondly, the construction of the simplicial structure
on the product is natural with respect to inclusions of
subcomplexes. Thirdly, for every choice of four points in $\Sigma$,
there is a subcomplex with at most $4(\dim \Sigma + 1)$ vertices
containing these four points.
Since~\eqref{eq:for-epsilon-there-is-delta} holds for
$\Delta_{4(\dim(\Sigma)+1)-1}$, it holds for $\Sigma$.
\end{proof}

\begin{lemma} \label{lem:f_versus_p_2(f)_metrically}
Let $(X,d_X)$ and $(Y,d_Y)$ be metric spaces. Let $f \colon X \to Y$
be a map. Suppose for $\delta, \epsilon > 0$ that
$d_Y(f(x),f(x')) \le \epsilon/2$ holds for all $x,x' \in X$ which satisfy
$d_X(x,x') \le \delta$.

Then $d_{P_2(Y)}(P_2(f)(z),P_2(f)(z')) \le \epsilon$ holds 
for all $z,z' \in P_2(X)$ which satisfy
$d_{P_2(X)}(z,z') \le \delta$.
\end{lemma}
\begin{proof}
Suppose for $(x:x'), (y:y') \in P_2(X)$ that 
$d_{P_2(X)}((x:x'),(y:y')) \le \delta$
holds. We get from definition that
$d_X(x,y) + d_X(x',y') \le \delta$ or
$d_X(x,y') + d_X(x',y) \le \delta$ holds. 
This implies that $d_X(x,y), d_X(x',y') \le \delta$ or
$d_X(x,y'),d_X(x',y) \le \delta$ is valid. 
We conclude from the assumptions
that $d_Y(f(x),f(y))\le \epsilon/2$ and  $d_Y(f(x'),f(y')) \le \epsilon/2$ hold
or that $d_Y(f(x),f(y'))\le \epsilon/2$ and  
$d_Y(f(x'),f(y)) \le \epsilon/2$ hold.
This implies that $d_Y(f(x),f(y)) + d_Y(f(x'),f(y')) \le \epsilon$ or
$d_Y(f(x),f(y')) + d_Y(f(x'),f(y)) \le \epsilon$ is true. Hence
\begin{eqnarray*}
\lefteqn{d_{P_2(Y)}(P_2(f)(x:x'),P_2(f)(y:y'))} & &
\\
& = &
d_{P_2(Y)}(f(x):f(x')),(f(y):f(y')))
\\
& = &
\min \{d_Y(f(x),f(y)) + d_Y(f(x'),f(y')), d_Y(f(x),f(y')) + d_Y(f(y),f(x')) \}
\\
& \le &
\epsilon.
\end{eqnarray*}
\end{proof}

Let $S \subseteq G$ be a finite subset and $\Lambda > 0$.
Let $(X,d)$ be a metric space
with a homotopy $S$-action $(\phi,H)$.  Since $P_2(X)$ is functorial
in $X$ and there is a natural map 
$P_2(X) \times [0,1] \to P_2(X \times [0,1])$, 
we obtain an induced homotopy $S$-action $(P_2(\phi),P_2(H))$
on $P_2(X)$. Let $d_{S,\Lambda,G \times P_2(X)}$ be the metric on $G \times
P_2(X)$ associated in Definition~\ref{def:d_S_lambda}  to
$(P_2(X),d_{P_2(X)})$ and the homotopy $S$-action
$(P_2(\phi),P_2(H))$, where $d_{P_2(X)}$ has been introduced in
Definition~\ref{def:P_2}~\ref{def:P_2:d} with respect to the given
metric $d$ on $X$.  Let $d_{S,\Lambda,G \times X}$ be the metric on $G \times
X$ associated in Definition~\ref{def:d_S_lambda} to the given
metric $d$ and homotopy $S$-action $(\phi,H)$ on $X$.  Let
$d_{S,\Lambda,P_2(G \times X)}$ be the metric on $P_2(G\times X)$ introduced in
Definition~\ref{def:P_2}~\ref{def:P_2:d} with respect to the metric
$d_{S,\Lambda,G \times X}$.

\begin{lemma}\label{lem:omega_G_timesP_2(X)_to_P_2(G_timesX)}
The map
$$\omega \colon G \times P_2(X) \to P_2(G \times X), \quad
(g,(x:y)) \mapsto ((g,x):(g,y))$$ is well-defined. We have for
$(g,(x:x'))$ and $(h,(y:y'))$ in $G \times P_2(X)$
\begin{multline*}
d_{S,\Lambda,P_2(G \times X)}\left(\omega((g,(x:x'))),\omega((h,(y:y')))\right)
\\
\le  2 \cdot d_{S,\Lambda,G\times P_2(X)}((g,(x:x')),(h,(y:y'))).
\end{multline*}
\end{lemma}
\begin{proof}
Consider $(g,(x:x'))$ and $(h,(y:y'))$ in $G \times P_2(X)$.
Consider $\epsilon > 0$. 
By definition we find $n \in \IZ$, $n \ge 0$,
elements $x_0,\dots,x_n$, $x_0',\dots,x_n'$, 
$z_0,\dots,z_n$ and $z_0',\dots,z_n'$ in $X$, 
elements $a_1,b_1,\dots,a_n,b_n$ in $S$ and
maps $f_1,\widetilde{f}_1,\dots,f_n,\widetilde{f}_n \colon X \to X$
such that
$$\begin{array}{l}
   (x:x') = (x_0:x_0'), (z_n:z_n') = (y:y');
   \\f_i \in F_{a_i}(\varphi,H),   \widetilde{f}_i \in F_{b_i}(\varphi,H),
   P_2(f_i)(z_{i-1}:z_{i-1}') = P_2(\widetilde{f}_i)(x_i:x_i');\\
   h = g a_1^{-1} b_1 \dots a_n^{-1} b_n;\\
   n + \sum_{i= 0}^n\Lambda \cdot d_{P_2(X)}((x_i:x_i'),(z_i:z_i'))
   \le d_{S,\Lambda,G \times P_2(X)}((g,(x:x')),(h,(y:y'))) + \epsilon.
\end{array}
$$
Next we construct sequences of elements
$x_0'', \ldots , x_n''$,
$z_0'', \ldots , z_n''$, $x_0''', \ldots , x_n'''$,  
and $z_0''', \ldots , z_n'''$ in $X$
such that
\begin{eqnarray*}
&&( x_i'': x_i''')  =  (x_i:x_i'); \quad
(z_i'':z_i''' ) = (z_i:z_i');
\\
&& d(x_i'',z_i'') + d(x_i''',z_i''') = 
         d_{P_2(X)}\left((x_i:x_i'),(z_i:z_i')\right);
\\
&& f_i(z_{i-1}'') =  \widetilde{f_i}(x_i'');
\quad f_i(z_{i-1}''') =  \widetilde{f_i}(x_i''').
\end{eqnarray*}

The construction is done inductively.
Put $x_0'' := x$ and $x_0''' := x_0'$.

Suppose that we have defined $x_0''$, $z_0''$, $\cdots$ ,$z_{i-1}''$, $x_i''$
and $x_0'''$, $z_0'''$, $\cdots$ ,$z_{i-1}'''$, $x_i'''$.
We have to specify $z_i''$ and $z_i'''$. 
By definition
$$d_{P_2(X)}\left((x_i:x_i'),(z_i:z_i')\right) =
\min\{d(x_i,z_i) +  d(x_i',z_i'), d(x_i,z_i') +  d(x_i',z_i)\}.$$
If  $x_i'' = x_i$ and $d_{P_2(X)}\left((x_i:x_i'),(z_i:z_i')\right) = d(x_i,z_i) +  d(x_i',z_i')$
and hold or if  $x_i'' = x_i'$ and 
$d_{P_2(X)}\left((x_i:x_i'),(z_i:z_i')\right) = d(x_i,z_i') +  d(x_i',z_i)$
hold, then put $z_i'' := z_i$ and $z_i''' := z_i'$.
If $x_i'' = x_i$ and $d_{P_2(X)}\left((x_i:x_i'),(z_i:z_i')\right) = d(x_i,z_i') +  d(x_i',z_i)$
hold or if $x_i'' = x_i'$ and 
$d_{P_2(X)}\left((x_i:x_i'),(z_i:z_i')\right) = d(x_i,z_i) +  d(x_i',z_i')$
hold, then put $z_i'' := z_i'$ and $z_i''' := z_i$.

Suppose that we have defined
$x_0''$, $z_0''$, $\cdots$, $x_{i-1}''$, $z_{i-1}''$ and
$x_0'''$, $z_0'''$, $\cdots$, $x_{i-1}'''$, $z_{i-1}'''$.
Then we have to specify $x_i''$ and $x_i'''$.
Since $P_2(f_i)(z_{i-1}:z_{i-1}') = P_2(\widetilde{f}_i)(x_i:x_i')$,
we have $f_i(z_{i-1}) = \widetilde{f}_i(x_i)$ and
$f_i(z_{i-1}') = \widetilde{f}_i(x_i')$ or we have
$f_i(z_{i-1}) = \widetilde{f}_i(x_i')$ and
$f_i(z_{i-1}') = \widetilde{f}_i(x_i)$. In the first case
put $x_i'' := x_i$ and $x_i''' := x_i'$ if $z_{i-1}'' = z_{i-1}$
and put  $x_i'' = x_i'$ and $x_i''' = x_i$ if $z_{i-1}'' = z_{i-1}'$.
In the second case
put $x_i'' := x_i'$ and $x_i''' := x_i$ if $z_{i-1}'' = z_{i-1}$
and put  $x_i'' := x_i$ and $x_i''' := x_i'$ if $z_{i-1}'' = z_{i-1}'$.
This finishes the construction of the elements $x_i''$, $x_i'''$,
$z_i''$ and $z_i'''$. One easily checks that the desired properties hold.

Put $y '' := z_n''$, $y''' := z_n'''$, $x'':= x_0''$ and $x''':=x_0'''$. 
We conclude from Definition~\ref{def:d_S_lambda}
\begin{eqnarray*}
d_{S,\Lambda,G \times X}((g,x''),(h,y''))
& \le & n + \sum_{i=0}^n \Lambda \cdot d(x_i'',z_i'');
\\
d_{S,\Lambda,G \times X}(g,x'''),(h,y'''))
& \le & n + \sum_{i=0}^n \Lambda \cdot d(x_i''',z_i''').
\end{eqnarray*}
This implies
\begin{eqnarray*}
\lefteqn{d_{S,\Lambda,G \times X}((g,x''),(h,y'')) +
d_{S,\Lambda,G \times X}(g,x'''),(h,y'''))}
\\
& \le &
2n +  \sum_{i=0}^n \Lambda \cdot \left(d(x_i'',z_i'') + d(x_i''',z_i''')\right)
\\
& \le &
2\cdot \left(n + \sum_{i=0}^n \Lambda \cdot
\left(d(x_i'',z_i'') + d(x_i''',z_i''')\right)\right).
\\
& \le &
2\cdot \left(n + \sum_{i=0}^n \Lambda \cdot
d_{P_2(X)}\left((x_i:x_i'),(z_i:z_i')\right)\right).
\\
& \le &
2 \cdot \left(d_{S,\Lambda,G \times P_2(X)}((g,(x:x')),(h,(y:y'))) + 
                                                   \epsilon\right).
\end{eqnarray*}
Since $\epsilon > 0$ was arbitrary, we conclude
\begin{multline*}
d_{S,\Lambda,G \times X}((g,x''),(h,y'')) + d_{S,\Lambda,G \times X}(g,x'''),(h,y'''))
\\
\le 2 \cdot d_{S,\Lambda,G \times P_2(X)}((g,(x:x')),(h,(y:y'))).
\end{multline*}
This implies
\begin{eqnarray*}
\lefteqn{d_{S,\Lambda,P_2(G \times X)}
              \left(\omega((g,(x:x'))),\omega((h,(y:y')))\right)}
\\
& = &
d_{S,\Lambda,P_2(G \times X)}\left(((g,x):(g,x')),((h,y):(h,y'))\right)
\\
& = &
\min\{d_{S,\Lambda,G \times X}((g,x),(h,y)) + d_{S,\Lambda,G \times X}((g,x'),(h,y')),
\\
& & \hspace{30mm}
d_{S,\Lambda,G \times X}((g,x),(h,y')) + d_{S,\Lambda,G \times X}((g,x'),(h,y))\}
\\
& \le &
d_{S,\Lambda,G \times X}((g,x''),(h,y'')) + d_{S,\Lambda,G \times X}(g,x'''),(h,y'''))
\\
& \le &
2 \cdot d_{S,\Lambda,G \times P_2(X)}((g,(x:x')),(h,(y:y'))).
\end{eqnarray*}
\end{proof}


\typeout{---------  The transfer in L-theory  ----------------}

\section{The transfer in $L$-theory}
\label{sec:The_transfer_in_L-theory}

\begin{summary*}
  In this section we construct a controlled $L$-theory transfer.
  Its formal properties are similar to the $K$-theory case, see
  Proposition~\ref{prop:L-theory-transfer-lift}, but its 
  construction is more complicated and uses the multiplicative hyperbolic
  Poincar\'e chain complex already mentioned in the
  introduction.
  This yields a suitable controlled
  symmetric form  on the fiber $P_2(X)$ for the transfer,
  see Proposition~\ref{prop:controlled-symmetric-chain-cx}.
  Here we make crucial  use of the flexibility of algebraic $L$-theory:
  There are many more $0$-dimensional Poincar\'e chain complexes,
  then there are $0$-dimensional manifolds.
\end{summary*}

Throughout this section we fix the following convention.

\begin{convention}
Let
\begin{itemize}
\item $G$ be a group;
\item $N \in \IN$;
\item $(X,d)$ be a compact contractible $N$-dominated metric space;
\item $Y$ be a $G$-space;
\item $\cala$ be an additive $G$-category with involution.
\end{itemize}
We equip $X \times X$ with the metric $d_{X \times X}$, defined by
$$d_{X \times X}((x_0,y_0),(x_1,y_1)) =
   d(x_0,y_0) + d(x_1,y_1)\quad 
       \;\text{for }(x_0,y_0), (x_1,y_1) \in X \times X.$$
\end{convention}

Similar to the tensor product constructed in 
Section~\ref{sec:transfer-up-to-homotopy} there is a tensor product
\begin{equation*}
  \calc(X,d;\IZ) \ox \calc(X,d;\IZ) \to \calc(X \x X, d_{X \x X}; \IZ)
\end{equation*}
induced by the canonical tensor product 
$\calf(\IZ) \ox \calf(\IZ) \to \calf(\IZ)$.
This tensor product is strictly compatible with the involution:
we have $\inv(M \ox N) = \inv(M) \ox \inv(N)$ for objects
$M$ and $N$, and similar
$\inv(f \ox g) = \inv(f) \ox \inv(g)$
for morphisms $f$ and $g$.
This tensor product is symmetric in the following sense:
for objects $M$ and $N$ there is a canonical isomorphism
$\flip_{M,N} \colon M \ox N \to N \ox M$.
This tensor product has a canonical extension to the idempotent 
completions.

We fix sign conventions for the induced tensor product of
chain complexes.
If $C$ and $D$ are chain complexes (over $\calc(X,d;\IZ)$) 
with differentials  $d^C$ and $d^D$
then the differential $d^{C \ox D}$ of the chain complex $C \ox D$ 
(over $\calc(X \x X, d_{X \x X}; \IZ)$)
is defined by
$d^{C \ox D}|_{C_p \ox D_q} = d^C \ox \id_{D_q} + (-1)^{p} \id_{C_p} \ox d^D$.
If $f \colon A \to C$ and $g \colon B \to D$ are maps of chain complexes, 
then we define
$f \ox g \colon A \ox B \to C \ox D$ by setting 
$(f \ox g)|_{A_p \ox B_q}  := (-1)^{|g| \cdot p} f|_{A_p} \ox g|_{B_q}$
(where $|g|$ is the degree of $g$).
The following flip will be important.
For chain complexes $C$ and $D$ we define an isomorphism
$\flip_{C,D} \colon C \ox D \to D \ox C$ 
by $(\flip_{C,D})|_{C_p \ox C_q} := (-1)^{pq}\flip_{C_p,C_q}$.

\begin{proposition}
\label{prop:controlled-symmetric-chain-cx}
Let $S$ be a finite subset of $G$ (containing $e$)
such that $S = S^{-1}$, i.e., if $g \in S$, then $g^{-1} \in S$.
Let $(\varphi,H)$ be a homotopy $S$-action on $X$.
For every $\e > 0$ there exists
a homotopy $S$-chain complex $\bfD = (D,\varphi^D,H^D)$
over $\Idem (\calc(P_2(X),d_{P_2(X)};\IZ))$
together with a chain isomorphism $\mu = \mu^D \colon D^{-*} \to D$
over $\Idem (\calc(P_2(X),d_{P_2(X)};\IZ))$
such that
\begin{enumerate}
\item \label{prop:controlled-symmetric-chain-cx:length}
      $D$ is concentrated in degrees $-N,\dots,N$;
\item \label{prop:controlled-symmetric-chain-cx:e-controlled}
      $D$ is $\e$-controlled;
\item \label{prop:controlled-symmetric-chain-cx:contractible}
      there is a homotopy $S$-chain equivalence $f \colon \bfP \to \bfT_{x_0}$
      to the trivial homotopy $S$-chain complex at $x_0 \in X$, such that
      $f \circ \mu^D \circ f^{-*}$ is the canonical identification
      of $T^{-*}$ with $T$;
\item \label{prop:controlled-symmetric-chain-cx:control-varphi}
      if $g \in S$ and $(x,y) \in \supp \varphi^D_a$ then
      $d(x,P_2(\varphi_g)(y)) \leq \e$;
\item \label{prop:controlled-symmetric-chain-cx:control-H}
      if $g,h \in S$ with $gh \in S$, and $(x,y) \in \supp H^P_{a,b}$
      then there is $t \in [0,1]$  such that
      $d(x,P_2(H_g(-,t))(y)) \leq \e$.
\item \label{prop:controlled-symmetric-chain-cx:control-psi}
      $\supp \mu \subseteq \{ (z,z) \mid z \in P_2(X) \}$,
      $\mu^{-*} = \mu$,
      $\mu \circ (\varphi_{g^{-1}}^D)^{-*} = \varphi^D_g \circ \mu$
      for all $g \in S$.
\end{enumerate}
\end{proposition}

The idea of the construction of $(\bfD,\mu)$ is very simple:
we take $\bfP$ from Lemma~\ref{prop:projective-chain-complex}
define $(\bfD,\mu)$ as the multiplicative hyperbolic 
Poincar\'e chain complex on $\bfP$ viewed over $P_2(X)$. 

\begin{proof}
 [Proof of Proposition~\ref{prop:controlled-symmetric-chain-cx}]
Let $\pr \colon  X\times X \to P_2(X)$ be
the obvious projection. Then 
$$d_{P_2(X)}(\pr(x_0,y_0),\pr(x_1,y_1))
  \le d_{X \times X}((x_0,y_0),(x_1,y_1))$$ 
holds for all  $(x_0,y_0),(x_1,y_1) \in X \times X$.

Let $\bfP = (P,\varphi^P,H^P)$ be a homotopy 
$S$-chain complex over $\Idem( \calc(X,d;\IZ))$
fulfilling the assertions of Lemma~\ref{prop:projective-chain-complex}
with respect to $\e' := \e / 2$ in place of $\e$. 
We obtain a chain complex $P^{-*} \otimes P$ over 
$\Idem( \calc(X \times X,d_{X \times X};\IZ))$. 
(At $(x,y) \in X \times X$ we have 
$(P^{-*} \ox P)_{(x,y)} = (P^{-*})_x \otimes P_y$). 
Define the chain complex $D$ over $\Idem (\calc(P_2(X),d_{P_2(X)};\IZ))$ 
to be the image of $(P^{-*}) \otimes P$ under the functor $\pr_*$  
from chain complexes over $\Idem( \calc(X \times X,d_{X \times X};\IZ))$ 
to chain complexes over $\Idem (\calc(P_2(X),d_{P_2(X)};\IZ))$ 
induced by $\pr$. 
Hence we have for $(x:y) \in P_1(X)$
$$D_{(x:y)} = \bigoplus_{\substack{(x',y') \in X \times X,\\ \pr(x',y') = (x:y)}} 
(P^{-*})_{x'} \otimes P_{y'}.$$
One easily checks assertions~\ref{prop:controlled-symmetric-chain-cx:length}
and~\ref{prop:controlled-symmetric-chain-cx:e-controlled} are satisfied.

In the sequel we will define certain chain maps and homotopies on the
level of $X \times X$ and is to be understood that we will apply the
functor $\pr_*$ to it to obtain constructions over $P_2(X)$.
We define the homotopy $S$-action by putting
$\varphi^D_g := (\varphi_{g^{-1}}^P)^{-*} \otimes \varphi^P_g$
and 
$H^D_{g,h} := (H^P_{h^{-1},g^{-1}})^{-*} \otimes (\varphi_g^P \circ \varphi_h^P)
                 + (\varphi_{(gh)^{-1}}^P)^{-*} \otimes H^P_{g,h}$.
We define $\mu \colon D^{-*} \to D$ as
$\flip \colon (P^{-*} \ox P)^{-*} = P \ox P^{-*} \to P^{-*} \ox P$.

One easily checks that because of 
Proposition~\ref{prop:projective-chain-complex}
assertions~\ref{prop:controlled-symmetric-chain-cx:contractible},~%
\ref{prop:controlled-symmetric-chain-cx:control-varphi}
and~\ref{prop:controlled-symmetric-chain-cx:control-H} are satisfied.

Notice that the support of $\mu$ is contained in the subset
$\Xi = \{((x,y), (x',y')) \mid x = y', x' = y\}$ of 
$(X \times X) \times (X \times X)$ and that the image of $\Xi$ under 
$\pr \times \pr \colon (X \times X) \times 
               (X \times X) \to P_2(X) \times P_2(X)$ 
is contained in the diagonal
$\{(z,z) \mid z \in P_2(X)\}$ of $P_2(X) \times P_2(X)$.
Straightforward calculations shows that
$\mu = \mu^{-*}$ and
$\mu \circ (\varphi^D_g)^{-*} = (\varphi^D_g) \circ \mu$.
This implies assertion~\ref{prop:controlled-symmetric-chain-cx:control-psi}.
\end{proof}

\begin{proposition}
\label{prop:L-theory-transfer-lift}
Let $T \subset S$ be finite subsets of $G$ (both containing $e$)
such that for $g,h \in T$, we have $gh \in S$.
Assume $T = T^{-1}$, i.e., if $g \in T$, then  $g^{-1} \in T$.
Let $\alpha \colon M^{-*} \to M$ be a quadratic form
such that $\alpha$ is a $T$-morphism in $\calo^G(Y;\cala)$
and $\alpha + \alpha^*$ is an $T$-isomorphism in $\calo^G(Y;\cala)$.
Let $\Lambda > 0$.
Then there is a $0$-dimensional ultra-quadratic $(S,2)$-Poincar{\'e} complex
$(C,\psi)$ over 
$\Idem(\calo^G(Y, G \times P_2(X),d_{S,\Lambda,G \times P_2(X)};\cala))$
which is concentrated in degrees $N,\dots,-N$,
such that
\begin{equation*}
      [p(C,\psi )] = [(M,\alpha)] \in
                L^{\langle 1 \rangle}_0(\Idem(\calo^G(Y;\cala)));
\end{equation*}
where  
$p \colon \Idem(\calo^G(Y, G \times P_2(X),d_{S,\Lambda,G \times P_2(X)};\cala))
             \to \Idem(\calo^G(Y;\cala))$
is the functor induced by the projection $ G \times P_2(X) \to \pt$.
\end{proposition}

Recall from Section~\ref{sec:The_space_P_2(X)}
that we use the metric $d_{P_2(X)}$ on $P_2(X)$, see
Definition~\ref{def:P_2}~\ref{def:P_2:d}, in order to
construct the metric $d_{S,\Lambda, G \x P_2(X)}$ as in
Definition~\ref{def:d_S_lambda}.
The proof will use a controlled version of the classical 
$L$-theory transfer, see~\ref{subsec:classical-L-transfer}.

\begin{proof}
[Proof of Proposition~\ref{prop:L-theory-transfer-lift}]
Let $\e := 1 / \Lambda$.
Let $\bfD = (D,\varphi^D,H^D)$, $\mu^D \colon D^{-*} \to D$
satisfy the assertions of Proposition~\ref{prop:controlled-symmetric-chain-cx}.
We define $\widetilde{C} := M \otimes D$ and
$\widetilde{\psi} := \tr^{\bfD} \alpha \circ (\id_{M^{-*}} \otimes \mu)$.
Using Proposition~\ref{prop:controlled-symmetric-chain-cx}~%
\ref{prop:controlled-symmetric-chain-cx:control-psi}
we compute
\begin{eqnarray*}
\widetilde{\psi} + \widetilde{\psi}^{-*} & = &
\sum_{a \in T} \alpha_a \otimes (\varphi^D_{a} \circ \mu^D)
  + (\alpha^*)_a \otimes ((\mu^D)^{-*} \circ ((\varphi^D)^{-*})_{a})
\\
& = &
\sum_{a \in T} \alpha_a \otimes (\varphi^D_{a} \circ \mu^D)
  + (\alpha^*)_{a} \otimes (\mu^D \circ (\varphi^D_{a^{-1}})^{-*})
\\
& = &
\sum_{a \in T} \alpha \otimes (\varphi^D_{a} \circ \mu^D)
  + (\alpha^*)_{a} \otimes  (\varphi^D_{a} \circ \mu^D)
\\
& = &
\sum_{a \in T} (\alpha + \alpha^*)_a \otimes (\varphi^D_{a} \circ \mu^D)
\\
& = &
\tr^{\bfD} (\alpha + \alpha^{*})_a \circ (\id \otimes \mu^D).
\end{eqnarray*}
Lemma~\ref{lem:homotopy-functorial} implies that
$(\id \otimes (\mu^D)^{-1}) \circ \tr^{\bfD}((\alpha + \alpha^*)^{-1})$
is a chain homotopy inverse for $\widetilde{\psi} + \widetilde{\psi}^{-*}$ 
with homotopies
given by $\sum_{a,b \in T} ((\alpha + \alpha^*)_a
                      \circ ((\alpha + \alpha^*)^{-1})_b) \otimes H^D_{a,b}$
and $\sum_{a,b \in T} (((\alpha + \alpha^*)^{-1})_a
                      \circ (\alpha + \alpha^*)_b) \otimes H^D_{a,b}$.
We conclude from  
Proposition~\ref{prop:controlled-symmetric-chain-cx} that the pair
$(\widetilde{C},\widetilde{\psi})$ is a $0$-dimensional
ultra-quadratic $(S,\e)$-Poincar{\'e} complex over $\Idem
\calo^G(Y,P_2(X),d_{P_2(X)};\cala)$. Let $f \colon \bfP \to
\bfT_{x_0}$ be the weak equivalence from
assertion~\ref{prop:controlled-symmetric-chain-cx:contractible} in
Proposition~\ref{prop:controlled-symmetric-chain-cx}. 
Let $q \colon
\calo^G(Y, P_2(X),d_{P_2(X)};\cala) \to \calo^G(Y;\cala)$ be the
functor induced by $P_2(X) \to \pt$. Then $\id_M \otimes f$ is a chain
homotopy equivalence. 
Using assertion~\ref{prop:controlled-symmetric-chain-cx:contractible} in
Proposition~\ref{prop:controlled-symmetric-chain-cx} it is not hard to
check that $(\id_M \otimes f) \circ \psi \circ (\id_M \otimes f)^{-*}$ is
chain homotopic to $\tr^{\bfT_{x_0}} \alpha$. 
Note that $q(\tr^{\bfT_{x_0}} \alpha) = \alpha$ 
(up to a canonical isomorphism $q( M \otimes T) \cong M$). 
Using \eqref{nl:things-are-equal-in-L} we conclude
$q[(\widetilde{C},\widetilde{\psi})] = [(M,\alpha)]
            \in L^{\langle 1 \rangle}_0(\Idem(\calo^G(Y;\cala)))$.
Let $F \colon \calo^G(Y,P_2(X),d_{P_2(X)};\cala) \to
        \calo^G(Y, G \times P_2(X), d_{S,\Lambda};\cala)$
be the functor induced by the map
$(g,x,e,t) \mapsto (g,g,x,e,t)$ and set $(C, \psi) := f (C,\psi)$.
Since $p \circ F = q$ we have $[p( C, \psi)] = [(M,\alpha)]$.
{}From the Definition~\ref{def:d_S_lambda} of $d_{S,\Lambda}$
and our choice of $\e$ it follows that
$(C, \psi)$ is a $0$-dimensional
ultra-quadratic $(S,2)$-Poincar{\'e} complex
over $\Idem \calo^G(Y,G \times P_2(X),d_{S,\Lambda};\cala)$.
\end{proof}


\typeout{--------- Proof of axiomatic Theorem  ----------}

\section{Proof of Theorem~\ref{the:axiomatic}}
\label{sec:Proof_of_axiomatic_Theorem}

\begin{proof}
[Proof of Theorem~\ref{the:axiomatic}~\ref{thm:axiomatic:K}]
Because of Lemma~\ref{lem:calfj_closed_under}~%
\ref{lem:calfj_closed_under:colim} 
we can assume without loss of generality
that $G$ is finitely generated. 
Let $N$ be the number appearing in 
Definition~\ref{def:transfer-reducible}.
According to Theorem~\ref{thm:obstruction-category}~%
\ref{thm:obstruction-category:K}
it suffices to show $K_1(\calo^G(E_\calf G;\cala)) = 0$
for every additive $G$-category $\cala$.
Fix such an $\cala$.
Consider $a \in K_1(\calo^G(E_\calf G;\cala))$.
Pick an automorphism $\alpha \colon M \to M$ in $\calo^G(E_\calf G;\cala))$
such that $[(M,\alpha)] = a$.
By definition $\alpha$ is an $T$-automorphism for some finite subset
$T$ of $G$ (containing $e$).
We can assume without loss of generality
that $T$ generates $G$, otherwise we enlarge $T$
by a finite set of generators.
Set $S := \{ ab \mid a,b \in T \}$.
Let $\e = \e(N,\cala,G,\calf,S)$ be the number appearing in
Theorem~\ref{thm:stability-element}~\ref{thm:stability-element:K}.
Set $\beta := 2$.
Let $(X,d)$,$(\varphi,H)$, $\Lambda$, $\Sigma$ and $f$ be as in
Proposition~\ref{prop:contracting-maps}.
Consider the following commuting diagram of functors
\begin{equation*}
\xymatrix{\calo^G(E_\calf G, G \times X, d_{S,\Lambda};\cala)
\ar[dr]_{p} \ar[rr]^F & &
\calo^G(E_\calf G, \Sigma, d^1;\cala)
\ar[ld]^{q}
\\
& \calo^G(E_\calf G;\cala)}
\end{equation*}
where $p$ resp.\ $q$
are induced by projecting $G \times X$ resp.\ $\Sigma$
to a point and $F$ is induced by $f$.
By Proposition~\ref{prop:K-theory-transfer-lift}
there is a
$(\beta,S)$-chain homotopy equivalence $[(C,\hat \alpha)]$
over $\Idem(\calo^G(E_\calf,G \times X,d_{S,\Lambda};\cala))$
such that $[p(C,\hat \alpha)] = a$.
Proposition~\ref{prop:contracting-maps}~\ref{prop:contractin-maps:estimate}
implies that $F(\hat \alpha)$ is an $(\e,S)$-chain homotopy equivalence
over $\Idem(\calo^G(E_\calf G, \Sigma, d^1;\cala))$.
By Theorem~\ref{thm:stability-element}~\ref{thm:stability-element:K}
$[F(C,\hat \alpha)] = 0$.
Therefore $a = [q \circ F (C,\hat \alpha)] = 0$.
\end{proof}

We record the following corollary to
Theorem~\ref{the:axiomatic}~\ref{thm:axiomatic:K}.

\begin{corollary}
\label{cor:K-of-calo-vanishes}
Let $G$ be a finitely generated group that is
transfer reducible over $\calf$.
Let $\cala$ be an additive $G$-category.
Then for $i \leq 1$ we have
\begin{equation*}
  K_i(\calo^G(\EGF{G}{\calf_2};\cala)) = 0.
\end{equation*}
\end{corollary}

\begin{proof}
Clearly $G$ is also transfer reducible over $\calf_2$ and therefore by
Theorem~\ref{the:axiomatic}~\ref{thm:axiomatic:K}
the assembly map \eqref{eq:assembly-map-calf-K} is an isomorphism
for $n < 1$ and surjective for $n=1$.
The result follows because the $K$-theory of $\calo^G(\EGF{G}{\calf_2};\cala)$
is the cofiber of this assembly map,
compare~\cite[Section~3.3]{Bartels-Lueck-Reich(2008hyper)}.
\end{proof}

\begin{proof}
[Proof of Theorem~\ref{the:axiomatic}~\ref{thm:axiomatic:L}]
Because of 
Lemma~\ref{lem:calfj_closed_under}~\ref{lem:calfj_closed_under:colim} 
we can assume without loss of generality
that $G$ is finitely generated. 
Let $N$ be the number appearing in
Definition~\ref{def:transfer-reducible}.
According to 
Theorem~\ref{thm:obstruction-category}~\ref{thm:obstruction-category:L}
it suffices to show $\Li_0(\calo^G(E_{\calf_2} G;\cala)) = 0$
for every additive $G$-category $\cala$ with involution. 
Fix such an $\cala$. 
By~\eqref{nl:K-vanishes-decorations} and
Corollary~\ref{cor:K-of-calo-vanishes}
we know that
$$L^{\langle 1 \rangle}_0 (\calo^G(E_{\calf_2} G;\cala)) \to
 \Li_0 (\calo^G(E_{\calf_2} G;\cala))$$
is an isomorphism. 
It suffices to show that this map is also the zero map.

Consider $a \in L^{\langle 1 \rangle}(\calo^G(E_{\calf_2} G;\cala))$.
Pick a quadratic form $(M,\alpha)$ over the category 
$\calo^G(E_{\calf_2} G;\cala))$
such that $[(M,\alpha)] = a$.
By definition there is a finite subset $T$
of $G$ (containing $e$) such that $\alpha$ is $T$-controlled and
$\alpha + \alpha^*$  is a $T$-isomorphism.
We can assume without loss of generality
that $T = T^{-1}$ and that $T$ generates $G$.
Set $S := \{ ab \mid a,b \in T \}$.
Let $\epsilon = \epsilon(2N,\cala,G,\calf_2,S)$ be the number appearing in
Theorem~\ref{thm:stability-element}~\ref{thm:stability-element:L}.
By Lemma~\ref{lem:P_2-simplicial}~\ref{lem:P_2-simplicial:d_P_2-vs-d-1}
there is $\delta$ such that for every simplicial complex $\Sigma'$
of dimension $\leq N$ we have for $x,y \in P_2(\Sigma')$
\begin{equation*}
d_{P_2(\Sigma',d^1)} (x,y) \leq \delta \quad \implies 
                    \quad d^1_{P_2(\Sigma')}(x,y) \leq \e.
\end{equation*}
Set $\beta := 2$.
Let $(X,d)$,$(\varphi,H)$, $C$, $\Sigma$ and $f$ be as in
Proposition~\ref{prop:contracting-maps}, but with respect to $\delta/2$ 
in place of $\epsilon$ and $2\beta$ instead of $\beta$. 
In particular we have for $(g,x), (h,y) \in G \times X$
$$d_{S,\Lambda, G \times X}((g,x),(h,y) \le 2 \beta
\quad \implies \quad
d^1_{\Sigma}(f(g,x),f(h,y)) \le \delta/2.$$
We conclude from Lemma~\ref{lem:f_versus_p_2(f)_metrically}
for $z,z' \in P_2(G \times X)$
$$
d_{S,\Lambda, P_2(G \times X)}(z,z') \le 2 \beta
\quad \implies \quad
d_{P_2(\Sigma,d^1)}(P_2(f)(z),P_2(f)(z') \le \delta.
$$
Let $\widehat{f} \colon G \times P_2(X) \to P_2(\Sigma)$
be the composite of $P_2(f)$
with the map $\omega \colon G \times P_2(X) \to P_2(G \times X)$
defined in Lemma~\ref{lem:omega_G_timesP_2(X)_to_P_2(G_timesX)}.
Because of Lemma~\ref{lem:omega_G_timesP_2(X)_to_P_2(G_timesX)}
we have for $(g,(x:x')),(h,(y:y')) \in G \times P_2(X)$
\begin{multline*}
d_{S,\Lambda,G \times P_2(X)}((g,(x:x')),(h,(y:y'))) \le \beta
\\
\implies \quad
d_{S,\Lambda,P_2(G \times  X)}(\omega(g,(x:x')),\omega(h,(y:y'))) \le 2\beta.
\end{multline*}
We conclude  for $(g,(x:x'),(h,(y:y') \in G \times P_2(X)$
\begin{multline}
\label{eq:beta-gives-e}
d_{S,\Lambda,G \times P_2(X)}((g,(x:x')),(h,(y:y'))) \leq \beta
\\ \implies  \quad
    d^1_{P_2(\Sigma)}( \widehat{f}(g,(x:x'), \widehat{f}(h,(y:y') ) \leq \epsilon.
  \end{multline}
Consider the following commuting diagram of functors
\begin{equation*}
\xymatrix@!C=8em
{\calo^G(\EGF{G}{\calf_2}, G \times P_2(X), d_{S,\Lambda,G \times  P_2(X)};\cala)
\ar[dr]_{p} \ar[rr]^F & &
\calo^G(\EGF{G}{\calf_2}, P_2(\Sigma), d^1_{P_2(\Sigma)};\cala)
\ar[ld]^{q}
\\
& \calo^G(\EGF{G}{\calf_2};\cala)
}
\end{equation*}
where $p$ resp.\ $q$
are induced by projecting $G \times P_2(X)$ resp.\ $P_2(\Sigma)$
to a point and $F$ is induced by $\widehat{f}$.
By Proposition~\ref{prop:L-theory-transfer-lift}
there is a $0$-dimensional
ultra-quadratic $(S,\beta)$-Poincar{\'e} complex
$(C,\psi)$ over
$\Idem(\calo^G(\EGF{G}{\calf_2}, G \times P_2(X), 
                   d_{S,\Lambda,G \times P_2(X)};\cala))$
concentrated in degrees $-N,\dots,N$ such that
$$[p(C,\psi)] = [(M,\alpha)] \in
     L_0^{\langle 1 \rangle}(\Idem(\calo^G(\EGF{G}{\calf_2};\cala))).$$
We conclude from~\eqref{eq:beta-gives-e}
that $F(C,\psi)$ is an $(S,\e)$-Poincar{\'e} complex.
{}From Theorem~\ref{thm:stability-element}~\ref{thm:stability-element:L}
and Lemma~\ref{lem:calf_2-isotropie-for-P_2} we deduce
\begin{equation*}
  [F(C,\psi)] = 0  \in
  \Li_0(\calo^G(\EGF{G}{\calf_2}, P_2(\Sigma), d^1;\cala)).
\end{equation*}
Therefore
$a = [q \circ F (C,\psi)] = 0$ holds in 
$\Li_0(\calo^G(\EGF{G}{\calf_2};\cala))$.
\end{proof}


\appendix

\typeout{-- Classical transfers and the multiplicative hyperbolic form  --}

\section{Classical transfers and the multiplicative hyperbolic form}
\label{sec:classical-transfer_and_mult-hyperbolic}

\begin{summary*}
  In this appendix we review classical (uncontrolled) transfers.
  We also discuss the multiplicative hyperbolic form in
  an uncontrolled context and show that this construction 
  yields a homomorphism from $K_0(\Lambda)$ to $L^0_p(\Lambda)$. 
\end{summary*}

\subsection{Transfer for the Whitehead group}
\label{subsec:Whitehead-transfer}
We briefly review the transfer for the Whitehead group
for a fibration $F \to E \xrightarrow{p} B$ of 
connected finite $CW$-complexes.
For simplicity we will assume that $\pi_1(p) \colon \pi_1(E) \to
\pi_1(B)$ is bijective and  we will identify in the sequel
$G := \pi_1(E) = \pi_1(B)$.

Recall that the fiber transport gives a homomorphism of monoids 
$G \to [F,F]$. 
Thus we obtain a finite free $\IZ$-chain complex $C = C_*(F)$,
namely, the cellular $\IZ$-chain complex of $F$, together with an
operation of $G$ up to chain homotopy, i.e., a homomorphism of monoids
$\rho \colon G \to [C,C]_{\IZ}$ to the monoid of chain homotopy
classes of $\IZ$-chain maps $C \to C$
(compare~\cite[Section~5]{Lueck(1987)}).  
An algebraic transfer map $p^* \colon \Wh(B) \to \Wh(E)$ in terms of
chain complexes is given in~\cite[Section~4]{Lueck(1986)}.  
We recall its definition in the special case, 
where $\pi_1(p)$ is bijective.

Given an element $a = \sum_{g \in G} \lambda_g g \in \IZ G$, define a
$\IZ G$-chain map of finitely generated free $\IZ G$-chain complexes,
unique up to $\IZ G$-chain homotopy, by
\begin{equation}
\label{eq:a_otimes_tC}
  a \otimes_t  C \colon \IZ G \otimes_{\IZ} C \to \IZ G \otimes_{\IZ} C,
  \quad g' \otimes x \mapsto \sum_{g \in G}
  \lambda_g \cdot g'g^{-1} \otimes r(g)(x),
\end{equation}
where $r(g) \colon C \to C$ is some representative of $\rho(g)$
(compare~\cite[Section~5]{Lueck(1987)}). 
Thus we obtain a ring homomorphism 
$\IZ G \to [\IZ G \otimes_{\IZ} C,\IZ G \otimes_{\IZ} C]_{\IZ G}$ 
to the ring of $\IZ G$-chain homotopy classes of $\IZ G$-chain maps 
$\IZ G \otimes_{\IZ}C \to \IZ G \otimes_{\IZ}C$.  
It extends in the obvious way to matrices over $\IZ G$, namely, for a
matrix $A \in M_{m,n}(\IZ G)$ we obtain a $\IZ G$-chain map, unique up
to $G$-homotopy,
$$A \otimes_t C \colon \IZ G^m \otimes_{\IZ} C \to \IZ G^n \otimes_{\IZ} C.$$
The algebraic transfer $p^* \colon \Wh(G) \to \Wh(G)$ sends the class of an
invertible matrix $A \in GL_n(\IZ G)$ to the Whitehead torsion of the
$\IZ G$-self-chain homotopy equivalence 
$A \otimes_t C \colon \IZ G^n \otimes_{\IZ} C \to \IZ G ^n \otimes_{\IZ} C$.


\subsection{Classical $L$-theory transfer}
\label{subsec:classical-L-transfer}
To obtain an $L$-theory transfer we have additionally to assume that
$F$ is a finite $n$-dimensional Poincar\'e complex. 
For simplicity we assume that $F$ is an oriented $n$-dimensional 
Poincar\'e complex and the fiber transport $G \to [F,F]$ takes values 
in homotopy classes of orientation preserving self-homotopy equivalences 
and \--- as before \--- that $\pi_1(p)$ is bijective. 
We review the algebraically
defined transfer maps $p^* \colon L_m(\IZ G) \to L_{m+n}(\IZ G)$ 
(see~\cite{Lueck-Ranicki(1988)}).
Because $F$ is a Poincar\'e complex, there is a symmetric form
$\varphi \colon C^{-*} \to C$, 
where $C^{-*}$ denotes the dual of the cellular chain complex of $F$,i.e.,
$(C^{-*})_n = (C_{-n})^*$.
If $\psi \colon M^{*} \to M$ is a quadratic form over $\IZ[G]$,
then the composition
\[
\psi \ox_t (C,\varphi) \colon
(M \ox C)^{-*} \cong M^{*} \ox C^{-*} 
\xrightarrow{\id \ox \varphi}
M^{*} \ox C
\xrightarrow{\psi \ox_t C} 
M \ox C
\]
defines an ultra-quadratic form on $M \ox C$.
The $L$-theory transfer sends the class of 
$(M,\psi) \in L_0(\IZ G)$ to the
class of
$(M \ox C,\psi \ox_t (C,\varphi)$.  


\subsection{The multiplicative hyperbolic form}
\label{subsec:multilicative-hyperbolic-form}
Let $\Lambda$ be a commutative ring.
Let $P$ be a finitely generated projective $\Lambda$-module.
Since $\Lambda$ is commutative, the dual $\Lambda$-module
$P^* = \hom_{\Lambda}(P;\Lambda)$ and the tensor product
$P \otimes_{\Lambda} P^*$ are finitely generated projective $\Lambda$-modules.
Define the $\Lambda$-isomorphism
\begin{eqnarray}
& \psi_P \colon P^* \otimes_{\Lambda} P \to (P^* \otimes_{\Lambda} P)^* &
\label{psi_P}
\end{eqnarray}
by sending $\alpha \otimes x  \in  P^* \otimes_{\Lambda} P$ 
to the $\Lambda$-homomorphisms
$P^* \otimes_{\Lambda} P \to \Lambda, \; 
     \beta \otimes y \mapsto \alpha(y) \cdot \beta(x)$.
The composite
$$P^* \otimes_{\Lambda} P \cong
((P^* \otimes_{\Lambda} P)^*)^* \xrightarrow{\psi_P^*}  
(P^* \otimes_{\Lambda} P)^*$$
agrees with $\psi_P$. 
Hence $(P^* \otimes_{\Lambda} P,\psi_P)$ is a non-singular symmetric
$\Lambda$-form. 
We call it the \emph{multiplicative hyperbolic
symmetric $\Lambda$-form} associated
to the finitely generated  projective $\Lambda$-module $P$
and denote it by $H_{\otimes}(P)$. 
Ian Hambleton pointed out that under the 
identification $P^* \otimes_\Lambda P \cong \End_\Lambda(P,P)$
this form corresponds to the trace form $(A,B) \mapsto \tr(AB)$.
The name multiplicative hyperbolic form
comes from the fact that it is the obvious multiplicative
version of the standard hyperbolic symmetric
form $H(P)$ which is given by the $\Lambda$-isomorphism
$P^* \oplus P \to (P^* \oplus P)^*$ sending 
$(\alpha,x) \in  P^* \oplus_{\Lambda} P$
to the $\Lambda$-homomorphisms
$P^* \oplus P \to \Lambda, \; (\beta,y) \mapsto \alpha(y) + \beta(x)$;
just replace $\oplus$ by $\otimes$ and $+$ by $\cdot$.

The hyperbolic symmetric form of a finitely 
generated projective $\Lambda$-module $P$
represents zero in the symmetric $L$-group $L^0_p(\Lambda)$. 
This is not true for the multiplicative version.
Namely, define a  homomorphism
\begin{eqnarray}
& H_{\otimes}^{\Lambda} \colon K_0(\Lambda) \to L^0_p(\Lambda) &
\label{H_otimes_from_K_to_L}
\end{eqnarray}
by sending the class $[P]$ of a finitely generated 
projective $\Lambda$-module $P$
to the class $[H_{\otimes}(P)]$  of the non-singular symmetric $\Lambda$-form
$H_{\otimes}(P)$. 
We have to show that this is well-defined, i.e., we must prove
$$[H_{\otimes}(P \oplus Q)] = [H_{\otimes}(P)] + [H_{\otimes}(P)] 
                                              \in L^0_p(\Lambda)$$
for two finitely generated projective $\Lambda$-modules $P$ and $Q$.
This follows from the fact that we have an isomorphism of
$\Lambda$-modules
\begin{multline*}
(P \oplus Q)^* \otimes_{\Lambda} P \oplus Q \cong
P^* \otimes_{\Lambda} P \oplus Q^* \otimes_{\Lambda} Q \oplus
Q^* \otimes_{\Lambda} P \oplus P^* \otimes_{\Lambda} Q
\\
\cong
P^* \otimes_{\Lambda} P \oplus Q^* \otimes_{\Lambda} Q \oplus
\left(Q^* \otimes_{\Lambda} P \oplus (Q^* \otimes_{\Lambda} P)^*\right)
\end{multline*}
which induces an isomorphism of non-singular symmetric $\Lambda$-forms
$$H_{\otimes}(P \oplus Q)
\cong H_{\otimes}(P) \oplus H_{\otimes}(Q) \oplus H(Q^*
\otimes_{\Lambda} P).$$

Since $\Lambda$ is commutative, the tensor product $\otimes_{\Lambda}$
induces the structure of a commutative ring on $ K_0(\Lambda)$ and
$L^0_p(\Lambda)$. 
One easily checks that the map
$H_{\otimes}^{\Lambda}$ of~\eqref{H_otimes_from_K_to_L} is a ring homomorphism.

\begin{example}[$\Lambda = \IZ$]
If we take for instance $\Lambda = \IZ$, we obtain isomorphisms
\begin{eqnarray*}
\rk \colon K_0(\IZ) & \xrightarrow{\cong} & \IZ;
\\
\sign \colon L^0_p(\IZ) & \xrightarrow{\cong} & \IZ,
\end{eqnarray*}
by taking the rank of a finitely generated free abelian group and the signature
(see~\cite[Proposition~4.3.1 on page~419]{Ranicki(1981)}). 
Obviously
$$H_{\otimes}^{\IZ} \colon K_0(\IZ) \to L^0_p(\IZ)$$
sends $[\IZ]$ to the  class
of the symmetric form $\IZ \to \IZ^*$ sending $1 \in \IZ$ to the identity
$\id_{\IZ} \in \IZ^*$. Hence $H_{\otimes}^{\IZ}$ is a bijection.
\end{example}


\subsection{The chain complex version of $H_\ox$}
\label{subsec:ch-cx-version-of-multi-hyperbolic}
Next we give a chain complex version of this construction.
Let $C$ be a finite projective $\Lambda$-chain complex, i.e., a
$\Lambda$-chain complex such that each $\Lambda$-module $C_i$ is finitely
generated projective and $C_i$ is non-trivial for 
only finitely many $i \in \IZ$.

Given two $\Lambda$-chain complexes $C$ and $D$, define their
\emph{tensor product}
$$(C \otimes_{\Lambda} D, c \otimes_{\Lambda} d)$$
to be the $\Lambda$-chain
complex whose $n$-th-chain module is $\bigoplus_{i,j, i+ j = n} C_i
\otimes_{\Lambda} D_j$. 
The differential is given by the formula
$$(c \otimes d)(x \otimes y) := 
    c(x) \otimes y + (-1)^{|x|} \cdot x \otimes d(y).$$
We need $\Lambda$ to be  commutative to ensure that
$C \otimes_{\Lambda} D$ is indeed a $\Lambda$-chain complex.
If $C$ and $D$ are finite projective $\Lambda$-chain complexes, then
$C \otimes_{\Lambda} D$ is a finite projective $\Lambda$-chain complex.

If $f \colon A \to C$ and $g \colon B \to D$ are maps of chain complexes
of degree $|f|$ and $|g|$, then we define
$f \otimes g \colon A \otimes B \to C \otimes D$ by
$(f \otimes g) (x \otimes y) := (-1)^{|g| \cdot |x|} f(x) \otimes g(y)$. 
The flip isomorphism
$$\flip \colon  C \otimes_{\Lambda} D \xrightarrow{\cong} D \otimes_{\Lambda} C$$
is given by $\flip(x \otimes y) := (-1)^{|x|\cdot |y|} \cdot y \otimes x$.
For chain complexes $C$ and $D$ we define a chain map
$$\mu_{C,D} \colon C^{-*} \otimes D^{-*} \to (C \otimes D)^{-*}$$
by $\mu_{C,D}(\alpha \otimes \beta)(a \otimes b) := 
               (-1)^{|\beta| \cdot |a|}\alpha(a) \beta(b)$.
Then $\otimes$ and $^{-*}$ are compatible in the following sense:
if $f \colon A \to C$ and $g \colon B \to D$ are maps of chain complexes,
then 
$\mu_{A,B} \circ (f^{-*} \otimes g^{-*}) = (f \otimes g)^{-*} \circ \mu_{C,D}$.
If $C$ and $D$ are finite projective $\Lambda$-chain complexes,
then $\mu_{C,D}$ is an isomorphism and yields a canonical
identification. 
Suppressing this identification
the formula reads $f^{-*} \otimes g^{-*} = (f \otimes g)^{-*}$.
Define
$$\mu_C \colon C \otimes_{\Lambda} C^{-*} \xrightarrow{\cong} (C^{-*}
\otimes_{\Lambda} C)^{-*} $$
to be the composite
$$ C \otimes_{\Lambda} C^{-*} \xrightarrow{\cong}
(C^{-*})^{-*} \otimes_{\Lambda} C^{-*} \xrightarrow{\mu_{C^{-*},C}}  
(C^{-*} \otimes_{\Lambda} C)^{-*}.$$
Explicitly $\mu_C$ sends  $x \otimes \alpha \in C_i \otimes_{\Lambda} (C_{-j})^*$
to the $\Lambda$-map $C_i^* \otimes C_{-j}\to \Lambda,\; \beta \otimes y \mapsto
\beta(x) \cdot \alpha(y)$, where we think of $\hom_{\Lambda}(C_i^* \otimes
C_{-j};\Lambda)$ as a submodule of the $(i+j)$-th chain module of
$(C^{-*} \otimes_{\Lambda} C)^{-*}$ in
the obvious way. 
Define an isomorphism of $\Lambda$-chain complexes
\begin{eqnarray}
\psi_C \colon  (C^{-*} \otimes_{\Lambda} C)^{-*} &\xrightarrow{\cong} &
C^{-*} \otimes_{\Lambda} C
\label{chain_iso_psi_C}
\end{eqnarray}
by the composition of the inverse of $\mu_C$
with the flip isomorphism 
$\flip \colon  C \otimes_{\Lambda} C^{-*} \xrightarrow{\cong}
C^{-*} \otimes_{\Lambda} C$. 
It straight forward to check that $(C^{-*} \ox_\Lambda C,\psi_C)$ is a
$0$-dimensional symmetric Poincar\'e $\Lambda$-chain complex.
 We call
it the \emph{multiplicative hyperbolic symmetric Poincar\'e
  $\Lambda$-chain complex} associated to the finite projective
$\Lambda$-chain complex $C$ and denote it by $H_{\otimes}(C)$.

Given a finite projective $\Lambda$-chain complex $C$, define its
\emph{(unreduced) finiteness obstruction} to be
$$o(C) = \sum_{n \in \IZ} (-1)^n \cdot [C_n] \quad \in K_0(R).$$

\begin{lemma}
\label{lem:H_otimes_and_chain_complex}
Let $C$ be a finite projective $\Lambda$-chain complex. Then
the homomorphism defined in~\eqref{H_otimes_from_K_to_L}
$$H_{\otimes}^{\Lambda} \colon K_0(\Lambda) \to L^0_p(\Lambda)$$
sends $o(C)$ to the class $[H_{\otimes}(C)]$ of the symmetric $0$-dimensional
Poincar\'e $\Lambda$-chain complex $H_{\otimes}(C)$.
\end{lemma}

\begin{proof}
  Let $(C^{-*} \otimes_{\Lambda} C)_{\ge 1}$ be the $\Lambda$-chain
  complex which has in dimension $n \ge 1$ the same chain modules as
  $C^{-*} \otimes_{\Lambda} C$, whose differentials in dimensions $n
  \ge 2$ are the same as the one for $C^{-*} \otimes_{\Lambda} C$, and
  whose chain modules in dimensions $\le 0$ are trivial. Let $p \colon
  C^{-*} \otimes_{\Lambda} C \to (C^{-*} \otimes_{\Lambda} C)_{\ge 1}$
  be the obvious surjective $\Lambda$-chain map. Since the composite
  $p \circ \psi_C \circ p^{-*}$ is trivial, we can
  perform algebraic surgery in the sense of Ranicki
  (see~\cite[Section~1.5]{Ranicki(1981)}) on $p$ to obtain a new
  symmetric $0$-dimensional Poincar\'e chain complex $(D,\psi)$ such
  that $H_{\otimes}(C)$ and $(D,\psi)$ are algebraically bordant and
  hence represent the same class in $L^0_p(\Lambda)$.  Since
  $\psi_C$ is a chain isomorphism, one easily checks that
  $(D,\psi)$ is $\Lambda$-chain homotopy equivalent to the
  $0$-dimensional Poincar\'e complex whose underlying $\Lambda$-chain
  complex is concentrated in dimension zero and given there by the
  $\Lambda$-module $(C^{-*} \otimes_{\Lambda} C)_0$ and whose
  Poincar\'e $\Lambda$-chain homotopy equivalence is the inverse of
  the $\Lambda$-isomorphism $\psi_{(C^{-*} \otimes_{\Lambda} C)_0}$.
  Hence the class $[H_{\otimes}(C)] \in L^0_p(\Lambda) $ of
  $H_{\otimes}(C)$ corresponds to the non-singular symmetric
  $\Lambda$-form
$$\psi_{(C^{-*} \otimes_{\Lambda} C)_0} \colon (C^{-*}
\otimes_{\Lambda} C)_0 \xrightarrow{\cong}
\left(C^{-*} \otimes_{\Lambda} C)_0\right)^*.$$
Recall that
$$(C^{-*} \otimes_{\Lambda} C)_0 = 
  \bigoplus_{i \in \IZ} C_i^* \otimes_{\Lambda} C_i.$$
and that under this decomposition $\psi_{(C^{-*} \otimes_{\Lambda} C)_0}$
decomposes as the direct sum over $i \in \IZ$ of the inverses of
the $\Lambda$-isomorphisms (see~\eqref{psi_P})
$$(-1)^i \cdot \psi_{C_i} \colon  C_i^* \otimes_{\Lambda} C_i
\to \left(C_i^* \otimes_{\Lambda} C_i\right)^*.$$
Notice that we pick a sign $(-1)^i$ since in the definition of
the flip isomorphism for chain complexes a sign appears.
This implies in $L^0_p(\Lambda)$
\begin{eqnarray*}
[H_{\otimes}(C)]
& = &
[(C^{-*} \otimes_{\Lambda} C)_0,\psi_{(C^{-*} \otimes_{\Lambda} C)_0}]
\\
& = &
\sum_{i \in \IZ} [(C_i^* \otimes_{\Lambda} C_i,(-1)^i \cdot \psi_{C_i})]
\\
& = &
\sum_{i \in \IZ} (-1)^i \cdot [(C_i^* \otimes_{\Lambda} C_i,\psi_{C_i})]
\\
& = &
\sum_{i \in \IZ} (-1)^i \cdot H_{\otimes}^{\Lambda}([C_i])
\\
& = &
H_{\otimes}^{\Lambda}\bigg(\sum_{i \in \IZ} (-1)^i \cdot  [C_i]\bigg)
\\
& = &
H_{\otimes}^{\Lambda}(o(C)).
\end{eqnarray*}
\end{proof}


\typeout{---------------- References -----------------------------------}



\end{document}